\documentclass[a4paper,11pt]{amsart}

\usepackage{amsmath,amsxtra,amssymb,latexsym, amscd,amsthm}
\usepackage[mathscr]{eucal}
\usepackage{mathrsfs}
\usepackage{bbm}
\usepackage{enumerate}
\usepackage{pict2e}
\usepackage{tikz}
\usepackage{graphicx}
\usepackage[a4paper]{geometry}
\usepackage{color}
\usepackage{hyperref}
\numberwithin{equation}{section}

\usepackage{calc}
\usepackage{graphicx}
\usepackage{framed}

\usepackage{caption}   
\usepackage{ifthen}       

\newcommand{\Adja}{\mathcal A}  
\newcommand{\Lapl}{\mathcal L}  
\newcommand{\Schr}{\mathcal H}  
\newcommand{\Mult}{\mathcal M}  
\newcommand{\Diag}{\mathcal D}  
\newcommand{\Tran}{\mathcal T}  
\newcommand{\IdOp}{\mathcal I}  

\newcommand{\Torus}{\mathbb T}  

\newcommand{\Ftrafo}[1]{\widehat{#1}}  
\newcommand{\invFtrafo}[1]{\widecheck{#1}}  

\newcommand{\Energy}{\mathcal E}
\renewcommand{\phi}{\varphi}  

\theoremstyle{plain}
\newtheorem{thm}{Theorem}[section]

\newtheorem{prp}[thm]{Proposition}
\newtheorem{cor}[thm]{Corollary}
\newtheorem{lem}[thm]{Lemma}
\theoremstyle{definition}
\newtheorem{defa}[thm]{Definition}
\newtheorem{problem}[thm]{Problem}

\newtheorem{rem}[thm]{Remark}

\newtheorem{assumption}[thm]{Assumption}
\newtheorem{exa}[thm]{Example}
\newtheorem*{rem*}{Remark}

\newcommand{\dd}{\mathrm{d}}
\newcommand{\ee}{\mathrm{e}}
\newcommand{\ii}{\mathrm{i}}

\newcommand{\N}{\mathbb{N}}
\newcommand{\Z}{\mathbb{Z}}
\newcommand{\Q}{\mathbb{Q}}
\newcommand{\R}{\mathbb{R}}
\newcommand{\C}{\mathbb{C}}

\newcommand{\T}{\mathbb{T}}

\newcommand{\cA}{\mathcal{A}}

\newcommand{\cD}{\mathcal{D}}
\newcommand{\cH}{\mathcal{H}}

\newcommand{\cL}{\mathcal{L}}

\newcommand{\fa}{\mathfrak{a}}

\newcommand{\eps}{\varepsilon}
\renewcommand{\phi}{\varphi}  

\newcommand{\funda}{V_0} 

\DeclareMathOperator{\one}{\mathbbm{1}}

\DeclareMathOperator{\ran}{ran}
\DeclareMathOperator{\diag}{diag}

\DeclareSymbolFont{extraup}{U}{zavm}{m}{n}
\DeclareMathSymbol{\varheart}{\mathalpha}{extraup}{86}
\DeclareMathSymbol{\vardiamond}{\mathalpha}{extraup}{87}
\title{
The curious spectra and dynamics of non-locally finite crystals
}
\author{Joachim Kerner, Olaf Post, Mostafa Sabri and Matthias T\"aufer}
\address{Lehrgebiet Angewandte Stochastik, Fakult\"at Mathematik und Informatik, FernUniversit\"at in Hagen, 58084 Hagen, Germany}
\email{joachim.kerner@fernuni-hagen.de}
\address{Fachbereich 4 – Mathematik, Universit\"at Trier, 54286 Trier, Germany}
\email{olaf.post@uni-trier.de}
\address{Science Division, New York University Abu Dhabi, Saadiyat Island, Abu Dhabi, UAE.}
\email{mostafa.sabri@nyu.edu}
\address{Lehrgebiet Analysis, Fakult\"at Mathematik und Informatik, FernUniversit\"at in Hagen, 58084 Hagen, Germany}
\email{matthias.taeufer@fernuni-hagen.de}

\makeatletter
\newlength{\temp@wc@width}
\newlength{\temp@wc@height}
\newcommand{\widecheck}[1]{%
  \setlength{\temp@wc@width}{\widthof{$#1$}}%
  \setlength{\temp@wc@height}{\heightof{$#1$}}%
  #1\hspace{-\temp@wc@width}%
  \raisebox{\temp@wc@height+2pt}[\heightof{$\widehat{#1}$}]%
     {\rotatebox[origin=c]{180}{\vbox to 0pt{\hbox{$\widehat{\hphantom{#1}}$}}}}%
}
\makeatother

\begin{document}

\begin{abstract}
This paper is devoted to the investigation of the spectral theory and dynamical properties of periodic graphs which are not locally finite but carry non-negative, symmetric and summable edge weights. These graphs are shown to exhibit rather intriguing behaviour: for example, we construct a periodic graph whose Laplacian has purely singular continuous spectrum. Regarding point spectrum, and different to the locally finite case, we construct a graph with a partly flat band whose eigenvectors must have infinite support. Concerning dynamical aspects, under some assumptions we prove that motion remains ballistic along at least one layer. We also construct a graph whose Laplacian has purely absolutely continuous spectrum, exhibits ballistic transport, yet fails to satisfy a dispersive estimate. This provides a negative answer to an open question in this context. Furthermore, we include a discussion of the fractional Laplacian for which we prove a phase transition in its dynamical behaviour. Generally speaking, many questions still remain open, and we believe that the studied class of graphs can serve as a playground to better understand exotic spectra and dynamics.
\end{abstract}

\maketitle

\section{Introduction}

This paper is about non-locally finite periodic graphs. The study of periodic operators --- differential or discrete --- is an old and fascinating subject in mathematics as well as physics, see, for example,~\cite{KuchmentFloquetTheory,KuchmentPeriodicOperators} and references therein.
From the physical point of view, periodic systems are ubiquitous in nature and therefore prominently arise in fields such as condensed matter physics.
As an example, the investigation of energy bands associated to a periodic Schr\"odinger operator allows to understand conductivity properties of metals.
On the other hand, due to technological advances in the last decades, new materials with exciting properties such as graphene~\cite{KP} (whose underlying periodic structure is the two-dimensional honeycomb lattice) have gained increasing attention.
As a matter of fact, many other two-dimensional periodic systems such as the Kagome lattice are now studied with great effort due to their peculiar properties, often in connection with the existence of so-called flat bands; see~\cite{KangKagome} and references therein.

In the discrete setting, beyond periodic Schr\"odinger operators on $\Z^d$ which already offer many challenges~\cite{KuchmentPeriodicOperators,Liu,Liu2}, much ongoing activity has been devoted to understanding the spectrum and dynamics of general periodic graphs which are locally finite~\cite{KorSa,KorSa2,FLM,BdMS23,SabriY-23}. It is known that the spectrum consists  of bands of absolutely continuous spectrum, and some bands may be degenerate, thus inducing an eigenvalue. The problem of characterizing these ``flat bands'' and understanding how they behave under perturbation is still open, except in special settings such as for Archimedean tilings~\cite{TaeuferPeyerimhoff, KernerTaeuferWintermayr}, see~\cite{HN,SabriY-23} for some partial results. Away from this point spectrum, wavefunctions are completely delocalized~\cite{MS} and singular continuous spectrum is absent \cite{HN}. Concerning the dynamics, the speed of motion is known to be ballistic for any initial state associated to the continuous spectrum~\cite{BdMS23}, and the spreading is ergodic~\cite{BdMS24}. Dispersion (the fact that waves flatten out with time as they spread) is understood only in particular settings~\cite{YZ20,YZ22,AS23}.

In the present paper we go beyond the locally finite setting and consider periodic graphs where a vertex may have infinitely many neighbours. In order to maintain the physical meaning that an edge encodes a probability of hopping, we assume that the edges carry positive, locally summable weights. Our main interest is to test whether known spectral and dynamical properties of locally finite crystals carry over to this more general setting, and more specifically, whether we can exhibit phenomena which are either impossible or currently out of reach in the locally finite case. Before discussing some results, let us give some background on as to why we are interested in such questions.

In the topological study of quantum walks and Hamiltonians, in the context of Kitaev's periodic table, the class of operators which naturally arises is the one of \emph{essentially local} operators \cite{CGGSVWW,CGSVWW,CGGSVWW2,CGWW,GS,CS}. The classification is best understood up to norm-continuous deformations respecting this  notion of locality, which is formulated in terms of a commutator relation, and is much weaker than \emph{strict locality} (as in locally finite graphs, or graphs with exponentially decaying weights). In principle, it may thus happen that two operators in the same class can be connected only by breaking strict locality. By focussing on periodic operators, the question that arises is whether the spectral and dynamical properties observed in the locally finite case are really good representatives of the full set of essentially local periodic Hamiltonians. In this paper, we consider the subclass of periodic graphs with summable weights, which seems to be close to the borderline where essential locality holds, cf. Proposition~\ref{prp:local}. As we will see, already in this subclass, one can observe much richer phenomena than in the locally finite case.

Similar questions have already been studied in statistical mechanics, where particle models with infinite range interactions are quite common \cite{GallavottiMiracleSole,GallavottiBook,RuelleBook}. In particular, while the $1d$ Ising model with finite range interaction is uninteresting in the sense that it does not exhibit a phase transition, the picture is different if one allows for long-range interactions. In fact, if the associated Hamiltonian is given by $H=-\sum_{i>j} J(i-j)\mu_i\mu_j$ where $J(n)\ge 0$ describes the interaction and $\mu_i \in \{-1,1\}$ refer to spin variables, then it was shown by Ruelle \cite{RuelleCMP} that no phase transition occurs if $\sum_n nJ(n)<\infty$, and by Dyson \cite{Dyson} that a phase transition \emph{exists} if $J(n)=n^{-\alpha}$ with $1<\alpha<2$; see also \cite{WraggGehring,PhysicsReportPaper} for later developments. Quite interestingly, we will encounter parallels of this situation in our analysis of periodic operators. 

Closer to our framework is the intense and ongoing investigation of the spectra and dynamics of long-range operators \cite{Han,JL,SS,GebMol,DER,LPW,Shi,JS,CRS+,PKL+,NagGarg}, which can be regarded as weighted adjacency matrices of certain non-locally finite graphs. In those references, the focus typically is on generalizing spectral and dynamical properties from the locally finite to the long-range case. This often needs significant efforts, and does not require the operator to be periodic. Most of our paper is headed in a different direction by searching for exotic properties in the special setting of periodic operators. 


Another motivation comes from the increasing interest in interpreting random matrices as weighted graphs, for example in order to understand the convergence of empirical spectral measures. The limit of these graphs, if it exists (cf.~\cite{AS,AL}), is not necessarily locally finite, one example being the Poisson weighted infinite tree (PWIT). Of course, while a random graph is far from having a perfect $\Z^d$-periodic structure, it nevertheless seems desirable to investigate the spectrum and dynamics in this ideal, non-generic framework for a better contrast.

As we alluded to previously, our main contribution with this paper is to show that in the class of periodic graphs, the transition to the non-locally finite setting leads to a host of new and unexpected spectral and dynamical phenomena. To get a taste of the kind of results presented here, let us state the following theorem; we refer to Table~\ref{tab:summary} for a more comprehensive summary of our results.

\begin{thm}\label{thm:intro}
Consider the periodic graph $\Gamma$ over $\Z$ where each $n\in \Z$ is connected to the vertex $n+k$ with edge-weight $w(k)$ for $k\neq 0$. Let $\cA_\Gamma$ be the associated adjacency matrix.
\begin{enumerate}[\rm (1)]
\item If $w(k) = k^{-2}$ for all $k$, then $\cA_\Gamma$ has purely absolutely continuous spectrum, any initial state $\psi\neq 0$ spreads out at ballistic speed, and $\cA_\Gamma$ disperses faster than $\cA_\Z$: we have $\|\ee^{-\ii t\cA_\Gamma}\psi\|_\infty \le C t^{-1/2} \|\psi\|_1$.
\item If $w(k)=k^{-2}$ for odd $k$ and $w(k)=0$ for even $k$, then $\cA_\Gamma$ has purely absolutely continuous spectrum, any initial state $\psi\neq 0$ spreads out at ballistic speed, however $\cA_\Gamma$ violates dispersion: we have $\|\ee^{-\ii t\cA_\Gamma} \delta_n\|_\infty> c>0$ at all times. In fact, $\delta_n$ starts with full mass at $n$, then splits into two tents to the right and the left, traveling at linear speed, with peaks not decaying. The peaks are located near $n\pm k$ at time $t\approx k$, and the mass decays polynomially away from the peaks.
\item If $w(k)=k^{-2}$ for odd $k$, $w(k)=2k^{-2}$ for $k\in 2\Z\setminus 4\Z$, $w(k)=0$ for $k\in 4\Z$, then $\cA_\Gamma$ has a flat band (infinitely degenerate eigenvalue), and any eigenvector corresponding to it has infinite support. It violates dispersion, and in fact, a sizable mass of the initial state $\delta_0$ stays at the origin at all times.
\end{enumerate}
\end{thm}

\begin{figure}
  \begin{tikzpicture}[scale=0.40]
    \newcommand\n{16}
    \newcommand\m{16}
    \newcommand\mythickness{4.0}
    \pgfmathsetmacro{\denomi}{int(2*1*1)}
    \draw[line width = \mythickness/\denomi] (-\n- .5, 0) -- (\n + .5, 0);

    \foreach \j in {2,...,\n}  
    {
      \pgfmathsetmacro{\denomi}{int(\j*\j)}
      \draw[line width = \mythickness/\denomi] (0,0) arc (180:0:\j*.5 cm and {1.75 cm})
      node[label={[yshift=0cm]below:\tiny$\frac1{\denomi}$}]{};
      \draw[line width = 4./\denomi] (0,0) arc (0:180:\j*.5 cm and {1.75 cm})
      node[label={[yshift=0cm]below:\tiny$\frac1{\denomi}$}]{};
    }
    \foreach \i in {-\n,...,\n}
    {
      \ifthenelse{\i=1 \OR \i=-1}
      {
        \draw[thick, fill = white] (\i,0) circle (4pt)
        node[label={[yshift=-0.05cm]below:\tiny$1$}]{};
      }
      {
        \draw[thick, fill = white] (\i,0) circle (4pt);
      }
    }
    \draw[thick, fill = black] (0,0) circle (3pt);
  \end{tikzpicture}
  \caption{The graph of Theorem~\ref{thm:intro}(1); we only plot the edges to the nearest neighbour and the first edges from the black vertex to its
    infinitely many neighbours; the numbers indicate the weight of the
    corresponding edge.}
  \label{fig:3ex}
\end{figure}
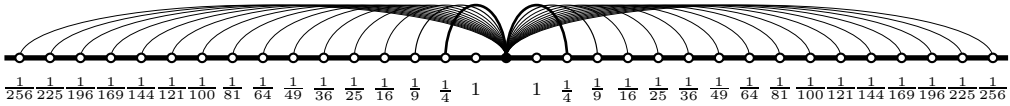

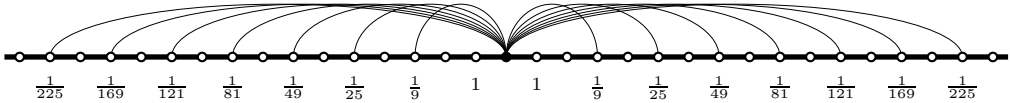
\begin{figure}
  \begin{tikzpicture}[scale=0.40]
    \newcommand\n{16}
    \newcommand\m{16}
    \newcommand\mythickness{2.0}

    \pgfmathsetmacro{\denomi}{int(1*1)}
    \draw[line width = \mythickness/\denomi] (-\n- .5, 0) -- (\n + .5, 0);

    \foreach \j in {2,...,\n}
    {
      \pgfmathsetmacro{\denomi}{int(\j*\j)}
      \ifthenelse {\isodd\j}
      {
        \draw[line width = \mythickness/\denomi] (0,0) arc (180:0:\j*.5 cm and {1.75 cm})
        node[label={[yshift=0cm]below:\tiny$\frac1{\denomi}$}]{};
        \draw[line width = \mythickness/\denomi] (0,0) arc (0:180:\j*.5 cm and {1.75 cm})
        node[label={[yshift=0cm]below:\tiny$\frac1{\denomi}$}]{};
      }
      {
      }
    }
    \foreach \i in {-\n,...,\n}
    {
      \ifthenelse{\i=1 \OR \i=-1}
      {
        \draw[thick, fill = white] (\i,0) circle (4pt)
        node[label={[yshift=0cm]below:\tiny$1$}]{};
      }
      {
        \draw[thick, fill = white] (\i,0) circle (4pt);
      }
    }
    \draw[thick, fill = black] (0,0) circle (3pt);
  \end{tikzpicture}
  \caption{The graph of Theorem~\ref{thm:intro}(2) under the same conventions.}
  \label{fig:2ex}
\end{figure}
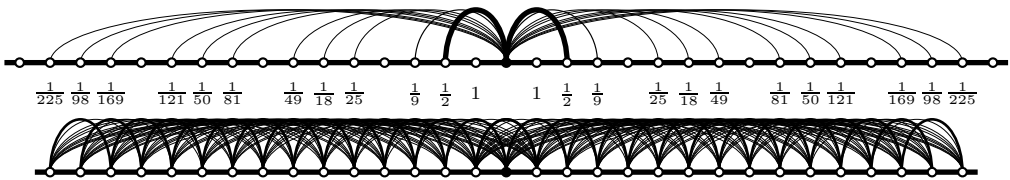
\begin{figure}
  \begin{tikzpicture}[scale=0.40]
    \newcommand\n{16}
    \newcommand\m{16}
    \newcommand\mythickness{4.0}

    \pgfmathsetmacro{\denomi}{int(2*1*1)}
    \draw[line width = \mythickness/\denomi] (-\n- .5, 0) -- (\n + .5, 0);

    \draw (2,0) [line width = \mythickness/2] (0,0) arc (180:0:2*.5 cm and {1.75 cm})
        node[label={[yshift=0.05cm]below:\tiny$\frac12$}]{};
    \draw (1,0) 
        node[label={[yshift=-0.05cm]below:\tiny$1$}]{};
    \draw (-1,0)
        node[label={[yshift=-0.05cm]below:\tiny$1$}]{};
    \draw (-2,0) [line width = \mythickness/2] (0,0) arc (180:0:-2*.5 cm and {1.75 cm})
        node[label={[yshift=0.05cm]below:\tiny$\frac12$}]{};    
    
    \foreach \j in {3,...,\n}
    {
      \ifthenelse {\isodd\j}
      {
        \pgfmathsetmacro{\denomi}{int(\j*\j)}
      }
      {
        \pgfmathsetmacro{\denomi}{int(\j*\j/2)}
      }
      
      \ifthenelse{\j=4 \OR \j=8 \OR \j=12 \OR \j=16}
      {}
      {
        \draw[line width = \mythickness/\denomi] (0,0) arc (180:0:\j*.5 cm and {1.75 cm})
        node[label={[yshift=0cm]below:\tiny$\frac1{\denomi}$}]{};
        \draw[line width = \mythickness/\denomi] (0,0) arc (0:180:\j*.5 cm and {1.75 cm})
        node[label={[yshift=0cm]below:\tiny$\frac1{\denomi}$}]{};
      }
    }
    \foreach \i in {-\n,...,\n}
    {
        \draw[thick, fill = white] (\i,0) circle (4pt);
    }
    \draw[thick, fill = black] (0,0) circle (3pt);
  \end{tikzpicture}
  \begin{tikzpicture}[scale=0.40]
    \newcommand\n{16}
    \newcommand\m{16}
    \newcommand\mythickness{4.0}

    \pgfmathsetmacro{\denomi}{int(2*1*1)}
    \draw[line width = \mythickness/\denomi] (-\n+ .5, 0) -- (\n - .5, 0);

    \foreach \k in {2,...,\n}
    {
      \foreach \j in {2,...,\k}
      {
        \ifthenelse {\isodd\j}
        {
          \pgfmathsetmacro{\denomi}{int(2*\j*\j)}
        }
        {
          \pgfmathsetmacro{\denomi}{int(\j*\j)}
        }
        \ifthenelse{\j=4 \OR \j=8 \OR \j=12 \OR \j=16}
        {}
        {
          \draw[line width = \mythickness/\denomi] (-\k+1,0)
          arc (180:0:\j*.5 cm and {1.75 cm});
          \draw[line width = \mythickness/\denomi] (\k-1,0)
          arc (0:180:\j*.5 cm and {1.75 cm});
        }
      }
    }
    \pgfmathsetmacro{\nn}{\n-1}
    \foreach \i in {-\nn,...,\nn}
    {
        \draw[thick, fill = white] (\i,0) circle (4pt);
    }
    \draw[thick, fill = black] (0,0) circle (3pt);
  \end{tikzpicture}    
  \caption{The graph of Theorem~\ref{thm:intro}(3) under the same conventions (on top). Below is an attempt to draw all edges.}
  \label{fig:1ex}
\end{figure}

Let us highlight that Theorem~\ref{thm:intro}(2) answers negatively Open Question~9 in~\cite{DMY} in the present setting of non-locally finite crystals. Theorem~\ref{thm:intro} also illustrates the rather puzzling behaviour of dispersion: starting from $\Z$, if we add edges to all vertices with $w(k)\sim k^{-2}$, we boost dispersion from $\lesssim t^{-1/3}$ to $t^{-1/2}$. If we add edges only to odd vertices, we kill dispersion, but motion stays ballistic: initial states spread out like bumps. On the other hand, the intermediate regime in (3) is very different: part of the initial state stays localized at all times.

Let us also mention that our analysis includes a discussion of the fractional Laplacian $(-\Delta)^\alpha$ for which we establish a certain phase transition: more precisely, we show that in dimension $d=1$, the motion turns from ballistic for $\alpha>\frac{1}{4}$ to super-ballistic (blowing up in finite time) for $\alpha \le \frac{1}{4}$. Interestingly, this transition does not occur in higher dimension.

\begin{table}
\begin{tabular}{|p{2.5cm}|p{5.5cm}|p{5.5cm}|}
  \hline
  Periodic graph& Locally Finite & Non-locally Finite\\
  \hline\hline
  Flat bands& Flat bands are absent if the fundamental cell is a single vertex. A spectral band is either entirely flat or a.e. non-constant. To each flat band there exists an eigenvector of compact support. & None of these properties is true, Theorem~\ref{thm:parfla}. \\
  \hline
  Top of spectrum (spec\-tral bot\-tom for Laplacian)& The top of the spectrum of a periodic Schr\"odinger operator cannot be an eigenvalue.& This remains true, Theo\-rem~\ref{thm:top}.\\
    \hline

  Regularity of eigenvalues& Eigenvalue functions are analytic a.e.& Eigenvalue functions are continuous but may be differentiable nowhere, Proposition~\ref{prp:weyl}. Edge-weight decay gives better regularity, \S~\ref{sec:eireg}.\\
    \hline

  Spectral type & Spectrum is purely absolutely continuous, with possibly a finite number of infinitely degenerate eigenvalues (flat bands). No singular continuous spectrum. & There exists a crystal whose free Laplacian has purely singular continuous spectrum, Theorem~\ref{thm:pitt}. A criterion exists for pure AC spectrum, Theorem~\ref{thm:purely_absolutely_continuous_spectrum}.\\
    \hline

  Speed of motion& Any initial state associated to the continuous spectrum spreads out at a ballistic speed.& Mean squared displacement can blow up, Proposition~\ref{prp:smalquart}. Ballistic speed remains true in some generality at least along one layer, for point masses, Corollary~\ref{cor:balayer}.\\
    \hline

  Dispersion& Proved in one dimension for $\cA_\Z+Q$, with $Q$ periodic. Also in specific periodic graphs. Expected more generally. & A graph can have purely absolutely continuous spectrum, exhibit ballistic motion, but violate dispersion (sliding tents behaviour), Theorem~\ref{thm:open}. There exists a periodic graph over $\Z$ dispersing faster than $\cA_\Z$, Proposition~\ref{prp:megadis}.\\
    \hline
Combes-Thomas estimate& Resolvent kernel has exponential off-diagonal decay.& Only polynomial decay in general, and this is sharp, Proposition~\ref{prp:ct}.\\\hline
\end{tabular}
\caption{Summary of the main results}\label{tab:summary}
\end{table}

A particularly positive aspect of our framework is that many models we construct are exactly solvable. This puts the paper among a long tradition of models \cite{AGHH,Taka} which previously helped to deepen the understanding of spectral problems.
 
Our paper is organized as follows: In \S~\ref{SectionMathBack} we formally introduce the periodic graphs and operators on which we will focus in the sequel. In \S~\ref{sec:floquet} we discuss Floquet theory and in \S~\ref{sec:loca} we address the aforementioned question of locality. Basic but important examples are constructed in \S~\ref{sec:exa}, and more details on the construction of such graphs are given in \S~\ref{sec:gen1d}. The spectral analysis starts in  \S~\ref{APartlyFlatBand}, where we discuss new phenomena related to flat bands (point spectra). In \S~\ref{Sec:regularity} we study the regularity of the Floquet eigenvalue functions. In \S~\ref{sec:sinc} we construct a periodic graph with purely singular continuous spectrum; a criterion for purely absolutely continuous spectrum appears in \S~\ref{sec:ac}. In \S~\ref{sec:speed} we show that the speed of transport remains ballistic along a layer under weak assumptions. Dispersion is discussed in \S~\ref{sec:dis}, where we finish the proof of Theorem~\ref{thm:intro} and also discuss related questions for fractional and integer powers of the standard Laplacian. Resolvent estimates are briefly discussed in \S~\ref{sec:ct}, illustrating differences to the locally-finite case. We conclude the paper in \S~\ref{sec:conc} with a discussion of several open problems.

\section{Non-Locally Finite Crystals}
\subsection{Periodic graphs and Schr\"odinger operators}\label{SectionMathBack}
In this paper we study periodic infinite graphs $\Gamma$ defined over a countable vertex set $V$.  Following~\cite[Sec.~0.1.1]{KellerLW-book}, we formally introduce (weighted) graphs (with a potential) as follows:
\begin{defa}[Graph]
	\label{def:graph}
	A \emph{graph} over a countable vertex set $V$ is a pair
        $\Gamma = (w, Q)$, consisting of a function
        $w \colon V \times V \to [0, \infty)$ such that
	\begin{itemize}
		\item
		$w(v,v') = w(v',v)$ for all $v,v' \in V$,
              \item $w(v,v) = 0$ for all $v \in V$,\footnote{If one
                  consider quotient graphs, then it is useful to allow for
                  loops at a vertex $v$, i.e.,\
                  $w(v,v)>0$. In our convention, loops are absorbed in the potential $Q$.}
		\item
		$\sum_{v' \in V} w(v,v') < \infty$ for all $v \in V$,
	\end{itemize}
	and a function $Q \colon V \to \R$.

        We call every pair $(v,v') \in V \times V$ with $w(v,v') > 0$ an \emph{edge} with \emph{edge weight} $w(v,v')$. The vertices $v$ and $v'$ are \emph{neighbours} in this case, denoted $v \sim v'$.
        
        The function $Q$ is called a \emph{potential}.
\end{defa}
For convenience, we assume the vertex set $V$ to be embedded in $\R^d$.
\begin{defa}[Crystal]
  \label{def:per-graph}
  We say that $\Gamma=(w,Q)$ is a crystal, or a \emph{$\Z^d$-periodic graph with finite fundamental domain} if the following holds:
  \begin{itemize}
  \item $\Gamma$ is connected.
  \item there is a free group action of $\Z^d$ on $V$ with \textbf{finite
    quotient}; due to the embedding in $\R^d$ we may assume that the
    group action is given by translation by linearly independent
    vectors $\fa_1,\dots,\fa_d$, i.e., the group action is then
    \begin{equation*}
      \Z^d \times V \to V, \qquad
      (k, v) \mapsto v+k_\fa,
      \quad\text{where}\quad
      k_\fa:= \sum_{i=1}^dk_i\fa_i.
    \end{equation*}
  \item $w$ and $Q$ are $\Z^d$-invariant, i.e.,
    \begin{equation*}
      w(v+k_\fa,v'+k'_\fa)=w(v,v'+k'_\fa-k_\fa)
      \qquad\text{and}\qquad
      Q(v+k_\fa)=Q(v)
      \end{equation*}
      for all $v,v' \in V$ and $k,k' \in \Z^d$.
  \item We call $d$ the \emph{dimension (of periodicity)} of the
    periodic graph $\Gamma$.
  \end{itemize}
Denoting the lattice spanned by $\fa_1,\dots,\fa_d$ by
$\Z_\fa^d = \{n_\fa:n\in\Z^d\}$, it follows that
\begin{equation}\label{eq:vertex}
  V = \funda + \Z_\fa^d \cong \funda \times \Z^d
\end{equation}
for some \textbf{finite fundamental cell} $\funda \subset V$.  Moreover, this
decomposition is unique, i.e.\ each vertex $v \in V$ can uniquely be
written as $v = v_0 + k_\fa$ for $v_0\in \funda$ and $k \in \Z^d$
which allows to bijectively identify $V$ with $\funda \times \Z^d$.
For convenience, we index the vertices in the
fundamental cell by
\begin{equation}
  \label{eq:fund.cell}
  \funda = \{ v_1, \dots, v_\nu \},
  \quad\text{where}\quad
  \nu:= |\funda|
\end{equation}

An important consequence of the invariance of the potential $Q$ and
the weight $w$ under the periodic action of the graph is that $Q$ only
takes $\nu$ possible values $Q(v_1),\dots, Q(v_{\nu})$ for
$v_i\in \funda$, which are copied over the translates of $\funda$, and
similarly, there are only $\nu$ families of weights $w(v_i,\cdot)$,
each family being infinite if the vertex $v_i$ has infinitely many
neighbours, and these weights are then periodized over $\Gamma$.
We define
\begin{equation}
\label{eq:index}
   \begin{aligned}
   I_{ij} &:= \{k\in \Z^d:v_j+k_\fa\sim v_i\},
   \\
   w_{ij}(k)
   &:=
   w(v_i, v_j + k_\fa),
   \\
   Q_i &:= Q(v_i)
 \end{aligned}
\end{equation}
for $i,j \in \{1, \dots, \nu \}$.  Note that $I_{ji}=-I_{ij}$ and $w_{ji}(k)=w_{ij}(-k)$ by the assumed symmetry and invariance of $w$.  We can think of $I_{ij}$ as the
\emph{indices} $k$ in the group
$\Z^d$ telling that vertex $v_i$ is connected with the vertex $v_j$ in
the $k$-copy of the fundamental domain.
\end{defa}

\begin{figure}[h!]
\begin{center}
\setlength{\unitlength}{1cm}
\thicklines
\begin{picture}(2,2)(-1,-2)
	 \put(-7,-1){\line(1,0){6}}
	 \put(-6,-1){\line(0,1){1}}
	 \put(-4,-1){\line(0,1){1}}
	 \put(-2,-1){\line(0,1){1}}
	 \put(-6,-2){\line(0,1){1}}
	 \put(-4,-2){\line(0,1){1}}
	 \put(-2,-2){\line(0,1){1}}
	 \put(-4,-1){\circle*{.2}}
	 \put(-4,0){\circle*{.2}}
	 \put(-6,-1){\circle*{.2}}
	 \put(-6,0){\circle*{.2}}
	 \put(-2,-1){\circle*{.2}}
 	 \put(-6,-2){\circle*{.2}}
	 \put(-4,-2){\circle*{.2}}
	 \put(-2,-2){\circle*{.2}}
	 \put(-2,0){\circle*{.2}}
 	 \put(1,0){\line(1,0){6}}
	 \put(1,-1){\line(1,0){6}}
	 \put(1,-2){\line(1,0){6}}
	 \put(2,-1){\line(0,1){1}}
	 \put(2,-2){\line(0,1){1}}
	 \put(4,-1){\line(0,1){1}}
	 \put(4,-2){\line(0,1){1}}
	 \put(6,-1){\line(0,1){1}}
	 \put(6,-2){\line(0,1){1}}
	 \put(2,0){\circle*{.2}}
	 \put(2,-1){\circle*{.2}}
	 \put(2,-2){\circle*{.2}}
	 \put(4,0){\circle*{.2}}
	 \put(4,-1){\circle*{.2}}
	 \put(4,-2){\circle*{.2}}
	 \put(6,0){\circle*{.2}}
	 \put(6,-1){\circle*{.2}}
	 \put(6,-2){\circle*{.2}}
\end{picture}
\caption{Both graphs have $d=1$ and $\nu=3$. Left: $I_{11}=I_{33}=\emptyset$, $I_{22}=\{\pm 1\}$, $I_{12}=I_{23}=\{0\}$, $I_{13}=\emptyset$. Right: $I_{ii}=\{\pm 1\}$ for each $i$, $I_{12}=I_{23}=\{0\}$, $I_{13}=\emptyset$.}\label{fig:3simp}
\end{center}
\end{figure}
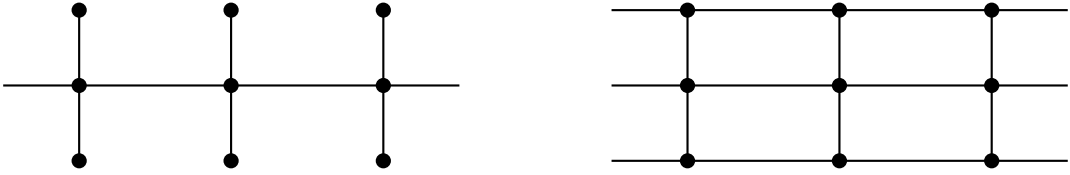

Let us now define an important property of graphs used in this article:
\begin{defa}[Non-locally finite graphs]
  A $\Z^d$-periodic graph $\Gamma$ with finite fundamental cell is \emph{non-locally finite} if
  $I_{ij}$ is \emph{infinite} for at least one pair
  $(i,j) \in \{1,\dots,\nu\}^2$, or equivalently, $w_{ij}(k) > 0$
  for \emph{infinitely} many $k \in \Z^d$. Otherwise, the graph is
  locally finite.
 \end{defa}

 Thus, a vertex can in general have infinitely many neighbours, however, let us emphasize again that throughout the article, 
 \begin{framed}
\begin{center}
The fundamental cell $V_0$ is always finite, with $\nu$ vertices. 
\end{center}
 \end{framed}
 When discussing the special case $\nu = 1$, that is graphs with exactly one vertex in the fundamental cell, we will simply write $w(k) := w_{11}(k)$. 
 \begin{exa}
The adjacency matrix of the Euclidean lattice $\Gamma=\Z^d$ has $V_0 = \{0\}$, $\nu=1$, $w(k) = 1$ if $k=\pm e_i$, and $w(k)=0$ otherwise, where $(e_i)_{i=1}^d$ is the standard basis of $\R^d$.

Each of the graphs in Theorem~\ref{thm:intro} has $V_0 = \{0\}$, $\nu=1$ and $w(k)$ as described in the theorem.
 \end{exa}
In a next step, we introduce the relevant operators, all acting on $\ell^2(V)$.
\begin{defa}[Adjacency, Laplace and Schr\"odinger operators on graphs]
  \label{def:graph.ops}
  Let $\Gamma=(w, Q)$ be a graph.
  \begin{itemize}
  \item The \emph{adjacency operator} $\Adja=\Adja_{\Gamma}$ on
    $\Gamma$ is the operator given by
    \[
      (\Adja_{\Gamma} \psi)(v)
      =
      \sum_{v' \in V}
      w(v,v')
      \psi(v').
    \]
  \item The \emph{(discrete) Laplace operator} $\Lapl=\Lapl_{\Gamma}$ on
    $\Gamma$ is the operator given by
    \[
      (\Lapl_{\Gamma} \psi)(v)
      =
      \sum_{v' \in V}
      w(v,v')
      \bigl(\psi(v)-\psi(v')).
    \]
  \item A \emph{(discrete) Schrödinger operator} on $\Gamma$ is an operator of the form
    \[
      \Schr_{\Gamma} = \Adja_\Gamma + Q
      ,
    \]
    where $Q$ is the operator multiplying with $Q$,
    i.e.\ $(Q\psi)(v)=Q(v)\psi(v)$.
  \end{itemize}
 
\end{defa}

We have
\begin{equation*}
  \Lapl_\Gamma=\Diag_\Gamma-\Adja_\Gamma,
\end{equation*}
where $\Diag_\Gamma$ is the multiplication with the (weighted) degree
function $\cD_\Gamma$ given by
\begin{equation}
  \label{eq:degree}
  \cD_\Gamma(v)=\sum_{v' \in V} w(v,v'),
\end{equation}
i.e.\  $(\Diag_\Gamma \psi)(v)=\cD_\Gamma(v)\psi(v)$.
For \emph{regular} (weighted) graphs, i.e.\  graphs such that there is
$r>0$ with $\cD_\Gamma(v)=r$ for all $v \in V$, we have
$\Lapl_\Gamma=r\IdOp-\Adja_\Gamma$ and the spectrum resp.\ all
spectral properties from one operator can be seen for the other one
and vice versa.

Note that a Schr\"odinger operator for us is ``adjacency plus
potential'' and not ``Laplacian plus potential''. We follow this convention as in \cite{DT,Liu,Liu2,LPW,AS22,BdMS23}, which changes nothing to the spectral analysis since $\cD_\Gamma$ can be regarded as a periodic potential. 

We have the following Green's
formula~\cite[Prp.1.4]{KellerLW-book}, between $\Lapl_\Gamma$ and the
so-called \emph{energy form} $\Energy_\Gamma$:
\begin{equation}
  \label{eq:dir1}
  \langle \Lapl_\Gamma \psi, \psi\rangle
  = \Energy_\Gamma(\psi,\psi)
  := \frac12\sum_{v,v' \in V} w(v,v')|\psi(u)-\psi(u')|^2.
\end{equation}
This simply follows from expanding the square modulus on the right.

Periodic operators commute with the group action, i.e.\  if $\Schr$
denotes any of the operators $\Adja_\Gamma$, $\Lapl_\Gamma$,
$\Diag_\Gamma$, $\Mult_Q$ or $\Schr_\Gamma$, and
\begin{equation*}
  \Tran_k\colon \ell^2(V)\to\ell^2(V),
  \qquad
  \Tran_k\psi(v)=\psi(v-k_\fa),
\end{equation*}
denotes the $k$-translate of $\psi$, then
\begin{equation*}
  [\Schr, \Tran_k]:=\Schr \Tran_k-\Tran_k\Schr=0
\end{equation*}
for all $k \in \Z^d$.

Using~\eqref{eq:vertex}, we identify $\ell^2(V)$ with
$\ell^2(\Z^d)^\nu$; more precisely, the map
\begin{equation*}
  \ell^2(V) \to \ell^2(\Z^d)^\nu, \qquad
  \psi \mapsto (\psi_i)_{i \in \{1,\dots,\nu\}}
  \quad\text{with}\quad
  \psi_i(k)=\psi(v_i+k_\fa).
\end{equation*}
is unitary, and we mostly use the second representation. For example, a function $\psi$ over the graphs in Figure~\ref{fig:3simp} is identified with a vector function $\psi(k)$ over $\Z$ with three (vertical) components $(\psi_1(k),\psi_2(k),\psi_3(k))$. In
particular, periodic operators can equivalently
be regarded as operating on $\ell^2(\Z^d)^\nu$, e.g. $\Schr_\Gamma$ can be written as
\begin{equation}
  (\Schr_\Gamma \psi)_i(k)
  = \sum_{j=1}^\nu \sum_{k'\in I_{ij}} w_{ij}(k')\psi_j(k+k') + Q_i\psi_i(k),
  \quad
  i \in \{1, \dots, \nu \},
	\
	k \in \Z^d.
\end{equation}
\begin{assumption}\label{AssumptionEdgeWeights}
  Using the previous notations, one immediately checks that Definition~\ref{def:graph} is equivalent to the following conditions which we assume throughout the paper:
  \begin{subequations}
    \begin{align}
      &w_{ij}(k) > 0 \text{ for all } i,j \in \{1, \dots, \nu \}
        \text{ and } k \in I_{ij}, \\
      &0\notin I_{ii} \quad \forall\,i\in\{1,\dots,\nu\},\\
      &w_{ji}(-k) =w_{ij}(k) \text{ for all } i,j \in \{1, \dots, \nu \}
        \text{ and } k \in I_{ij},\\
      &\sum_{k\in I_{ij}} w_{ij}(k) < \infty \text{ for all } i,j \in \{1, \dots, \nu \}.
    \end{align}
  \end{subequations}
\end{assumption}

\begin{rem}
    These conditions on the edge weights are quite standard in the mathematical literature~\cite{KellerLW-book}. Non-local finiteness in graphs gives access to a nice correspondence with regular Dirichlet forms, see~\cite[Lem.~1.16--1.17]{KellerLW-book}.
\end{rem}

In the following, we shall always invoke~Assumption~\ref{AssumptionEdgeWeights} unless stated otherwise.

\begin{lem}[Boundedness of periodic operators]
  Let $\Gamma$ be a $\Z^d$-periodic graph with finite fundamental cell.
  Then, the adjacency operator $\Adja_\Gamma$, the Laplacian $\Lapl_\Gamma$, and the Schr\"odinger operator $\Schr_\Gamma$ are bounded and self-adjoint
  with operator norms fulfilling
  \[
    \lVert \Adja_\Gamma \rVert
    \leq
    \max_{1 \leq i \leq \nu}
    \sum_{j = 1}^\nu
    \lVert w_{ij} \rVert,
    \quad
    \lVert \Lapl_\Gamma \rVert
    \leq
    2\max_{1 \leq i \leq \nu}
    \sum_{j = 1}^\nu
    \lVert w_{ij} \rVert,
    \quad
    \|\Schr_\Gamma\|
    \le \max_{1\le i\le \nu} \sum_{j=1}^\nu\|w_{ij}\|_1+\|Q\|_\infty \,.
  \]
  Furthermore, $\Lapl_\Gamma$ is non-negative.
\end{lem}
\begin{proof}
  Observe that
  \begin{align*}
    \|\Adja_\Gamma \psi\|^2
    = \sum_{v \in V} \Bigl|\sum_{v' \in V} w(v,v')\psi(v')\Bigr|^2
    &\le \sum_{v, v' \in V}  w(v,v')\sum_{v' \in V} w(v,v') |\psi(v')|^2\\
    &\le \sup_{v \in V} \lVert w(v,\cdot) \rVert_1 \sup_{v' \in V} \lVert w(\cdot,v') \rVert_1 \sum_{v' \in V} \lVert \psi(v') \rVert^2.
  \end{align*}
  By periodicity, the suprema can be replaced by maxima over the fundamental cell. 
  Furthermore, for any $v=v_i$, where $i \in \{1, \dots, \nu\}$, one has
  \begin{align*}
    \|w(v,\cdot)\|_1
    = \sum_{v' \in V} w(v,v')
    = \sum_{j = 1}^\nu \sum_{k \in \Z^d} w_{ij}(k)
    = \sum_{j = 1}^\nu \|w_{ij}\|_1 < \infty.
  \end{align*}
  This shows boundedness of $\cL_\Gamma$. 
  Self-adjointness follows from symmetry of the weights and the fact that $Q$ is real-valued.
  The arguments for $\Lapl_\Gamma$ and $\Schr_\Gamma$ are completely analogous.
  Non-negativity of $\Lapl_\Gamma$ follows from~\eqref{eq:dir1}. 
\end{proof}

\subsection{Fourier and Floquet transform}
\label{sec:floquet}

Let us first recall some facts about  Fourier series.  Let $\Torus= \R/\Z$.
We often identify $\Torus$ with $[0,1]$ or any other interval of
length $1$ with endpoints identified.  Then $\Torus^d$ denotes the
flat $d$-dimensional torus; we partly follow the conventions
in~\cite[Ch.~3]{grafakos:09}.

If we consider functions $f \colon \Torus^d \to \C$, we often use
$f \colon [0,1)^d \to \C$ as a
representative of the $1$-periodic function.  In formulas, it is
tacitly understood that we extend such a representative
$1$-periodically.

For $f \colon \Torus^d \to \C$ integrable, the \emph{Fourier
  coefficients} $\Ftrafo f \colon \Z^d \to \C$ is defined by 
\begin{equation}
  \label{eq:fourier-trafo}
  \Ftrafo f(k)
  = \int_{\Torus^d} f(\theta) \ee^{-2\pi \ii \theta \cdot k} \dd \theta,
  \qquad k \in \Z^d.
\end{equation}
We have that
\begin{equation}
  \label{eq:f-trafo}
  \Ftrafo \cdot \colon L^2(\Torus^d) \to \ell^2(\Z^d),
  \qquad
  f \mapsto (\Ftrafo f(k)),
\end{equation}
is unitary.  We denote the \emph{Wiener algebra} by
\begin{align*}
  A(\Torus^d)
  := \{f \in L^2(\Torus^d) \mid \Ftrafo f \in \ell^1(\Z^d) \}.
\end{align*}

The \emph{convolution} $f \ast g \colon \Torus^d \to \C$ of two
functions $f,g \colon \Torus^d \to \C$ is defined by
\begin{equation}
  \label{eq:convolution}
  (f\ast g)(\theta)
  = \int_{\Torus^d}f(\theta-\theta')g(\theta') \dd \theta'.
\end{equation}
We have
\begin{equation}
  \label{eq:conv.f-trafo}
  \Ftrafo {f \ast g} = \Ftrafo f \cdot \Ftrafo g,
\end{equation}
where $(\Ftrafo f \cdot \Ftrafo g)(k)=\Ftrafo f(k)\cdot \Ftrafo g(k)$
denotes the pointwise multiplication.

The inverse transform, also called the \emph{Fourier series} of
 $\phi \colon \Z^d \to \C$, is given by
 \begin{equation}
   \label{eq:inv-fourier-trafo}
   \invFtrafo \phi(\theta)
   = \sum_{k \in \Z^d} \phi(k) \ee^{2\pi \ii \theta \cdot k},
   \qquad \theta \in \Torus^d.
 \end{equation}
 The convergence of the series
 in~\eqref{eq:inv-fourier-trafo} is in $L^2$-sense.  If
 $\varphi \in \ell^1(\Z^d)$ then the convergence of the series is uniform by the Weierstra\ss \ M-test
 and the series is continuous. We mainly deal with this case.  If
 $f \in A(\Torus^d)$, then $f=\invFtrafo{\Ftrafo f}$ implies in particular that pointwise, we have
 \begin{equation}
   \label{eq:inv.f-trafo-2}
   f(\theta)
   = \sum_{k \in \Z^d} \Ftrafo f(k) \ee^{2\pi \ii \theta \cdot k} .
 \end{equation}

 The \emph{(discrete) convolution} $\phi \ast \gamma$ of two sequences
 $\phi,\gamma \colon \Z^d \to \C$ is defined by 
\begin{equation}
  \label{eq:disc.convolution}
  (\phi \ast \gamma)(k)
  = \sum_{k' \in \Z^d} \phi(k')\gamma(k-k').
\end{equation}
We have
\begin{equation}
  \label{eq:disc.conv.f-trafo}
  \invFtrafo {\phi \ast \gamma}
  = \invFtrafo \phi \cdot \invFtrafo \gamma,
\end{equation}
where $\cdot$ denotes the pointwise multiplication of functions
$\Torus^d \to \C$.

The next statement is easily verified: %
\begin{lem}
  \label{lem:symm}
  Let $f \colon \Torus \to \C$ be in $A(\Torus)$. Then the following
  are equivalent:%
  \begin{subequations}
    \begin{align}
      \label{eq:symm.a}
      \Ftrafo f(-k)
      &= \Ftrafo f(k)
      &\text{for all}\qquad  k \in \Z,\\
      \label{eq:symm.b}
      \Ftrafo f(k)
      &=\int_\Torus \cos(2\pi k\theta) f(\theta) \dd \theta
      & \text{for all}\qquad  k \in \Z\\
      \label{eq:symm.c}
      f(\theta)
      &= \widehat{f}(0) + 2\sum_{k=1}^\infty \Ftrafo f(k) \cos(2\pi k\theta)
      & \text{(absolute convergence),}\\
      \label{eq:symm.d}
      f(\theta)
      &=f(1-\theta)
      &\text{for almost all}\qquad
      \theta \in \Torus.
    \end{align}
  \end{subequations}
  If the latter property holds we say that $f$ is \emph{symmetric
    (with respect to $1/2$)}.
\end{lem}
\begin{proof}
  \eqref{eq:symm.a}$\implies$\eqref{eq:symm.b}: %
  We have
  \begin{align*}
    \Ftrafo f(k)
    =\frac12\bigl(\Ftrafo f(k) + \Ftrafo f(-k)\bigr)
    = \frac12\int_\Torus
    \underbrace{\bigl(\ee^{-2\pi\ii \theta k} + \ee^{2\pi\ii \theta k}\bigr)}%
    _{=2\cos (2\pi k \theta)}%
    f(\theta) \dd \theta,
  \end{align*}
  and the result follows. 

  \eqref{eq:symm.b}$\implies$\eqref{eq:symm.c}: %
  From~\eqref{eq:symm.b} we have $\widehat{f}(-k)=\widehat{f}(k)$. It follows from~\eqref{eq:inv.f-trafo-2} that
  \[
  f(\theta) = \widehat{f}(0) + \sum_{k=1}^\infty (\widehat{f}(k)+\widehat{f}(-k))\ee^{2\pi\ii\theta k} = \widehat{f}(0) + \sum_{k=1}^\infty \widehat{f}(k)(\ee^{2\pi\ii \theta k} + \ee^{-2\pi\ii \theta k})
  \]
  substituting the second index $k$ into $-k$. The claim follows.

  \eqref{eq:symm.c}$\implies$\eqref{eq:symm.d}: %
  This follows from $\cos(2\pi k(1-\theta))=\cos(2\pi k \theta)$.

  \eqref{eq:symm.d}$\implies$\eqref{eq:symm.a}: %
  Using the symmetry of $f$, we have
  \begin{align*}
    \Ftrafo f(k)
    = \int_\Torus \ee^{-2\pi\ii \theta k} f(1-\theta) \dd \theta
    = \int_\Torus \underbrace{\ee^{-2\pi\ii (1-\theta') k}}%
    _{=\ee^{-2\pi\ii \theta (-k)}}%
    f(\theta') \dd \theta
    = \Ftrafo f(-k)
  \end{align*}
  by substitution $\theta'=1-\theta$ for the second equality.
\end{proof}
We now introduce an important tool --- the Floquet transform --- which
will be used throughout the paper.
\begin{defa}[Floquet transform]
  The \emph{Floquet transform} $U:\ell^2(\Z^d)^\nu \to L^2(\Torus^d)^\nu$ is%
  
  \begin{equation}\label{eq:utile}
    (U\psi)_j(\theta) = \sum_{k\in \Z^d} \ee^{2\pi \ii\theta \cdot k} \psi_j(k)
    \quad
    \text{for $j \in \{ 1,\dots,\nu \}$}.
  \end{equation}
\end{defa}
Note that $U \psi=(\invFtrafo \psi_1,\dots,\invFtrafo \psi_\nu)$, where
$\invFtrafo \psi_i$ is the Fourier series
of~\eqref{eq:inv-fourier-trafo}.  In particular, $U\psi$ can be seen
as a vector-valued version of the Fourier series.

The following lemma states that, via the Floquet transform, periodic
operators such as adjacency, Laplace or Schr\"odinger operators are unitarily equivalent to multiplications by a matrix function on
$L^2(\Torus^d)^\nu$. This generalizes the fact that the adjacency operator on $\Gamma=\Z^d$, is equivalent
to multiplication by a cosine-type function in $L^2(\Torus^d)$.

\begin{lem}
  \label{lem:Floquet_transform}
  The Floquet transform $U$ is unitary.
  Furthermore, for any $\Z^d$-periodic graph with $\nu$ vertices in their fundamental cell,
  \begin{equation}\label{eq:dir}
    U \Adja_{\Gamma} U^{-1} = \Mult_A
    \quad
    \text{and}
    \quad
    U \Schr_{\Gamma} U^{-1} = \Mult_H
  \end{equation}
  where $\Mult_A$ and $\Mult_H$ are multiplication operators on
  $L^2(\Torus^d)^\nu$ by the matrix function
  \begin{equation}
    \label{eq:hthe}
    \begin{aligned}
      A(\theta)
      =
      (a_{ij}(\theta))_{i,j=1}^\nu,
      \qquad\qquad\qquad
      H(\theta)
      =
      (h_{ij}(\theta))_{i,j=1}^\nu,
      \\
      a_{ij}(\theta)
      =
      \sum_{k\in I_{ij}}
      w_{ij}(k) \ee^{2\pi\ii\theta\cdot k},
	\qquad
	h_{ij}(\theta)
	=
		\begin{cases}
			a_{ij}(\theta) \quad &i \neq j,\\
			a_{ii}(\theta) + Q_i \quad &i = j.
		\end{cases}
	\end{aligned}
	\end{equation}
Formula \eqref{eq:hthe} holds pointwise.
\end{lem}
\begin{proof}
  The proof is the same as~\cite[Lemma 2.1]{SabriY-23}, where the
  interchange of sums is now justified by Fubini's theorem (the
  function $\psi$ can first be taken to have compact support, then the
  result extends to all $\psi \in \ell^2(V)$).
  
Since $\sum_k w_{ij}(k)<\infty$, the series $\sum_k w_{ij}(k)\ee^{2\pi\ii\theta\cdot k}$ converges uniformly by the Weierstra\ss\ $M$-test. In particular, the formulas hold pointwise.
\end{proof}

If $\nu = 1$, the Floquet transform reduces to the Fourier transform and turns the adjacency matrix into the operator of multiplication with a function $h_{11}$ on the torus.

\begin{defa}\label{def:flofu}
If $\nu=1$, we call the function $h:=h_{11}$ in \eqref{eq:hthe} the \emph{Floquet function} corresponding to the graph.
We also say that a graph $\Gamma = (w,Q)$ is \emph{defined by a function $f \in A(\mathbb{T}^d)$} if $f=h$ is the Floquet function of $\Gamma$, that is if $\widehat f(k) = w(k)$ for all $k \in \Z^d\setminus \{0\}$ and $Q = \widehat f(0)$.
\end{defa}

Lemma~\ref{lem:Floquet_transform} implies that
\begin{equation}\label{eq:bands}
  \sigma(\Schr_{\Gamma}) = \bigcup_{j=1}^\nu \sigma_j \,,
\end{equation}
where $\sigma_j = \ran E_j(\cdot)$ are spectral bands; here,
$E_j(\theta)$ are the (real) eigenvalues of the (symmetric) matrix
$H(\theta)$, which are continuous on $\T^d$, see \S~\ref{Sec:regularity}.

In relation to the spectrum, let us mention that Shnol's theorem continues to hold in our setting if $\nu=1$ and $w(k)\lesssim |k|_2^{-\alpha}$ with $\alpha>2d$, where the result of \cite{Han} directly applies.

\subsection{Locality}\label{sec:loca}
In this subsection we restrict our attention to $d=1$, and also take $\nu=1$ for simplicity. It was observed in~\cite{CGGSVWW,CGSVWW,CGGSVWW2,CGWW,GS,CS} that in order to obtain a topological classification of one-dimensional quantum walks and spectrally-gapped insulators, one needs to relax the notion of locality as follows: a bounded operator $\Schr$ on $\ell^2(\Z)$ is \emph{essentially local} if the commutator
\begin{equation}\label{eq:comm}
T = \left[\Schr,\one\right]
\end{equation}
is a compact operator, where $\one$ is a half-line cutoff function, say $\one = \one_{\N}$. When speaking of quantum walks $W$, the condition is sometimes expressed instead as the requirement that $S=\one-W^\ast \one W$ is compact, however, as $W$ is unitary in that context, then $S$ is compact if and only if $WS=T$ is compact.

Taking $d=\nu=1$ in our case is equivalent to considering the operator
\begin{equation}\label{eq:opsimplest}
(\Schr_\Gamma f)(n) = \sum_{k\in \Z} w(k)f(n+k) + Q f(n)
\end{equation}
where $Q\in \R$ is a constant.

For $\Schr = \Adja_\Z$, the adjacency operator, it is easy to see that
$T$ is a rank-two operator. In case of exponentially decaying weights,
it is shown in~\cite[Lemma 2]{GS} that $T$ is trace class. 

The next statement suggests that our setting is close to the concept of essentially local operators.
More precisely, under Assumption~\ref{AssumptionEdgeWeights} $\Schr_\Gamma$ is always essentially local, and the commutator $T$ is not necessarily a Hilbert-Schmidt operator, suggesting that $T$ is barely compact in general and that our assumptions are not very restrictive.

\begin{prp}\label{prp:local}
Let $\Gamma$ be a $\Z$-periodic graph with a one-element fundamental cell.
\begin{enumerate}[\rm(a)]
\item
  \label{local.a}
  Under our standard assumptions $w\ge 0$, $w\in \ell^1(\mathbb{Z})$ and
  $w(-k)=w(k)$, the commutator $T$ in~\eqref{eq:comm} is compact, hence
  $\Schr_\Gamma$ in~\eqref{eq:opsimplest} is essentially local.
\item
  \label{local.b}
  If we assume moreover that $w(k)$ decreases with $|k|$ or that $w$
  is ``sup-summable'' in the sense that
  $\sum_{k\ge 1}\sup_{n\ge k} w(n)$ converges, then the operator $T$, defined in~\eqref{eq:comm} is a Hilbert-Schmidt operator.
\item
  \label{local.c}
  If we drop the extra assumption of~\eqref{local.b}, $T$ is no longer necessarily a Hilbert-Schmidt operator.
\end{enumerate}
\end{prp}
In~\eqref{local.b}, the assumption that $w$ decreases actually implies
the one of summability (in our context) since $w\downarrow$ implies
$\sum_{k\ge 1} \sup_{n\ge k} w(n) = \sum_{k\ge 1} w(k) <\infty$.

We defer the proof of Proposition~\ref{prp:local} to Appendix~\ref{app:prp_local}.

%
\section{A simple playground: Crystals with \texorpdfstring{$\nu=1$}{\texttt{nu=1}}}
\label{sec:1cell-graphs}
%

A simple and quite
fruitful idea in our framework is to create graphs by starting from
the Floquet side, by choosing appropriate functions on the torus and
ensuring that they induce connected graphs with nonnegative summable
weights, via the Floquet transform. In case $\nu=1$, i.e. when $V_0$ has only one vertex, we only need one function. Also note that a graph with $\nu=1$ is automatically regular since the degree function \eqref{eq:degree} becomes a constant.

\subsection{One-dimensional examples with one-element fundamental cell}
\label{sec:exa}
Let us consider the following simple examples which already induce
crystals with remarkable properties, as we will see later in the
article. We define

\begin{minipage}{0.48\linewidth}
  \begin{subequations}
    \begin{align}
      \label{eq:exa3}
      a(\theta)
      &=\Bigl(\theta-\frac12\Bigr)^2\\
      \label{eq:pecu}
      b(\theta)
      &= \begin{cases}
        \frac{1}{2}-\theta\ ,&\theta \in [0,\frac{1}{2}],\\
        \theta-\frac{1}{2}\ , &\theta\in [\frac{1}{2},1],
      \end{cases}\\
      \label{eq:parfla}
      c(\theta)
      &=
        \begin{cases}
          \frac14-\theta\ ,& \theta\in [0,\frac14],\\
          0\ ,& \theta\in [\frac14,\frac34]\ ,\\
          \theta-\frac34\ ,& \theta\in [\frac34,1].
        \end{cases}
    \end{align}
  \end{subequations}
\end{minipage}
\begin{minipage}{0.48\linewidth}
  \centering
  \begin{tikzpicture}[domain=0:4,scale=1]
    \draw[very thin,color=gray,dotted] (-0.1,-0.1) grid (4.1,2.1);

    \draw[->] (-0.2,0) -- (4.2,0) node[right] {\footnotesize$\theta$};
    \draw[->] (0,-0.2) -- (0,2.2) node[above] {\footnotesize$f(\theta)$};
    \draw (1,-0.1) -- (1,0.1) node[below,yshift=-1ex] {\small $\frac14$};
    \draw (2,-0.1) -- (2,0.1) node[below,yshift=-1ex] {\small $\frac12$};
    \draw (3,-0.1) -- (3,0.1) node[below,yshift=-1ex] {\small $\frac34$};
    \draw (4,-0.1) -- (4,0.1) node[below,yshift=-1ex] {\small $1$};
    \draw (-0.1,1) -- (0.1,1) node[left,xshift=-1ex] {\small $\frac14$};
    \draw (-0.1,2) -- (0.1,2) node[left,xshift=-1ex] {\small $\frac12$};

    \draw[thick]    plot (\x,{max(0,1-\x)+max(0,\x-3)}) node[right] {\footnotesize$f=c$};
    \draw[thick,dashed]   plot (\x,{abs(\x-4/2)})    node[right] {\footnotesize$f=b$};
    \draw[thick,densely dotted] plot (\x,{(2-\x)^2/4})
    node[above,left, xshift=-2.7ex, yshift=-2.4ex] {\footnotesize$f\!=\!a$};
  \end{tikzpicture}
  
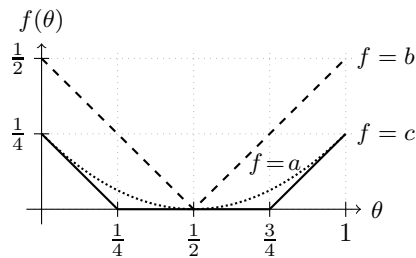
\captionof{figure}{The functions $a$ (dotted), $b$ (dashed)
    and $c$ (continuous line).}
  \label{fig:functions.a.b.c}
\end{minipage}
\begin{lem}
  \label{lem:Foufla}
  The functions $a$, $b$ and $c$ defined
  by~\eqref{eq:exa3},~\eqref{eq:pecu} and~\eqref{eq:parfla} have
  Fourier coefficients $\Ftrafo f \colon \Z \to \R$, such that
  $\Ftrafo f(k)\ge 0$, $\Ftrafo f(-k)=\Ftrafo f(k)$ for all $k$ and
  $\sum_k \Ftrafo f(k) <\infty$.

  More explicitly, $\Ftrafo a(0)=\frac{1}{12}$, $\Ftrafo b(0)=\frac14$,
  $\Ftrafo c(0)=\frac{1}{16}$ and for $k\neq 0$,
  \[
    \Ftrafo a(k) = \frac{1}{2\pi^2k^2},
    \qquad
    \Ftrafo b(k)
    = \begin{cases}
      \dfrac{1}{\pi^2k^2}& k\text{ odd}\\[1ex]
      0& k\text{ even} \end{cases}
    \qquad \text{and} \quad
     \Ftrafo c(k)
    = \begin{cases}
      \dfrac{1}{2\pi^2k^2}& k\text{ odd}\\[1ex]
      \dfrac{1-(-1)^{k/2}}{2\pi^2k^2}& k\text{ even.}
    \end{cases}
  \]
\end{lem}
\begin{proof}
  All claims follow by integrating by parts, for instance
  \begin{align*}
    \Ftrafo c(k)
    &=\int_0^1 c(\theta) \ee^{-2\pi \ii k\theta} \dd\theta
      = \int_0^{\frac14}\Bigl(\frac14-\theta\Bigr)
        \ee^{-2\pi \ii k\theta} \dd\theta
      + \int_{\frac34}^1\Bigl(\theta-\frac34\Bigr)
        \ee^{-2\pi \ii k\theta} \dd\theta\\
    &=\frac14\cdot \frac{1}{2\pi \ii k}
      +\frac{\ee^{-\pi \ii k/2}-1}{-4\pi^2 k^2}
      -\frac14\cdot \frac{1}{2\pi \ii k}
      - \frac{1-\ee^{-3\pi \ii k/2}}{-4\pi^2k^2} \\
    & = \frac{2-(-\ii)^k-\ii^k}{4\pi^2k^2}
      = \frac{2-(1+(-1)^k)\ii^k}{4\pi^2k^2},
  \end{align*}
  which yields $\Ftrafo a(k)$. The calculations for $\Ftrafo a(k)$ and
  $\Ftrafo b(k)$ are similar.
\end{proof}
For the next proposition, recall Definition~\ref{def:flofu}.

\begin{prp}\label{prp:abcvsabc}
The Floquet functions of the graphs in Theorem~\ref{thm:intro}(1), (2), (3) are given by $f_1(\theta)=2\pi^2 a(\theta)-\frac{\pi^2}{6}$, $f_2(\theta) = \pi^2b(\theta)-\frac{\pi^2}{4}$ and $f_3(\theta) = 2\pi^2c(\theta)-\frac{\pi^2}{8}$, respectively. In particular, the graphs defined by the functions $a,b,c$ exhibit exactly the same spectral and dynamical behaviour as those defined by $f_1,f_2,f_3$, respectively.
\end{prp}
\begin{proof}
The claim follow immediately from Lemma~\ref{lem:Foufla} by comparing the Fourier coefficents. For example, in Theorem~\ref{thm:intro}(3), we have $w(k)=2\pi^2 \widehat{c}(k)$ for $k\neq 0$ and $w(0)=0$. So its Floquet function is $f_3(\theta) = 2\pi^2\sum_{k\neq 0} \widehat{c}(k)\ee^{2\pi\ii k\theta} = 2\pi^2(c(\theta)-\widehat{c}(0))$.

If $\cH_\Gamma$ is the Schr\"odinger operator defined by $c$, it follows that the adjacency matrix in Theorem~\ref{thm:intro}(3) is $2\pi^2 \cH_\Gamma - \frac{\pi^2}{8} \operatorname{Id}$ . This clearly changes nothing to the spectral type or the dynamics, since $\ee^{\ii t (c\cH_{\Gamma} + c')} = \ee^{\ii t c'}\ee^{\ii t c\cH_\Gamma}$, and $|\ee^{\ii t c'}|=1$ is just a phase.

The case of the functions $a,b$ is similar. 
\end{proof}

\subsection{Functions defining one-dimensional periodic graphs}
\label{sec:gen1d}
In \S~\ref{sec:exa} we started with explicit functions
$a, b, c \colon \Torus \to \C$ which turned out to be in the Wiener Algebra $A(\T)$ and thus Floquet functions of $\Z$-periodic graphs. 
In this section we generalize this idea and prove some principles to construct Floquet functions defining periodic graphs. 

For a function $f(\theta) = \sum_{k\in \Z} \Ftrafo f(k) \ee^{2\pi\ii k\theta}$
on $[0,1]$ to correspond to the adjacency operator $\Adja_\Gamma$ of a
$\Z$-periodic graph satisfying our assumptions, it is necessary that
$\Ftrafo f(-k)=\Ftrafo f(k)$, since the $\Ftrafo f(k)$ will stand for
the edge-weight $w(k)$ which has to be symmetric.  It is also
necessary that $\Ftrafo f(k)\ge 0$ and $\sum_k \Ftrafo f(k) <\infty$.  The previous conditions are sufficient if the graph defined by $w(k)=\Ftrafo f(k)$ is connected.  To ensure this connectedness, it is sufficient to have
$\Ftrafo f(1)\neq 0$, for in this case $w(1)\neq 0$, which means that each
vertex $n$ is connected to $n\pm 1$ with edge weight $w(1)>0$.  More generally, we have the following statement. 
\begin{lem}[A condition ensuring connectedness]
  \label{lem:connected}
  Let $f \colon \Torus^d \to \C$ such that $\Ftrafo f(k) \ge 0$, $\Ftrafo f(-k) = \Ftrafo f (k)$, and $\sum_k \Ftrafo f(k) < \infty$.  
  If integer linear combinations of $\{k \in \Z^d \mid \Ftrafo f(k) \ne 0\}$ span the whole of $\Z^d$, then the crystal $\Gamma$ with $\nu=1$, defined by $f$ is connected. 
  In particular, if $d=1$ it suffices to have $\Ftrafo f(1)>0$.
\end{lem}

A preliminary result towards constructing a Floquet function is the following:

\begin{lem}\label{lem:Kah}
Suppose $f\in C(\T)$ has nonnegative Fourier coefficients. Then $f\in A(\T)$. 
\end{lem}
\begin{proof}
This result is known \cite[p.9]{Kah}, let us recall the argument. Let $\sigma_N=\frac{S_0+\dots+S_{N-1}}{N}$, where $S_j(\theta) = \sum_{k=-N}^N \widehat{f}(k)\ee^{2\pi\ii k\theta}$. Then as $f$ is continuous, Fej\'er's theorem states that $\sigma_N$ converges uniformly to $f$. In particular, $C_N=\frac{1}{N}\sum_{j=0}^{N-1} S_j(0) \to f(0)$. But $S_j(0)$ is increasing in $j$ since $\widehat{f}(k)\ge 0$. So for $m<N$, $C_N\ge \frac{1}{N}\sum_{j=0}^{m-1} S_j(0) + S_m(0)(1-\frac{m}{N})$. Taking $N\to\infty$ yields $S_m(0)\le f(0)$ for any $m$.
\end{proof}

We saw that the functions in Figure~\ref{fig:functions.a.b.c} defined crystals
satisfying our assumptions; these
functions are convex in $[0,1]$ and symmetric. This turns out to be a general fact:
\begin{lem}[Floquet functions defining a graph I]\label{lem:conv}
  Let $f \in C([0,1])$ be real-valued and symmetric (with respect
  to $1/2$).  If $f$ non-constant and convex, then it defines a crystal over $\Z$ with $\nu=1$ (i.e. one-element
  fundamental cell) and potential $Q=\int_0^1 f(\theta) \dd \theta$.
\end{lem}
\begin{proof}
Due to symmetry of $f$ we conclude that $\Ftrafo f(k)$ is given by~\eqref{eq:symm.b} and $\widehat{f}(-k)=\widehat{f}(k)$. As $f$ is convex, this implies $\widehat{f}(k)\ge 0$. This can be seen through two integrations by parts if $f$ is $C^2$. For general continuous $f$, see \cite[p.9]{Kah} or \cite[Thm. 35]{HR}. We deduce that $\sum_k \widehat{f}(k)<\infty$ by Lemma~\ref{lem:Kah}.
  
To see connectedness of $\Gamma$, we show that $\widehat{f}(1)>0$. As in \cite{HR}, write $\widehat{f}(1)=\int_0^{1/4}[f(\theta)-f(\frac{1}{2}+\theta)-f(\frac{1}{2}-\theta)+f(1-\theta)]\cos 2\pi\theta\,\dd\theta$ and use convexity to see that the integrand is nonnegative. If $\widehat{f}(1)=0$, the term in square brackets must thus be identically zero on $[0,\frac{1}{4}]$. Using symmetry, this leads to $f(\theta)=f(\frac{1}{2}\pm \theta)$. In particular, $f(0)=f(\frac{1}{2})=f(1)$. But $f(0)=\sum_k \widehat{f}(k)$ is the maximum value of $f$. If $f$ is non-constant, then $f$ must have a minimum between $f(0)$ and $f(\frac{1}{2})$, and between $f(\frac{1}{2})$ and $f(1)$. This contradicts convexity. Thus, $\widehat{f}(1)>0$.
\end{proof}

A function is in $A(\T^d)$ if it is a convolution \eqref{eq:convolution} of two $L^2$ functions, see \cite[p.10]{Kah}. The same argument shows:
\begin{lem}[Floquet functions defining a graph II]\label{lem:convo}
  Let $f \in L^2(\Torus)$ be real-valued and symmetric with respect to $1/2$.
  Then $h := f\ast f$ defines a $\Z$-periodic graph (possibly not-connected) with $\nu=1$.
  Connectedness is ensured if $\Ftrafo f(1) \neq 0$.

  If $f,g$ are two functions satisfying the assumptions of
  Lemma~\ref{lem:conv}, then $h = f\ast g$ defines a connected
  periodic graph with one vertex in its fundamental cell.
\end{lem}
\begin{proof}
  Symmetry of $f$ implies that $\Ftrafo f$ is real-valued and even, as
  we saw in the previous proof.  By~\eqref{eq:conv.f-trafo}, we have
  $\Ftrafo{h}(k) = \Ftrafo{f}(k)^2\ge 0$.  Also,
  $\sum_k \Ftrafo{h}(k) = \sum_k \Ftrafo{f}(k)^2 = \|f\|^2<\infty$. If
  $\Ftrafo{f}(1)\neq 0$, then $\Ftrafo{h}(1)\neq 0$ and the graph
  defined by $h$ is connected.

  Finally, if $f,g$ satisfy the assumptions of Lemma~\ref{lem:conv},
  then $\Ftrafo{h}(k)=\Ftrafo{f}(k)\cdot \Ftrafo{g}(k)\ge 0$ is summable by
  the Cauchy-Schwarz inequality, and one has $\Ftrafo h(1)\neq 0$ which
  ensures connectedness.
\end{proof}
\begin{rem}
  The Floquet function in~\eqref{eq:parfla} is covered
  by Lemma~\ref{lem:conv} and also by Lemma~\ref{lem:convo}. Indeed, by
  taking $f = \one_{[1/2-\eps,1/2+\eps]}$ with
  $\eps = 1/8$ we conclude $c=f\ast f$. Also note that
  such an $f$ has Fourier coefficients
  $\frac{(-1)^k \sin 2\pi k \eps}{\pi k}$, whose squares give the
  $\Ftrafo c(k)$ of $c$ if $\eps=1/8$, as expected.
\end{rem}

\begin{lem}[Floquet functions defining a graph III]\label{lem:prod}
  Let $f,g \colon \Torus \to \C$ have non-negative, symmetric and summable Fourier coefficients.
  Then the product $f \cdot g$ satisfies the same properties and hence defines a $\Z$-periodic graph with $\nu=1$. 
  It is connected if $(\Ftrafo{fg})(1)\neq 0$.
\end{lem}
\begin{proof}
  We have $\Ftrafo{f \cdot g}=\Ftrafo f \ast \Ftrafo g$
  by~\eqref{eq:disc.conv.f-trafo}.  If $\Ftrafo f(k) \ge 0$ and
  $\Ftrafo g(k) \ge 0$ then $(\Ftrafo f \ast \Ftrafo g)(k)\ge 0$ as sum
  over non-negative summands (see~\eqref{eq:disc.convolution}).  We
  also have
  \begin{align*}
    (\Ftrafo f \ast \Ftrafo g)(-k)
    = \sum_{k'} \Ftrafo f(k') \Ftrafo g(-k-k')
    &= \sum_{k'} \Ftrafo f(-k') \Ftrafo g(k+k')\\
    &= \sum_{k'} \Ftrafo f(k') \Ftrafo g(k-k')
    =(\Ftrafo f \ast \Ftrafo g)(k),
  \end{align*}
  where we used the symmetry of the coefficients in the second
  equality and switched the indices $k'\to-k'$ in the third.  Finally,
  \begin{equation*}
    \sum_k |(\Ftrafo f \ast \Ftrafo g)(k)|
    \le \sum_k \Ftrafo g(k) \sum_{k'} \Ftrafo f(k'-k)
    \le \|\Ftrafo f\|_1\|\Ftrafo g\|_1<\infty. \qedhere
  \end{equation*}
\end{proof}

In higher dimension we have no recipe to construct graphs, however, we quote the following partial result from \cite[Thm 3.3.16]{grafakos:09}.

\begin{lem}
If $f\in C^{[d/2]}(\T^d)$ and all partial derivatives of order $[\frac{d}{2}]$ are H\"older continuous, then $f\in A(\T^d)$.
\end{lem}

\subsection{The Fractional Laplacian}\label{sec:fractio}
Consider the usual discrete Laplacian on $\Z^d$, that is $-\Delta = 2d - \cA_{\Z^d}$. Given $\alpha\in (0,1)$, the \emph{fractional Laplacian} on $\Z^d$ is the operator $(-\Delta)^\alpha$, defined via functional calculus.

This operator is usually regarded as a ``nonlocal'' operator on $\Z^d$, see \cite{CRS+}. Here we show that it can be instead regarded as the standard Laplacian of a new lattice $\Gamma$ which is not locally finite but which is contained in our special framework.

\begin{prp}\label{prp:frac}
Consider the crystal $\Gamma$ over $\Z^d$ with $\nu=1$ defined by the weight function $w(k):= - (-\Delta)^\alpha(0,k)$ for $k\in \Z^d\setminus \{0\}$ and $w(0)=0$. Then $w$ satisfies Assumption~\ref{AssumptionEdgeWeights}. Moreover,
\begin{equation}\label{e:fralap}
(-\Delta)^\alpha = \mathcal{L}_\Gamma = \mathcal{D}_\Gamma-\mathcal{A}_\Gamma \,.
\end{equation}
\end{prp}
As we will see, the corresponding crystal $\Gamma$ satisfies that $w(k)>0$ for all $k\neq 0$, hence in $1d$, the graph looks exactly like Figure~\ref{fig:3ex}, but with bigger weights.
\begin{proof}
The proof relies on known properties of the fractional Laplacian. In fact, by \cite[Thm. 2.2]{GebMol}, we have $w(k)>0$ for all $k\neq 0$ and $w(k)\asymp |k|^{-d-\alpha}$ is summable. Next, $w(-k) = -(-\Delta)^\alpha(0,-k) = w(k)$ by \cite[eq. (5.2)-(5.5)]{GebMol}. This shows that Assumption~\ref{AssumptionEdgeWeights} is satisfied. To prove \eqref{e:fralap}, recall that
by the functional calculus,
\begin{equation}\label{e:frafuncal}
(-\Delta)^\alpha = \frac{-1}{|\Gamma(\frac{-\alpha}{2})|}\int_0^\infty \dd t\, t^{-1-\alpha/2}(\ee^{t\Delta}-\mathbf{1})\,.
\end{equation}
As is well-known~\cite{GebMol}, $\ee^{t\Delta}(n,m) = \prod_{j=1}^d \ee^{-2t} I_{m_j-n_j}(2t)$, where $I_k$ is the modified Bessel function.\footnote{Simply write $\ee^{t\Delta}(n,m) = \langle \mathscr{F}^{-1}\delta_n, (\mathscr{F}^{-1}\ee^{t\Delta}\mathscr{F})\mathscr{F}^{-1}\delta_m\rangle =  \int_{\T^d} \dd\theta\,\ee^{2\pi\ii (m-n)\cdot\theta} \ee^{-t\sum_{j=1}^d (2-2\cos 2\pi \theta_j)} $ where $\mathscr{F}f = \widehat{f}$ is the Fourier transform, yielding the representation.} In particular, $(-\Delta)^\alpha(n,m)=(-\Delta)^\alpha(0,n-m)$ for all $n,m$ and thus
\[
((-\Delta)^\alpha f)(n)  = (-\Delta)^\alpha(0,0)f(n) + \sum_{k\neq 0} (-\Delta)^\alpha(0,k) f(n+k)\,.
\]
Since $\sum_{k\neq 0} (-\Delta)^\alpha(0,k) f(n+k) = -(\cA_\Gamma f)(n)$, then \eqref{e:fralap} will follow if we show that $(-\Delta)^\alpha(0,0) = \cD_\Gamma = \sum_k w(k)$. Equivalently, we need to show that
\begin{equation}\label{e:sumker}
\sum_{k\in \Z^d} (-\Delta)^\alpha(0,k)=0 \,.
\end{equation}
To prove \eqref{e:sumker}, we use \eqref{e:frafuncal} to get
\[
\sum_{k\in \Z^d} (-\Delta)^\alpha(0,k) = \frac{-1}{|\Gamma(\frac{-\alpha}{2})|} \int_0^\infty \dd t \, t^{-1-\alpha/2} \Big(\sum_{k\in \Z^d} \prod_{j=1}^d \ee^{-2t} I_{k_j}(2t)-\delta_{0,k}\Big)\,.
\]
The generating function of $I_r$ has the form $\sum_{r\in \Z} a^r I_r(z) = \ee^{z(\frac{a+a^{-1}}{2})}$. Taking $z=2t$ and $a=1$, we get $\sum_r \ee^{-2t} I_r(2t)=1$. So $\sum_{k\in \Z^d} \prod_{j=1}^d \ee^{-2t} I_{k_j}(2t)=1$. This proves \eqref{e:sumker}.
\end{proof}

If $\nu=1$, the Floquet transform reduces to the Fourier transform, and yields that $U(-\Delta)^\alpha U^{-1}=M_h$, where $h(\theta)=(2d-\sum_{i=1}^d 2\cos 2\pi \theta_i)^\alpha = (\sum_{i=1}^d 4\sin^2\pi \theta_i)^\alpha$. Thus, although $(-\Delta)^\alpha$ is commonly viewed as a ``nonlocal operator'', Proposition~\ref{prp:local} implies that it is in fact \emph{essentially local} if $d=1$.

\begin{rem} Note that other powers of the Laplacian \emph{do not} necessarily define Laplacians on lattices. In fact, in $d=1$, $(-\Delta)^2=(2I-\cA_\Z)^2 = 4I - 2\cA_\Z +\cA_\Z^2$. We see that $(-\Delta)^2(0,1)= -2$,  but $(-\Delta)^2(0,2)=1$, which changes sign. For integer powers $p$, it might be more appropriate to work with powers of the adjacency matrix; this defines a locally finite lattice $\Gamma$ with $w(k) = \cA_{\Z^d}^p(0,k)$.
\end{rem}

\section{Flat bands on non-locally finite graphs }\label{APartlyFlatBand}

We start the spectral investigation by studying the point spectrum of
$\Schr_\Gamma$. We show that dropping local finiteness can lead to new
phenomena regarding flat bands.
\subsection{A partly flat band}
In this subsection we prove the following statement. It is remarkable since all described scenarios are impossible in the locally finite setting as explained in~\cite[p.2 and Lemma~2.3 and Thm. 2.4]{SabriY-23}.
\begin{thm}\label{thm:parfla}
  The graph $\Gamma$, defined by the function $c$ of Theorem~\ref{thm:intro}(3), Figure~\ref{fig:1ex}, satisfies the following:
  \begin{enumerate}[\rm(i)]
  \item
    \label{parfla.a}
  $\Adja_\Gamma$ has an infinitely degenerate eigenvalue $\lambda_0$,
  \item
    \label{parfla.b}
    There is no eigenvector of finite support corresponding to $\lambda_0$,
  \item
    \label{parfla.c}
    A spectral band of $\Adja_\Gamma$ is partly flat, i.e.\
    $E_j(\cdot)$ is constant on a set of positive measure, but not on
    the whole torus $\Torus$.
  \end{enumerate}
  Each of~\eqref{parfla.a},~\eqref{parfla.b} or~\eqref{parfla.c} is
  impossible in the locally finite setting.
\end{thm}

\begin{proof}[Proof of Theorem~\ref{thm:parfla}]
The proof is based on a few relatively simple lemmas.

\begin{lem}\label{lem:M_a_first_example}
  Let $\Mult_c$ be the multiplication operator on $L^2(\Torus)$ by the
  function $c(\cdot)$ given in~\eqref{eq:parfla}.  Then
  $\sigma(\Mult_c)=[0, \frac14]$.  Furthermore, $\Mult_c$ has an
  eigenvalue zero with infinite multiplicity.
\end{lem}
\begin{proof}
  The first statement is obvious upon noting
  $\sigma(\Mult_c) = \operatorname{Ran}(c(\cdot)) = [0,\frac14]$.  For the
  second statement, we partition $[\frac14,\frac34]$ into $n$ disjoint
  subintervals $I_j$ of equal length and consider $f_j := \one_{I_j}$.
  Then $(\Mult_c f_j)(\theta) = c(\theta) f_j(\theta) = 0$ for all
  $\theta\in\Torus$, so $\Mult_c f_j=0$ for any $j$.  This shows that
  $f_j$ is an eigenvector of $\Mult_c$ for any $j$ to the eigenvalue
  $0$.  So $0$ is an eigenvalue of multiplicity at least $n$ for
  $\Mult_c$. Since $n$ is arbitrary, we conclude that the multiplicity
  is actually infinite.
\end{proof}

As discussed in Proposition~\ref{prp:abcvsabc}, the graph in Figure~\ref{fig:1ex} has $f_3(\theta)=2\pi^2c(\theta)-\frac{\pi^2}{8}$ as Floquet function. Since $\nu=1$, this also coincides with the single eigenvalue function, $E(\theta)$. From the construction of $c(\cdot)$, we
see that this spectral band is only partly flat. This
proves~\eqref{parfla.c} in
Theorem~\ref{thm:parfla}. Part~\eqref{parfla.a} is then a consequence of
the following lemma.

\begin{lem}[Partly flat bands I]\label{PropPartlyFlat}
  The graph $\Gamma$ of Theorem~\ref{thm:intro}(3), Figure~\ref{fig:1ex}, satisfies
  $\sigma(\Adja_\Gamma) = [\frac{-\pi^2}{8},\frac{3\pi^2}{8}]$ and $\frac{-\pi^2}{8}$ is an eigenvalue of
  infinite multiplicity.
\end{lem}

\begin{proof}
The Floquet function in Theorem~\ref{thm:intro}(3) is precisely $2\pi^2c(\theta)-\frac{\pi^2}{8}$, see Proposition~\ref{prp:abcvsabc}. The statement now
 follows from Lemma~\ref{lem:M_a_first_example} and Lemma~\ref{lem:Floquet_transform}.
\end{proof}

  It remains to prove~\eqref{parfla.b}. Let us remark that in general, if the
  infinite graph operator $\Schr_\Gamma$ has an eigenvector $\psi$,
  $\Schr_\Gamma \psi = \lambda_0\psi$, with $\psi$ of finite support
  $\Lambda_i$ for each $i=1,\dots,\nu$, then the function
  $p_i(\theta) = \sum_{k\in \Lambda_i} \ee^{-2\pi\ii k\cdot\theta}
  \psi_i(k)$ is an eigenvector for its Floquet matrix $H(\cdot)$,
  since
  $H(\theta)p(\theta) = H(\theta) (U\psi)(\theta) = (U\Schr_\Gamma
  \psi)(\theta) = \lambda_0 p(\theta)$.

  Specializing to our setting $\nu=1$ with $f_3(\theta)=2\pi^2c(\theta)-\frac{\pi^2}{8}$ and
  $\lambda_0=\frac{-\pi^2}{8}$, we see that if $\cA_\Gamma$ has an eigenvector
  $\psi$ of finite support $\Lambda\subset \Z$, then
  $p(\theta)= \sum_{k\in\Lambda} \ee^{-2\pi\ii k\theta} \psi(k)$ is a
  trigonometric polynomial such that $f_3(\theta)p(\theta)=\lambda_0p(\theta)$, i.e. $c(\theta)p(\theta)=0$ for all
  $\theta$. In particular, $(\frac14-\theta)p(\theta)=0$ on
  $[0,\frac14]$ implies $p(\cdot)$ vanishes on $[0,\frac14)$, so $p$
  is the zero polynomial, $p=0$. This implies $\psi = U^{-1}p = 0$, a
  contradiction. Hence, $\Schr_\Gamma$ has no eigenvector of finite
  support.
\end{proof}

\begin{rem}
  We have shown that the graph we constructed has no eigenvectors of
  finite support, so let us explain here how the eigenvectors look
  like. If $\cH_\Gamma$ is the Schr\"odinger operator corresponding to $c$, then
  $U \cH_{\Gamma} U^{-1} = \Mult_c$, so
  $\cH_\Gamma U^{-1} f_j = U^{-1} \Mult_c f_j=0$ for all $j$, where
  $f_j$ are eigenvectors of $\Mult_c$ for $\lambda=0$ constructed in the proof of
  Lemma~\ref{lem:M_a_first_example}.  Since $\cA_\Gamma = 2\pi^2 \cH_\Gamma - \frac{\pi^2}{8}\operatorname{Id}$, we get $\cA_\Gamma U^{-1}f_j = \frac{-\pi^2}{8}U^{-1}f_j$. This shows that $U^{-1} f_j$ is
  an eigenvector of $\cA_{\Gamma}$ for each $j$. The entries of the
  vectors $U^{-1}f_j$ can be explicitly calculated by noting that
  $\one_{[\theta^-,\theta^+]} = \sum_k \Ftrafo \one(k) \ee^{2\pi \ii k\theta}$
  with
  \begin{align*}
    \Ftrafo \one(k)
    = \int_{\theta^-}^{\theta^+} \ee^{-2\pi \ii k\theta}d\theta
    = \frac{\ee^{-2\pi \ii k \theta^-}-\ee^{-2\pi \ii k \theta^+}}{2\pi \ii k},
  \end{align*}
  and $(U^{-1} \one_{[\theta^-,\theta^+]})_k = \Ftrafo \one(k)$.  In
  particular,
  $|(U^{-1}f_j)_k|^2=(\sin^2\pi k(\theta^-_j-\theta^+_j))/(\pi^2k^2)$
  has full support on $\Gamma$ if $(\theta^-_j-\theta^+_j)\notin\Q$.
\end{rem}

\begin{rem}
  We can also construct graphs with $\nu=2$ which have flat bands that
  only admit eigenvectors of infinite support. In fact, if $c(\cdot)$
  is the function~\eqref{eq:parfla} and we define
  \[
    H(\theta)
    = \begin{pmatrix}
      c^3(\theta)& c^2(\theta)\\
      c^2(\theta)&c(\theta)
    \end{pmatrix},
  \]
  then $H(\cdot)$ defines the desired graph. The fact that
  $H(\cdot)$ defines a periodic graph uses Lemma~\ref{lem:prod},
  while the required property uses that
  $\ker H(\theta) = \C((-1,c(\theta)))$.
\end{rem}

\subsection{Partly flat bands: a generalization}

The following proposition generalizes some of the results of the
previous subsection.
\begin{prp}[Partly flat bands II]
  \label{PartlyFlatGeneral}
  Let $\Gamma$ be a $\Z^d$-periodic graph with finite fundamental cell, and let $H(\cdot) \in L^\infty(\Torus^d)^{\nu \times \nu}$
  be the Floquet matrix of the Schr\"odinger operator $\Schr_\Gamma$.
  Then $\Schr_\Gamma$ has an eigenvalue if and only if some eigenvalue function $E_p(\cdot)$ is
  constant on a subset $S$ of $\Torus^d$ of positive Lebesgue
  measure. The value of the eigenvalue for $\Schr_\Gamma$ is the
  constant value of $E_p(\cdot)$, and the eigenvalue has infinite multiplicity.
\end{prp}
\begin{proof}
  $\impliedby$: Assume $E_p(\cdot)$ takes the value $\lambda_0$ on
  $S \subset \Torus^d$ where $\lvert S \rvert > 0$. First diagonalize
  $H(\theta)$ to a diagonal matrix $D(\theta)$ with diagonal elements
  $E_i(\theta)$; say $V_\theta H(\theta)V_\theta^{-1} = D(\theta)$
  with unitary $V_\theta$.  Consider the vector function
  $f\in L^2(\Torus^d)^\nu$, where $f_i(\theta)=0$ for $i\neq p$ and
  $f_p(\theta) = \one_S(\theta)$. Then
  $D(\theta) f(\theta) =\lambda_0 f(\theta)$ for all
  $\theta$. Consequently
  $H(\theta)V_\theta^{-1} f(\theta) = V_\theta^{-1} D(\theta)
  f(\theta) = \lambda_0 V_\theta^{-1} f(\theta)$, so
  $g(\theta) = V_\theta^{-1} f(\theta)$ is an eigenvector for
  $H(\theta)$ to the eigenvalue $\lambda_0$. Consequently, $U^{-1} g$
  is an eigenvector for $\Schr_\Gamma$ to the eigenvalue $\lambda_0$.

$\implies$: Conversely, if $\Schr_\Gamma$ has an eigenvector, then $\Mult_H$ has an eigenvector, $\Mult_Hf=\lambda_0f$, implying $H(\theta)f(\theta)=\lambda_0 f(\theta)$ for all $\theta$. Diagonalizing $H(\theta)$ as before, we get $D(\theta)g(\theta) = \lambda_0 g(\theta)$ for all $\theta$, where $g(\theta)=V_\theta f(\theta)$. This implies $E_j(\theta)g_j(\theta)=\lambda_0 g_j(\theta)$ for all $j$ and $\theta$. Now for fixed $\theta$, $f(\theta)\neq 0$ implies $g(\theta)\neq 0$ since $V_\theta$ is unitary. So $f\neq 0$ implies $g\neq 0$, which implies $g_j\neq 0$ for some $j$, say $g_j(\theta)\neq 0$ on a set $S$ of positive measure. Since $(E_j(\theta)-\lambda_0)g_j(\theta)=0$ on $S$, we must have $E_j(\theta)=\lambda_0$ on $S$.

It remains to verify that $\lambda_0$ must have infinite multiplicity for $\Schr_\Gamma$. To see this, if $E_p(\theta)=\lambda_0$ on $S$ with $|S|>0$, partition $S$ into $n$ subsets $S_k$, each of positive Lebesgue measure, and instead of taking $f_p = \one_S$ in the first paragraph, take $f_p^{(k)} := \one_{S_k}$. Then $g^{(k)}(\theta)=V_\theta^{-1}f^{(k)}$ now satisfies that $U^{-1}g^{(k)}$ is an eigenvector for $\Schr_\Gamma$ corresponding to $\lambda_0$, for each $k=1,\dots, n$, implying the multiplicity of $\lambda_0$ is at least $n$. Since $n$ is arbitrary, the claim follows.
\end{proof}

\subsection{The top of the spectrum of the Schrödinger operator is not flat}
\label{sec:bottom_spec}

In this section we prove, for a connected periodic graph $\Gamma$, that the top of the spectrum of the Schrödinger operator
$\Schr_\Gamma = \Adja_\Gamma + Q$ cannot be an eigenvalue and that the
upper spectral band cannot be entirely flat, cf. also
Proposition~\ref{PartlyFlatGeneral}. 
A short proof of this in the
locally finite case is given
in~\cite[Th. 2.7]{SabriY-23}. Unfortunately, it strongly uses the fact
that if $\Schr_\Gamma $ has a flat band, then it has a corresponding
eigenvector of compact support -- which is no longer true in our
context, see Theorem~\ref{thm:parfla}. Therefore, a proof will need
more effort. Let us mention that this phenomenon has recently been
observed in~\cite{FHSZ} in the very general framework of unimodular
random graphs which are locally finite.

We shall first prove that the spectral bottom of the Laplacian cannot
be an eigenvalue. We start with the following statement. 
\begin{lem}\label{lem:matt}
  Let $\Gamma$ be a connected $\Z^d$-periodic graph with finite fundamental cell.
  Then $0\in \sigma(\Lapl_\Gamma)$.
\end{lem}
A key property we use is that
$\Gamma$ has polynomial growth.
Beyond this setting, this may no longer be true. For instance, on the
$(q+1)$-regular tree $\mathcal{T}_q$, the discrete Laplacian has spectrum $[q+1-2\sqrt{q},q+1+2\sqrt{q}]$ and thus the infimum of the spetrum is $(\sqrt{q}-1)^2$.
This spectral gap is caused by issues of non-amenability, see, e.g.~\cite{Woe}.
\begin{proof}
  We know $\sigma(\Lapl_\Gamma)\subset [0,\infty)$.  So, it suffices
  to construct for any $\eps > 0$ a function $\psi \in \ell^2(V)$
  such that
  \begin{equation*}
    \left\langle \Lapl_\Gamma \psi, \psi \right\rangle
    \leq \eps
    \lVert \psi \rVert_{\ell^2(V)}^2\ .
  \end{equation*}
  Recall that $\sum_{v' \in V} w(v,v') < \infty$ and that the
  fundamental cell is finite.  Hence, for every $\eps > 0$, there
  exists $L_0 = L_0(\eps) \in \N$ such that if $v \in \funda + k_\fa$
  for $k \in \Z^d$, then
  \begin{equation}
    \label{eq:farwsmal}
    \sum_{\substack{v' \in \Gamma \colon v' \in \funda + k'_\fa,
        \\ |k - k'|_\infty \geq L_0}}
    w(v,v') < \eps/2.
  \end{equation}
  Now, let
  \begin{equation*}
    \Lambda_N
    = \bigcup_{k\in \{0,\dots,N-1\}^d}(\funda+k_\fa)
    \subset V
  \end{equation*}
  be the $N$-cube and take $\psi = \one_{\Lambda_N}$ where
  $N \gg 1$ will be determined later.  Clearly,
  $\lVert \psi \rVert_{\ell^2(V)}^2 = \lvert \funda \rvert \cdot N^d =
  \nu N^d$.

  Applying~\eqref{eq:dir1}, the terms in the sum on the right hand
  side vanish if $u,u'$ are both in $\Lambda_N$ or both outside
  $\Lambda_N$.  Thus,
  \[
    \langle \Lapl_\Gamma \psi, \psi\rangle
    = \frac{1}{2}\Big(\sum_{\substack{v\in \Lambda_N\\v'\notin\Lambda_N,\ v'\sim v}}
    w(v,v')
    + \sum_{\substack{v\notin\Lambda_N\\v'\in\Lambda_N,\ v'\sim v}}
    w(v,v')\Big).
  \]
  We expand the first sum and bound it as follows:
  \begin{multline*}
    \sum_{\substack{v\in \Lambda_{N-L_0}\\v'\notin \Lambda_N,\ v'\sim v}}
    w(v,v')
    + \sum_{\substack{v\in \Lambda_N\setminus\Lambda_{N-L_0}\\v'\notin \Lambda_N,\ v'\sim u}}
    w(v,v') \\
    \le \nu \big((N-L_0)^d\eps/2
    + (N^d-(N-L_0)^d)\sup_v\|w(v,\cdot)\|_1\big)
    \le (\eps/2 + c_wN^{-1}) \|\psi\|_{\ell^2(V)}^2,
  \end{multline*}
  where $c_w$ is a constant, depending on $w$ and on the dimension.
  The second sum is controlled analogously.  Since $N>L_0$ is
  arbitrary, the statement follows.
\end{proof}

Applying Floquet theory to $\Lapl_\Gamma$, we know the spectrum
consists of bands $\bigcup_{j=1}^\nu \sigma_j$. 
As we saw in Theorem~\ref{thm:parfla}, some bands may be
flat, but the next statement shows that the infimum of the spectrum cannot correspond to a flat part of a band.
\begin{cor}\label{cor:botla}
  Let $\Gamma$ be a connected $\Z^d$-periodic graph with finite fundamental cell.
  Then, the spectral bottom of $\Lapl_\Gamma$ is not an eigenvalue.
\end{cor}
\begin{proof}
  We know $\Lapl_\Gamma \ge 0$ and $0\in \sigma(\Lapl_\Gamma)$.  If
  $\psi$ fulfils $\Lapl_\Gamma \psi=0$, then
  \[
    0
    = \langle \Lapl_\Gamma \psi,\psi \rangle
    = \frac{1}{2}\sum_{v,v'\in V} w(v,v')|\psi(v)-\psi(v')|^2.
  \]
  Since $w(v,v')>0$ if $v \sim v'$, it follows that
  $\psi(v) = \psi(v')$ for all $v\in \Gamma$ and all $v'\sim v$. Since
  $\Gamma$ is connected, we conclude that $\psi$ is constant. As
  $\psi\in \ell^2(V)$, this implies $\psi=0$, hence $\psi$ is not an
  eigenfunction.
\end{proof}
Lemma~\ref{lem:matt} is the main technical novelty needed to address the lack of local finiteness.
The passage from $\Lapl_\Gamma$ to $\Schr_{\Gamma} = \Adja_\Gamma+Q$ on the other hand requires a rather simple generalization of a perturbation argument in~\cite{KorSa2}, in which the
locally-finite case with weights
$w(v,v') = \frac{1}{\sqrt{\cD_\Gamma(v)\cD_\Gamma(v')}}$
(see~\eqref{eq:degree}) was considered. We only state the result here and defer details to Appendix~\ref{app}.

\begin{thm}\label{thm:top}
   Let $\Gamma$ be a connected $\Z^d$-periodic graph with finite fundamental cell.
  Then, the top spectral band of $\Schr_{\Gamma} = \Adja_\Gamma+Q$ is not entirely flat, and the maximum of the spectrum is not an eigenvalue.
\end{thm}
In the locally finite case, having an entirely flat band is equivalent
to having a partially flat one, which is equivalent to having an
eigenvalue; so, the ``moreover'' part is automatic. In our context, it
can actually happen that the top band is partially flat (see
Theorem~\ref{thm:parfla}), creating an eigenvalue. However, the
theorem says that even in this scenario, this eigenvalue cannot be at
the top of the spectrum.

%
\section{Regularity of eigenvalues}\label{Sec:regularity}
%

In this section we study the regularity of the eigenvalue functions $E_j=E_j(\cdot)$, $j=1,...,\nu$, of the Floquet matrix, cf. \S~\ref{sec:floquet}. As we will see in the following, graphs with barely $\ell^1$ weights can have rather ill-behaved eigenvalue functions, while those with fast decay will have regular eigenvalues.

\subsection{Lack of regularity: an example}
\label{SectionWeierstrass}
In the locally finite case, the eigenvalue functions $E_j(\cdot)$ are
analytic almost everywhere~\cite{SabriY-23}. The situation can be
drastically different in the non-locally finite case, even when
$\nu=1$.  
We prove the following:
\begin{prp}
  \label{prp:weyl}
  Let $\Gamma$ be a $\Z$-periodic graph with one-element fundamental cell given in Figure~\ref{fig:weierstrass}. More precisely, let the weights be given by
  \begin{equation}
    \label{eq:eignon}
    w(k)
    = \begin{cases}
      \frac{1}{4k}& \text{if } k=2^{n-1} \text{ for some } n\ge 1\ ,\\
      0&\text{otherwise}\ ,
    \end{cases}
  \end{equation}
  for $k \geq 0$, and $w(-k):=w(k)$.  
  Then the Floquet eigenvalue function of
  $\Adja_\Gamma$ is continuous but not differentiable at any point
  $\theta$.
\end{prp}
\begin{figure}
  \begin{tikzpicture}[scale=0.40]
    \newcommand{\logn}{4}
    \newcommand{\n}{16} 
    \newcommand\mythickness{8.0}

    \pgfmathsetmacro{\denomi}{int(4)}
    \draw[line width = \mythickness/\denomi] (-\n- .5, 0) -- (\n + .5, 0);
    \draw[] (-1,0)--(-1,0) node[label={[yshift=0cm]below:\tiny$\frac1{\denomi}$}]{};
    \draw[] (1,0)--(1,0) node[label={[yshift=0cm]below:\tiny$\frac1{\denomi}$}]{};
    \foreach \j in {1,...,\logn}
    {
      \pgfmathsetmacro{\denomi}{int(4*2^\j)}
      \draw[line width = \mythickness/\denomi] (0,0)
      arc (180:0:{2^\j*.5 cm} and {1.75 cm})
      node[label={[yshift=0cm]below:\tiny$\frac1{\denomi}$}]{};
      \draw[line width = \mythickness/\denomi]
      (0,0) arc (0:180:{2^\j*.5 cm} and {1.75 cm})
      node[label={[yshift=0cm]below:\tiny$\frac1{\denomi}$}]{};
    }
    \foreach \i in {-\n,...,\n}
    {
        \draw[thick, fill = white] (\i,0) circle (4pt);
    }
    \draw[thick, fill = black] (0,0) circle (3pt);
  \end{tikzpicture}
  \caption{The graph from the Weierstra\ss\ function having relatively sparse
    neighbours with rather large weight.}
  \label{fig:weierstrass}
\end{figure}
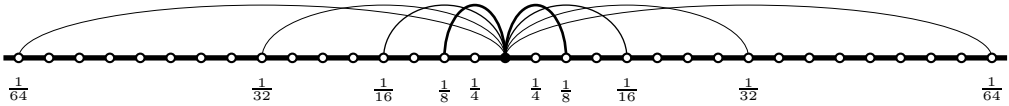
\begin{proof}
  Consider the \emph{Weierstra\ss \ function}
  \begin{equation}\label{eq:wei}
    W(\theta)
    = \sum_{n=0}^\infty \frac{\cos\big(2^n\pi \theta\big)}{2^n}
    = \cos\pi\theta + 2\sum_{n=1}^\infty \frac{1}{2^{n+1}} \cos\big(2\pi(2^{n-1}\theta)\big).
  \end{equation}
  It was proved by Hardy in a more general context~\cite{Har} that
  $W(\cdot)$ is continuous but not differentiable anywhere. The
  same is consequently true for the function%
  \[
    h(\theta) = W(\theta)-\cos\pi\theta = \sum_{k\in\Z} w(k)
    \cos(2\pi k \theta),
  \]
  where $w$ is given by~\eqref{eq:eignon}. Since $w(k)\ge 0$,
  $w(-k)=w(k)$ and $\sum_k w(k) = 2\sum_{n\ge 1} 2^{-n-1}=2$, the
  function $h$ is an admissible Floquet function and defines the
  required graph.
\end{proof}
Observe that in the example above, while most of the weights $w(k)$
are zero, the non-zero spikes decay quite slowly, at a rate of
order $1/k$.

\subsection{Regularity of the Floquet matrix elements}
    \label{subsec:regularity}
As a first step towards guaranteeing regularity for the eigenvalue
functions, we start by discussing decay conditions on the weights
which ensure some regularity for the Floquet matrix elements.
\begin{prp}\label{PropContinuity}
Under Assumption~\ref{AssumptionEdgeWeights} the entries of the Floquet matrix $h_{ij}:\Torus^d \rightarrow \mathbb{C}$ are uniformly continuous.
\end{prp}
\begin{proof} 
Let $(\theta_n)$ be a sequence such that $\theta_n\to \theta \in \Torus^d$. Then $\ee^{2\pi \ii k\cdot \theta_n} \to \ee^{2\pi \ii k\cdot \theta}$ and hence $h_{ij}(\theta_n)\to h_{ij}(\theta)$ by dominated convergence since $\sum_{k \in I_{ij}} w_{ij}(k)<\infty$, proving continuity. Uniform continuity then follows since each $h_{ij}$ is bounded on the compact torus, with $\|h_{ij}\|_\infty \le \|w_{ij}\|_1 +|Q_i|$.
\end{proof}
A faster decay of the weights leads to a higher regularity as illustrated by the following statement. 
\begin{prp}\label{prp:matrilip}
Under Assumption~\ref{AssumptionEdgeWeights}, if the weights satisfy $\sum_{k \in I_{ij}} w_{ij}(k)|k|_2<\infty$ then $h_{ij}$ is Lipschitz continuous. This condition is not necessary.
\end{prp}
\begin{proof}
Simply notice that, for some $k \in \mathbb{Z}^d$, $|\ee^{2\pi\ii k\cdot \theta}-\ee^{2\pi\ii k\cdot \phi}| = |\ee^{2\pi\ii k\cdot (\theta-\phi)}-1|\le \sum_{j=1}^d |\ee^{2\pi\ii k_j(\theta_j-\phi_j)}-1|\le \sum_{j=1}^d 2\pi |k_j(\theta_j-\phi_j)|\le 2\pi |k|_2|\theta-\phi|_2$. Hence, $|h_{ij}(\theta)-h_{ij}(\phi)|\le \sum_{k \in I_{ij}} w_{ij}(k) |\ee^{2\pi\ii k\cdot \theta}-\ee^{2\pi\ii k\cdot \phi}| \le C_{ij} |\theta-\phi|_2$ where $C_{ij}:=2\pi\sum_{k \in I_{ij}} w_{ij}(k)|k|_2$. Finally, the assumed condition is indeed not necessary as the examples of Figure~\ref{fig:functions.a.b.c} show.
\end{proof}
In the following, we set $\langle k\rangle:=\sqrt{1+|k|_2^2}=\sqrt{1+k_1^2+\dots+k_d^2}$, $k \in \mathbb{C}^d$.
%
%
\begin{prp}\label{prp:cr}
Under Assumption~\ref{AssumptionEdgeWeights}, if the weights $w_{ij}$ satisfy 
    \[
    \sum_{k\in I_{ij}} w_{ij}(k)^2 \langle k\rangle^{2s}<\infty
    \]
    with $s>r+\frac{d}{2}$, then $h_{ij}\in C^r(\Torus^d)$. 
This holds in particular if $w_{ij}(k)\lesssim \langle k\rangle^{-\alpha}$ with $\alpha>r+d$.
\end{prp}
\begin{proof}
The assumption implies that $\langle k\rangle^r w_{ij}(k)\in \ell^1(\Z^d)$, since $\sum_{k \in \mathbb{Z}^d}\langle k\rangle^{-\beta}<\infty$ for $\beta>d$. Hence, one can deduce the statement by differentiating term-wise, using dominated convergence as in Proposition~\ref{PropContinuity} and induction. Alternatively, the statement also follows from the Sobolev embedding theorem $H^s(\Torus^d)\hookrightarrow C^r(\Torus^d)$ if $s>r+\frac{d}{2}$, see~\cite[Thm. 5.6]{EinWa}.
%
The particular case follows by taking $s=\frac{\alpha+r}{2}$.
\end{proof}
As an immediate corollary, we obtain the following result.
\begin{cor}
    Under Assumption~\ref{AssumptionEdgeWeights}, if the weights $w_{ij}$ satisfy 
    \[
    \sup_{k \in I_{ij}} w_{ij}(k)\langle k\rangle^M <\infty
    \quad
    \text{for all $M \in \N$,}
    \]
    then $h_{ij}\in C^\infty(\Torus^d)$.
\end{cor}
An even faster decay leads to analytic Floquet matrix elements.
\begin{prp}\label{PropAnalytic}
      Under Assumption~\ref{AssumptionEdgeWeights}, if there are $M_{ij},\eps_{ij}>0$ such that the weights $w_{ij}$ additionally satisfy
\begin{equation}\label{eq:wexpdec}
w_{ij}(k)\le M_{ij}\ee^{-\eps_{ij}|k|}
\end{equation}
for all $k\in I_{ij}$, then $h_{ij}$ is analytic, and vice-versa.
\end{prp}
\begin{proof}
This is the Paley-Wiener theorem for Fourier series, see, e.g.~\cite[Ch.~3.12.D]{Ar} and~\cite[Lemma 5.6]{BT}.
\end{proof}
Note that the intermediate decay $w_{ij}(k) \le c_\eps \ee^{-\eps |k|^{1/s}}$ leads to the so-called Gevrey class $G^s(\Torus^d)$. The Paley-Wiener theorem states that $G^1(\Torus^d) = C^\omega(\Torus^d)$, $C^\omega(\Torus^d)$ being the class of analytic functions.

Note that $C^k(\T^d)$ and $C^\omega(\T^d)$ are understood as periodic functions of class $C^k$ and $C^\omega$. In particular, the function $a(\theta)=(\theta-\frac{1}{2})^2$ on $[0,1]$ is not analytic when considered on the torus (it is not differentiable at $0$), which explains why $\widehat{a}(k)\sim k^{-2}$, as proved in Lemma~\ref{lem:Foufla}, does not decay exponentially.

\subsection{Regularity of eigenvalue surfaces}\label{sec:eireg}
In this section we show that the regularity of the Floquet matrix entries $h_{ij}$ implies some regularity on the eigenvalue functions $E_j(\cdot)$.

In the case $\nu=1$, the Floquet matrix $H(\cdot)$ is a scalar function and we only have one weight function $w(k)$. As seen in the previous section, its decay (polynomial or exponential) entails a corresponding regularity ($C^r$ or analytic) of $E(\cdot)=H(\cdot)$. We shall discuss the case of more general $\nu$ in the following.
\begin{lem}\label{lem:ctseigen}
Under Assumptions~\ref{AssumptionEdgeWeights} there exist $\nu$ continuous functions $E_1(\cdot)\le \dots\le E_\nu(\cdot)$ on $\Torus^d$ consisting, for each $\theta \in \Torus^d$, of the eigenvalues of $H(\theta)$.
\end{lem}
\begin{proof}
Recall that, by Proposition~\ref{PropContinuity}, all $h_{ij}$ are continuous. The continuity of the eigenvalue functions now follows from Kato's classic argument~\cite[p. 106--109]{Kato}.
\end{proof}

\begin{lem}\label{lem:eigenlip}
  Assume each $h_{ij}$ is Lipschitz-continuous (e.g. assume
  $\sum_k w_{ij}(k)|k|_2<\infty$ for all $i,j=1,...,\nu$). Then, the
  following holds:
  \begin{enumerate}[\rm(a)]
  \item
  \label{eigenlip.a}
    The eigenvalue functions of Lemma~\ref{lem:ctseigen} are
    Lipschitz-continuous on $\Torus^d$.
  \item
    \label{eigenlip.b}
    If $\lambda_i(\theta)$ is a simple eigenvalue of $H(\theta)$, its
    eigenprojection is Lipschitz continuous near $\theta$.
  \item
    \label{eigenlip.c}
    If $\lambda_1(\theta),\dots,\lambda_{r(\theta)}(\theta)$ are the
    distinct eigenvalues of $H(\theta)$ (which may have higher
    multiplicity), fix a small contour $\gamma_j$ around $\lambda_j$
    enclosing no other eigenvalue of $H(\theta)$, then
    $\phi\mapsto P_{\gamma_j}(\phi)$, the total projection for the
    eigenvalues of $H(\phi)$ lying inside $\gamma_j$, is Lipschitz
    continuous near $\theta$.
  \end{enumerate}
\end{lem}
\begin{rem}
  Lemma~\ref{lem:eigenlip} is in general the best statement one can
  have by abstract perturbation theory. For example, suppose
  \[
    S(t)
    = \begin{cases}
      \begin{pmatrix} 1+t&0\\0&1-t\end{pmatrix} &\text{for } t\le 0,\\
      \begin{pmatrix} 1&t\\t&1\end{pmatrix} &\text{for } t> 0.
    \end{cases}
  \]
  Then\footnote{We came across this example by A.~Quas in an old
    \texttt{mathoverflow} post, entry \texttt{116123}.}  $S(t)$ is
  Lipschitz, but the eigenvectors are
  $u(t)=\begin{pmatrix} 1\\0\end{pmatrix}$ and
  $v(t)=\begin{pmatrix} 0\\1\end{pmatrix}$ for $t\le 0$ and
  $\begin{pmatrix} 1\\1\end{pmatrix}$ and
  $\begin{pmatrix} 1\\-1\end{pmatrix}$ for $t>0$. Clearly
  $P_{1+t}(t) = \frac{\langle \cdot, u(t)\rangle}{\|u(t)\|^2} u(t)$ is
  not continuous at $t=0$. However, claim~\eqref{eigenlip.c} does hold
  everywhere here.

  One may have the intuition that\eqref{eigenlip.b} still holds almost
  everywhere. This is true in the analytic framework because the
  eigenvalue functions meet on a trivial set (either zero or full
  measure), but there is no reason for this to hold when
  $t\mapsto S(t)$ is only smooth.
\end{rem}

\begin{proof}[Proof of Lemma~\ref{lem:eigenlip}]
  Each $h_{ij}$ is Lipschitz-continuous
  (cf. Proposition~\ref{prp:matrilip}), with Lipschitz constant
  $C_{ij} > 0$; let $C = \max_{i,j\le \nu} C_{ij}$. Then for the
  matrix norms, we obtain the estimates
  $\|H(\theta)-H(\phi)\|_1 = \max_j \sum_i
  |h_{ij}(\theta)-h_{ij}(\phi)|\le C \nu |\theta-\phi|_2$ and
  similarly $\|H(\theta)-H(\phi)\|_\infty \le C \nu
  |\theta-\phi|_2$. Hence,
  $\|H(\theta)-H(\phi)\|\le C\nu |\theta-\phi|_2$ for the operator
  norm.

  Claim~\eqref{eigenlip.a} now follows immediately from the Weyl
  inequalities: $|E_k(A+B)-E_k(A)|\le \|B\|$, with $A=H(\phi)$ and
  $B=H(\theta)-H(\phi)$ (see also~\cite{KMR} for a stronger
  statement).

  The proof of claims~\eqref{eigenlip.b} and~\eqref{eigenlip.c} follow
  again from Kato's arguments, let us give more details for the
  reader's convenience:

  Let $z\in \rho(H(\theta))$ be in the resolvent set of $H(\theta)$.
  Writing
  \begin{gather*}
    H(\phi)-z
    = (1+(H(\phi)-H(\theta))G_\theta(z))(H(\theta)-z)
      \quad\text{with}\quad
      G_\theta(z) = (H(\theta)-z)^{-1},\\
    \intertext{we see that if}
    \|H(\phi)-H(\theta)\|
    <\|G_\theta(z)\|^{-1},
      \qquad\text{then}\qquad
      z\in \rho(H(\phi)).
  \end{gather*}
  This is guaranteed by taking $\theta$ close enough to $\phi$. Under
  such assumption, the resolvent equation implies that
  $\|G_\theta(z)-G_\phi(z)\| =
  \|G_\theta(z)(H(\theta)-H(\phi))G_\phi(z)\|\le C_z |\theta-\phi|_2$,
  where $C_z=c_d \|G_\theta(z)\|\|G_\phi(z)\|$ for some $c_d > 0$.

  Now let $\lambda$ be an eigenvalue of $H(\theta)$ of multiplicity
  $m$ and let $\gamma$ be a closed curve enclosing $\lambda$ and no
  other eigenvalue. Then
  $ \frac{-1}{2\pi \ii}\int_\gamma G_\theta(z)\,\dd
  z=P_\gamma(\theta)$ is the eigenprojection onto the eigenspace of
  $H(\theta)$ for the eigenvalue $\lambda$. By choosing
  $|\theta-\phi|_2<\delta$ so small that
  $\|H(\theta)-H(\phi)\|<\min_{z\in \gamma}\|G_\theta(z)\|^{-1}$, we
  deduce that $z\in \rho(H(\phi))$ for all $z\in \gamma$ and
  $\frac{-1}{2\pi\ii}\int_\gamma G_\phi(z)\,\dd z = P_\gamma(\phi)$ is
  the eigenprojection of $H(\phi)$ for the eigenvalues lying in
  $\gamma$ (i.e.\ upon perturbation, $\lambda$ generates $m$
  eigenvalues in $\gamma$, which may or may not be distinct).

  Thus,
  $\|P_\gamma(\theta)-P_\gamma(\phi)\|\le C_\gamma |\theta-\phi|_2$,
  for $C_\gamma = \frac{1}{2\pi}\int_\gamma C_z\,\dd z$. This
  shows~\eqref{eigenlip.c}. When $m=1$, $P_\gamma=P_\lambda$ and we
  get~\eqref{eigenlip.b}.
\end{proof}
As a final result in this section we now formulate a statement which
guarantees analyticity of the eigenvalue functions $E_j=E_j(\cdot)$
almost everywhere.
\begin{prp}[Analyticity of eigenvalue functions]
  Let $\Gamma$ be a $\Z^d$-periodic graph with finite fundamental cell.
  Assume that the weights $w_{ij}$ decay exponentially as in~\eqref{eq:wexpdec}. 
  Then there exists a
  real analytic variety $X\subset \Torus^d$ of dimension $\le d-1$ (so
  $X$ is a closed nullset), such that each eigenvalue function $E_j(\cdot)$ of $\Schr_\Gamma$ is analytic on
  $\Torus^d\setminus X$. 
  Each eigenvalue has constant multiplicity on
  each of the finitely many connected components of
  $\Torus^d\setminus X$.
\end{prp}
\begin{proof}
The assumptions imply that each matrix entry $h_{ij}$ is analytic, so the arguments of Wilcox~\cite{Wil} in the (more technical) framework of periodic Schr\"odinger operators in $\R^3$ apply directly to the present context. See~\cite[Lemma 2.2]{SabriY-23} for more details.
\end{proof}
\begin{rem}
    The $C^\infty$ setting is in general more subtle~\cite{AKML}, i.e.\  to transmit smoothness of the matrix into smoothness of the eigenvalues, even for $d=1$. Because of this, the extra regularity in Proposition~\ref{prp:cr} compared to Proposition~\ref{prp:matrilip} seems only useful if $\nu=1$.
\end{rem}

\section{Spectral types}

\subsection{Singular continuous spectrum}\label{sec:sinc}
It is known that for locally finite periodic graphs there is no singular continuous spectrum; see, e.g.,~\cite[Prp. 4.5]{HN}. Remarkably, this is no longer true in our context, even for $d=\nu=1$.

\begin{thm}[A graph with purely singular continuous
  spectrum]\label{thm:pitt}
  Let $\Gamma$ be the $\Z$-periodic graph $\Gamma$ with one-element fundamental cell and
  weights given by
  \[
    w(k)
    = \begin{cases}
      \frac{1}{4k\sqrt{\log_2(2k)}} & \text{if } k=2^{n-1}
      \text{ for some } n\ge 1\ ,\\
      0&\text{otherwise}
    \end{cases}
  \]
  for $k \geq 0$, and $w(-k):=w(k)$. 
  Then $\Adja_\Gamma$ and $\Lapl_\Gamma$ have purely singularly continuous spectrum.
\end{thm}
This graph looks exactly like Figure~\ref{fig:weierstrass}, except that the weights decay slightly faster.

While Schr\"odinger operators with singular continuous spectra are quite abundant, constructions of free Laplacians with purely singular continuous spectra are rather limited~\cite{Si96,Bre07,BreF09}. In contrast to the graphs in these papers, our example is a ``dense'' graph, each vertex having infinitely many neighbours, and can in fact be made as dense as possible using Remark~\ref{rem:adw}. As this graph is regular, the theorem also holds for the Laplacian.

Quoting Simon~\cite{Si96}, our proof relies on a ``cheap trick'', which in our case combines Floquet theory with interesting results of Pitt et al. regarding the singular spectra of multiplication operators. Though the proof is a simple application of~\cite{AHP}, the conclusion is rather surprising for periodic graphs. Interestingly, that line of research originated while investigating Gaussian processes~\cite{Pi82}.
\begin{proof}
It suffices to prove the statement for $\Adja_\Gamma$ since $\Lapl_\Gamma$ only differs by a multiple of the identity and a minus sign.
Consider the following variant of the Weierstra\ss\ function,
\[
H(\theta): = \sum_{n=1}^\infty \frac{\cos(2^n\pi\theta)}{2^n\sqrt{n}}\,\ .
\]
It is shown in~\cite[p.~48]{Zyg} that $H$ belongs to the small Zygmund class $\lambda_\ast$ but is non-differentiable a.e. (in contrast to the Weierstrass function~\eqref{eq:wei}, it is still differentiable on an uncountable set of points). It follows from~\cite[Thm. 1]{AHP} that $\Mult_H$, the multiplication operator by $H(\cdot)$, has purely singular continuous spectrum. Since
\[
H(\theta) =2\sum_{n=1}^\infty\frac{1}{2^{n+1}\sqrt{n}}\cos(2\pi(2^{n-1}\theta))= 2\sum_{k=1}^\infty w(k)\cos(2\pi k\theta)
\]
then $H$ is the Floquet function of the graph in the statement and the claim follows.
\end{proof}
\begin{rem}\label{rem:adw}
Quite remarkably, this result remains true under $C^1$ perturbations, see~\cite[Thm. 1]{AHP}. In our context, this means that we can add further edges with fast enough decaying weights $\tilde{w}$, say $\tilde{w}(k)\lesssim |k|^{-2-\eps}$ for some $\eps > 0$, such that the resulting adjacency operator still has purely singular continuous spectrum.
\end{rem}
\subsection{A sufficient criterion for purely absolutely continuous spectrum}\label{sec:ac}
In the previous \S~\ref{sec:sinc} we saw that a periodic graph $\Gamma$ can have (purely) singular continuous spectrum. Also, in Proposition~\ref{PartlyFlatGeneral}, we saw that graphs can have flat bands. In this section, we explain how to avoid these scenarios. Let us first recall the Rademacher theorem~\cite[Thm 3.2]{EG} which we will use without reference here and in the next section: if $F:\R^n\to\R^m$ is locally Lipschitz continuous, then it is differentiable Lebesgue a.e.

We now give the following criterion ensuring a fully \emph{delocalized} regime.

\begin{thm}[Purely absolutely continuous spectrum]
    \label{thm:purely_absolutely_continuous_spectrum}
    Let $\Gamma$ be a crystal such that each eigenvalue function is Lipschitz continuous over $\Torus^d$. Suppose that for each $j$, we have $\nabla E_j\neq 0$ a.e. Then $\sigma(\Schr_{\Gamma})$ is purely absolutely continuous.

As a partial converse, if under the same assumptions we have $\nabla E_j = 0$ for some $j$ on an open subset $\emptyset\neq S \subset \Torus^d$, then $\Schr_{\Gamma}$ has an infinitely degenerate eigenvalue.
\end{thm}

Note that Lipschitz continuity of the eigenvalue functions can for instance be ensured by assuming $\sum_{k \in I_{ij}} w_{ij}(k)|k|_2<\infty$, cf.~Section~\ref{subsec:regularity}. 

The assumptions of the theorem can be relaxed, see Remark~\ref{rem:Luzin}.
\begin{proof}
Recall that $\sigma(\Schr_{\Gamma}) = \bigcup_{j=1}^\nu \sigma_j$, with $\sigma_j=\ran E_j(\cdot)$. We are thus reduced to studying the spectrum of each multiplication operator $\Mult_{E_j}$ on $\Torus^d$. The first part of the theorem then follows from the fact that if $\nabla f\neq 0$ a.e., then the associated multiplication operator $\Mult_f$ has purely absolutely continuous spectrum. Let us give the details for the reader's convenience:

Observe that, by functional calculus, $g(\Mult_{E_j}) = \Mult_{g\circ E_j}$ whenever $g\circ E_j$ is defined. Hence, the spectral measure $\mu_\psi$ of $\Mult_{E_j}$ in the state $\psi$ is given by
\[
\mu_\psi(B) = \langle \psi,\one_B(\Mult_{E_j})\psi\rangle = \int_{\Torus^d} (\one_B\circ E_j)(\theta)|\psi(\theta)|^2\,\dd\theta =\int_{E_j^{-1}(B)}|\psi(\theta)|^2\,\dd \theta\,, 
\]
which is also known as the ``occupation measure'' in the context of stochastic processes. Now recall the \emph{coarea formula}, saying that for real-valued Lipschitz-continuous functions $f$ and non-negative measurable functions $g$, we have\footnote{Some authors state the result for integrable $g$. The case of non-negative $g$ on $\Torus^d$ can be deduced by considering $g_n = g \one_{\{0\le g\le n\}}$ and using the monotone convergence theorem (though not needed here, one can also consider $g\ge 0$ on $\R^d$ via $h_n=g_n\one_{B_n}$, where $B_n$ is the ball of radius $n$).}
\[
\int_{f^{-1}(B)}g(x)|\nabla f(x)|\,\dd x = \int_B\Big(\int_{f^{-1}(t)}g(x)\,\dd H_{d-1}(x)\Big)\,\dd t \,,
\]
where $H_{d-1}$ is the $(d-1)$-dimensional Hausdorff measure. This replaces the integral over an inverse image by the integral over the various $(d-1)$-dimensional level sets of $f$.

Hence, since by assumption $\nabla E_j\neq 0$ a.e., we get $\mu_\psi(B) = \int_Bf_\psi(t)\,\dd t$ for the density $f_\psi(t) = \int_{E_j^{-1}(t)}\frac{|\psi(\theta)|^2}{|\nabla E_j(\theta)|}\,\dd H_{d-1}(\theta)$. In particular, if $|B|=0$ then $\mu_\psi(B)=0$ and $\mu_\psi$ is absolutely continuous. Since this holds for any $\psi$ and any $j$, this proves the first part.

Lastly, suppose that $\nabla E_j = 0$ on an open subset $S \subset \Torus^d$. Then $E_j(\cdot)$ is locally constant on $S$, hence $\Schr_{\Gamma}$ has an infinitely degenerate eigenvalue by Proposition~\ref{PartlyFlatGeneral}.
\end{proof}

\begin{exa}
The graphs defined by~\eqref{eq:exa3} and~\eqref{eq:pecu} have purely absolutely continuous spectra. Here, the coarea formula is simply a change of variables: $\dd\mu_\psi(t)=\frac{|\psi(f^{-1}(t))|^2}{|f'(f^{-1}(t)|}\,\dd t$, for $f=a,b$. That is, for measurable $B$, one has $\mu_\psi(B) = \int_B \frac{|\psi(\frac{1}{2}-\sqrt{t})|^2 + |\psi(\frac{1}{2}+\sqrt{t})|^2}{2\sqrt{t}}\,\dd t$ and $\mu_\psi(B) = \int_B (|\psi(\frac{1}{2}-t)|^2 + |\psi(\frac{1}{2}+t)|^2)\,\dd t$, respectively.

In view of Theorem~\ref{thm:parfla} and Proposition~\ref{prp:abcvsabc}, this completes the proof of the spectral types stated in Theorem~\ref{thm:intro}.
\end{exa}

\begin{rem}\label{rem:Luzin}
The main ingredient in the proof of Theorem~\ref{thm:purely_absolutely_continuous_spectrum} is that each $f=E_j$ should satisfy that $|B|=0\implies |f^{-1}(B)|=0$. This is known as the \emph{Luzin property $N^{-1}$}. We have given a proof using the coarea formula as this also gives access to the spectral measure. However, this property holds beyond Lipschitz continuity. Namely, it follows from \cite{Pon} that if $f:\T^d\to \R$ is \emph{approximately differentiable a.e.}, then $f$ has the $N^{-1}$ property iff $\nabla f_{ap}\neq 0$ a.e. Note that any function of locally bounded variation is approximately differentiable a.e., and that this includes Sobolev functions $W^{1,p}$ for any $1\le p\le \infty$, see \cite[Thm 6.4]{EG}.

This shows that if $f=E_j$ have some regularity, the above criterion is essentially a characterization. However, the situation appears to be wilder for $f$ which are merely continuous. Namely, it is shown in \cite{KT16} that there exists a continuous $f:[0,1]\to [0,1]$ which has the $N^{-1}$ property but such that $f_{ap}'$ exists almost nowhere.

Let us also stress that, for $d=1$, the assumption $f'\neq 0$ a.e. cannot be replaced by monotonicity. One can construct a Lipschitz continuous, strictly increasing function $f$ which violates the $N^{-1}$ property. 
This does not contradict the theorem. It is known in general that even $C^1[0,1]$ functions $f$ can be strictly increasing yet have $\{x:f'(x)=0\}$ of positive Lebesgue measure.
\end{rem}

\begin{exa}
The fractional Laplacian has purely absolutely continuous spectrum for any $\alpha>0$. Here $\nu=1$, so $E_j(\theta)$ is just $h(\theta) = (\sum_{i=1}^d 4\sin^2 \pi\theta_i)^\alpha$. This function is of locally bounded variation and has $\nabla h(\theta)\neq 0$ for any $ \theta\in \T^d\setminus \{0,\frac{1}{2},1\}$, where $x=(x,\dots,x)$ for $x=0,\frac{1}{2},1$. Note that, for $d=1$, $h$ is not Lipschitz if $\alpha<\frac{1}{2}$. We have $\sigma((-\Delta)^\alpha)=\operatorname{Ran}(h)=[0,(4d)^\alpha]$.
\end{exa}

\section{Transport}\label{Sec:transport}

\subsection{Speed of motion}\label{sec:speed}
It is known that in great generality, the speed of motion associated with Schr\"odinger operators cannot be faster than ballistic. This goes back to works of Radin and Simon \cite{RS} in the continuum, and has since been extended to locally finite graphs (periodic or not) in \cite[App. A]{BdMS23}, and also to some long-range operators in one dimension \cite[Thm 1.2]{JL} assuming the edge-weights decay fast enough. There is a natural question of whether the same holds true in our context of periodic graphs with weights decaying only as elements of $\ell^1$. Remarkably, this is no longer the case: motion can be super-ballistic, and the example is not artificial: it is the fractional Laplacian $(-\Delta)^\alpha$ for $0<\alpha\le \frac{1}{4}$.

\subsubsection{Super-ballistic motion}\label{sec:subal}
Consider the fractional Laplacian for $d=1$ (compare with Section~\ref{sec:fractio}). In this case the Floquet function reduces to $h(\theta)=4^\alpha \sin^{2\alpha}\pi\theta$. 

\begin{prp}\label{prp:smalquart}
The fractional Laplacian $(-\Delta)^\alpha$ on $\Z$ exhibits super-ballistic motion for any $\alpha\le \frac{1}{4}$. More precisely, one has $\|x\ee^{\ii t(-\Delta)^\alpha}\delta_0\|=\infty$ for any $t\neq 0$, showing blow up in finite time.
\end{prp}
Here, $(x\psi)(n):= n \psi(n)$ for $n\in \Z$, and $\|x\ee^{-\ii t H}\psi\|^2$ is commonly referred to as the \emph{mean squared displacement}; see Remark~\ref{rem:xpsi} for more comments on this quantity. 

Proposition~\ref{prp:smalquart} identifies a phase transition in the speed of motion of the fractional Laplacian: indeed, Example~\ref{eq:exabal} shows that motion is standard ballistic if $\alpha>\frac{1}{4}$.\footnote{The same value of $\alpha$ appears in the the context of \emph{rough paths}. It is known that the second moment of the L\'evy area of the $\varepsilon$-regularized fractional Brownian motion diverges as $\varepsilon\to 0$ if $\alpha \le \frac{1}{4}$, and converges if $\alpha>\frac{1}{4}$, see \cite{Unt}.} Curiously, this phenomenon only exists in $d=1$. For $d\ge 2$, as we show in Example~\ref{eq:exabal}, motion is always standard ballistic.

It is interesting to note that the fractional Laplacian was known to exhibit ``anomalous diffusion''. It seems this was first observed in \cite{Cha} in the continuum, see also \cite[eq. (29)]{PKL+} for results on the lattice. However, this referred to the behaviour of the heat kernel $\ee^{-t (-\Delta)^\alpha}\delta_0$. To our knowledge, the present results on the Schr\"odinger evolution $\ee^{\ii t (-\Delta)^\alpha}$ are new. 
\begin{proof}
Assume that $\beta:=2\alpha \le \frac{1}{2}$ and let $t\neq 0$. Let $\psi_t=\ee^{\ii t(-\Delta)^\alpha}\delta_0$ and suppose to the contrary that $\|x\psi_t\|<\infty$. Then $\sum_n |n|^2 |\psi_t(n)|^2<\infty$. But $\psi_t(n)= \langle \delta_n,\ee^{\ii t\Delta_\alpha}\delta_0\rangle = \langle \mathscr{F}^{-1}\delta_n,\mathscr{F}^{-1}\ee^{\ii t\Delta_\alpha}\delta_0\rangle = \int_0^1 \ee^{-2\pi\ii n\theta} \ee^{\ii th(\theta)}\,\dd\theta = \widehat{\phi_t}(n)$, where $\phi_t(\theta) = \ee^{\ii th(\theta)}$. Thus, we have $\sum_n |n|^2 |\widehat{\phi}_t(n)|^2<\infty$. This implies that $\phi_t\in H^1(\T)$, so it has a weak derivative, namely $(\partial_w\phi_t)(\theta)=2\pi\ii \sum_n  n \widehat{\phi_t}(n) \ee^{2\pi\ii n\theta}$, and by hypothesis $\partial_w\phi_t \in L^2(\T)$. Let us show that this cannot be true.

For $\theta\in (0,1)$, $\phi_t$ has a classical derivative 
\[
    \phi_t'(\theta) 
    = -\ii t \ee^{\ii t h(\theta)}h'(\theta) 
    = 
    \frac{-2^\beta\beta \pi\ii t \cos\pi\theta}{(\sin \pi\theta)^{1-\beta}}\ee^{\ii t h(\theta)}.
\]
Moreover, $\phi_t' \in L^1(\T)$. In fact, 
\[
    \int_0^1 \frac{|\cos \pi\theta|}{(\sin\pi\theta)^{1-\beta}}\,\dd\theta 
    = 
    2 \int_0^{1/2} \frac{\cos \pi\theta}{(\sin\pi\theta)^{1-\beta}}\,\dd\theta = \frac{2(\sin \pi \theta)^\beta}{\beta \pi}\bigg|_0^{1/2} = \frac{2}{\beta\pi}. 
\]    
    It follows that $\phi_t'=\partial_w \phi_t$, by uniqueness of the weak derivative. However, for some constants $c_\beta,C_\beta > 0$, 
\begin{align*}
    \|\phi_t'\|_{L^2}^2 
    &= 
    c_\beta t^2\int_0^1 \frac{\cos^2\pi\theta}{(\sin\pi\theta)^{2(1-\beta)}}\,\dd\theta 
    = 
    2c_\beta t^2\int_0^{1/2}\frac{\cos^2\pi\theta}{(\sin \pi\theta)^{2(1-\beta)}}\,\dd\theta
    \\
    &\geq 
    C_\beta t^2 \int_0^{1/4}\frac{1}{\theta^{2(1-\beta)}}\,\dd\theta
    =
    \infty, 
\end{align*}
    since $\beta\le \frac{1}{2}$, where we used that $\cos^2x = 1-\sin^2 x \ge 1-x^2\ge 1-\frac{\pi^2}{16}$ on $[0,\frac{\pi}{4}]$, and $\sin\pi\theta \le \pi\theta$.
This contradiction shows that $\|x \psi_t\|=\infty$ for any $t>0$.
\end{proof}

\subsubsection{Ballistic motion} We now show that for a large family of non-locally finite graphs, the speed of motion is partly ballistic. From \S\,\ref{sec:subal}, we see that we need a few (rather weak) additional assumptions for this.


Given $\psi\in \ell^2(V)\equiv \ell^2(\Z^d)^\nu$, we  define $x\psi$ by
\begin{equation}\label{eq:x}
(x\psi)(k) := (k_1\psi(k),\dots,k_d\psi(k))=k\psi(k)
\end{equation}
for $k\in \Z^d$, where each $k_i\psi(k)$ is a vector with $\nu$ entries $k_i\psi_p(k)$ for $p=1,\dots,\nu$.

We also introduce the function $\nu':\Torus^d \rightarrow \N$ counting the number of distinct eigenvalues of $H(\theta)$. This function is not constant in general, even in the analytic setting. Let $\mathcal{J}\subset \Torus^d$ be its points of discontinuity. Then over the open set $\Torus^d\setminus \overline{\mathcal{J}}$, the function $\nu'$ is locally constant, i.e.\  has a fixed value over each connected component of $\Torus^d\setminus \overline{\mathcal{J}}$.
\begin{thm}[Transport]\label{thm:bal}
Let $\Gamma$ be a crystal such that $\Schr_{\Gamma}$ has Lipschitz-continuous Floquet elements $h_{ij}$. 
Assume moreover that the closure of the jump set $\overline{\mathcal{J}}$ of $\nu'$ has Lebesgue measure zero. Then for any $\psi\in \ell^2(V)$ with $\|x\psi\|<\infty$, we have
\[
\lim_{t\to\infty} \frac{\|x\ee^{-\ii t \Schr_{\Gamma}}\psi\|^2}{t^2} =\frac{1}{4\pi^2} \int_{\Torus^d} \sum_{n=1}^\nu |\nabla_\theta E_n(\theta)|^2\| P_n(\theta)(U\psi)(\theta)\|_{\C^\nu}^2\,\dd\theta \,,
\]
where $P_n(\theta)=\langle\cdot , u_n(\theta)\rangle u_n(\theta)$ for an orthonormal eigenbasis $(u_n(\theta))_{n=1}^\nu$ of $H(\theta)$ corresponding to $E_n(\theta)$, and $|\nabla_\theta E_n(\theta)|^2 = \sum_{i=1}^d (\partial_{\theta_i} E_n(\theta))^2$.
\end{thm}
This limit is finite under the assumptions, which imply that $E_n$ is Lipschitz continuous by Lemma~\ref{lem:eigenlip}. Also recall that one way to ensure Lipschitz continuity of the entries $h_{ij}$ of the Floquet matrix is to assume $\sum_k w_{ij}(k)|k|_2<\infty$ for all $i,j$, cf. Proposition~\ref{prp:matrilip}.

\begin{rem}\label{rem:xpsi}
Whenever the right-hand side of the above limit is non-zero, the transport is said to be \emph{ballistic}. Let us explain the statement for the reader who may not be familiar with such estimates. One usually chooses $\psi$ of compact support. In a regime of localization, $\ee^{-\ii t \Schr_\Gamma}\psi$ stays essentially in a compact set, so even after multiplying it by a diverging factor $x$, the norm remains finite for all times and $\frac{\|x\ee^{-\ii t\Schr_\Gamma}\psi\|^2}{t^2}\to 0$. By contrast, in a regime of delocalization, one expects $\ee^{-\ii t\Schr_\Gamma} \psi$ to spread out with time, so multiplying it by $x$ should make its norm grow with time. The quantity $\|k_1 \ee^{-\ii t\Schr_\Gamma}\psi\|^2 = \sum_{k\in \Z^d} \sum_{p=1}^\nu |k_1|^2 |(\ee^{-\ii t\Schr_\Gamma}\psi)_p(k)|^2$ measures the squared displacement in the $1$st cardinal direction. The sum $\|x\ee^{-\ii t\Schr_\Gamma}\psi\|^2 = \sum_{i=1}^d \|k_i \ee^{-\ii t\Schr_\Gamma}\psi\|^2$ of these contributions is known as the mean squared displacement. The limit on the LHS is viewed as a mean asymptotic speed. If we want to focus on one direction, the proof shows that we can replace $x$ by $k_j$ in the LHS and get $\partial_{\theta_j}$ in the RHS instead of $\nabla_\theta$.

We refer the reader to~\cite{DMY} for more background and~\cite{BDMY} for a recent result.\footnote{Studying $\|x\ee^{-\ii t \cH_\Gamma}\|$ is standard for graphs embedded in Euclidean space. For locally-finite graphs in general, it makes sense to study $\|d(\cdot,o)\ee^{-\ii t \cH_\Gamma}\|$ instead, where $d(x,o)$ is the graph distance from a fixed origin $o$. In our framework, however, this not appropriate because there are many graphs for which $o\sim x$ for all $x\in \Gamma$, so that the factor $d(o,x)\equiv 1$ is not divergent and does not measure any spreading.}
\end{rem}

\begin{rem}
    When all $h_{ij}$ are analytic, $\mathcal{J}$ is a subvariety of $\Torus^d$ of dimension $\le d-1$, see~\cite{Wil}, in particular $|\overline{\mathcal{J}}|=0$. But this assumption seems to be satisfied in much greater generality; in fact, we are not aware of any concrete graph violating it.
\end{rem}

%
\begin{proof}
  The proof of Theorem~\ref{thm:bal} is similar to~\cite[Thm
  3.3]{BdMS23} except that we need further justifications
  here. By~\eqref{eq:utile},
  \[
    \partial_{\theta_i} (U\psi)_p(\theta)
    = \sum_{k\in \Z^d}\ee^{2\pi\ii \theta\cdot k}\ 2\pi\ii k_i \ \psi_p(k)
  \]
  and so
  \[
    \frac{1}{2\pi \ii} \nabla_\theta (U\psi)_p(\theta)
    = \sum_{k\in\Z^d} \ee^{2\pi\ii \theta\cdot k} (k_1\psi_p(k),\dots,k_d\psi_p(k))
    = (Ux\psi)_p(\theta).
  \]
  It follows that
\[
(Uxe^{-\ii t\Schr_{\Gamma}}\psi)(\theta) = \frac{1}{2\pi\ii} \nabla_\theta\ee^{-\ii t H(\theta)}(U\psi)(\theta)\, .
\]
Now we expand $H(\theta) = \sum_{n=1}^{\nu'} \lambda_n(\theta) P_{\gamma_n}(\theta)$, where $\lambda_n(\theta)$ are the distinct eigenvalues of $H(\theta)$ and $\gamma_n$ is a small countour around $\lambda_n$ containing no other eigenvalue, cf. Lemma~\ref{lem:eigenlip}. Observe that if we arrange the eigenvalues as $E_1(\theta)\le \dots\le E_\nu(\theta)$ counting multiplicity, we can choose $\lambda_n(\theta)=E_{k_n}(\theta)$ to be for example the first eigenvalue function in the enumeration equal to $\lambda_n(\theta)$. We now assume $\theta$ is in the open set $\Torus^d\setminus\overline{\mathcal{J}}$, so that in a small neighbourhood $O$ of $\theta$, no splitting or merging of eigenvalues occurs and $\nu'(\theta)=\nu'$ is fixed.

In view of Lemma~\ref{lem:eigenlip}, we can differentiate over $O$ to obtain
\begin{align}\label{e:nabco}
\nabla_\theta \ee^{-\ii tH(\theta)}(U\psi)(\theta) &= \sum_{n=1}^{\nu'} \nabla_\theta[\ee^{-\ii tE_{k_n}(\theta)}P_{\gamma_n}(\theta)(U\psi)(\theta)] \nonumber \\
&= \sum_{n=1}^{\nu'} \ee^{-\ii tE_{k_n}(\theta)}(-\ii t\nabla_\theta E_{k_n}(\theta))P_{\gamma_n}(\theta)(U\psi)(\theta) + O_{\theta}(1)
\end{align}
where $O_\theta(1)= \sum_{n=1}^{\nu'} \ee^{-\ii tE_{k_n}(\theta)} \nabla_\theta [P_{\gamma_n}(\theta)(U\psi)(\theta)]$. As $\theta\mapsto P_{\gamma_n}(\theta)$ is Lipschitz-continuous, its gradient as a matrix valued function exists a.e. Also, $\nabla_\theta(U\psi)(\theta)=(Ux\psi)(\theta)$, so $\int \|\nabla_\theta (U\psi)(\theta)\|^2_{\C^\nu}<\infty$ by assumption. We thus showed that for $
\theta\in O$
\[
\lim_{t\to\infty} \frac{(Uxe^{-\ii t\Schr_{\Gamma}}\psi)(\theta)}{t} = - \frac{1}{2\pi} \sum_{n=1}^{\nu'} \ee^{-\ii tE_{k_n}(\theta)}(\nabla_\theta E_{k_n}(\theta))P_{\gamma_n}(\theta)(U\psi)(\theta)\,.
\]
Recall that $\gamma_n$ is a small contour around $\lambda_n(\theta)$, which, for $\theta \in O$, only encloses the eigenvalue $E_{k_n}(\theta)$ the multiplicity  of which remains fixed over $O$. 
In particular $E_{k_n}(\theta) = E_{k_n+1}(\theta)=\dots=E_{k_n+m-1}(\theta)$ and $P_{\gamma_n}(\theta) = \sum_{E_k(\theta)=\lambda_n(\theta)}P_{E_k}(\theta)= \sum_{j=0}^{m-1} P_{E_{k_n+j}}(\theta)$ which yields 
\[
	\sum_{n=1}^{\nu'} f(E_{k_n}(\theta))P_{\gamma_n}(\theta) = \sum_n\sum_j f(E_{k_n+j}(\theta))P_{E_{k_{n+j}}}(\theta)=\sum_{k=1}^\nu f(E_k(\theta))P_{E_k}(\theta)
\]
Therefore
\begin{equation}\label{e:balpoi}
\lim_{t\to\infty} \frac{(Uxe^{-\ii t\Schr_{\Gamma}}\psi)(\theta)}{t} = - \frac{1}{2\pi}  \sum_{k=1}^{\nu} \ee^{-\ii tE_{k}(\theta)}(\nabla_\theta E_{k}(\theta))P_{E_k}(\theta)(U\psi)(\theta)
\end{equation}
for $\theta\in O$. On the other hand, since $\overline{\mathcal{J}}$ has measure zero, we have
\[
\lim_{t\to\infty} \frac{\|x\ee^{-\ii t\Schr_{\Gamma}}\psi\|^2}{t^2}=\lim_{t\to\infty} \frac{\|Ux\ee^{-\ii t\Schr_{\Gamma}}\psi\|^2}{t^2} = \lim_{t\to\infty} \int_{\Torus^d\setminus \overline{\mathcal{J}}} \frac{\|(Ux\ee^{-\ii t\Schr_{\Gamma}}\psi)(\theta)\|^2_{\C^\nu}}{t^2}\,\dd\theta.
\]
The statement thus follows from \eqref{e:balpoi} and the orthogonality of the $P_{E_k}$, using dominated convergence, which is applicable as the underlying functions in \eqref{e:nabco} are Lipschitz.
\end{proof}

\begin{defa}
We define the \emph{layers} of a crystal $\Gamma$ as the sets $L_i = \{v_i+k_\fa : k\in \Z^d\}$.
\end{defa}

This definition is motivated by the graphs in Figure~\ref{fig:3simp}, each consisting of three (horizontal) layers.

\begin{cor}\label{cor:balayer}
  Under the assumptions of Theorem~\ref{thm:bal}, the operator $\Schr_{\Gamma}$
  exhibits ballistic transport along at least one layer. 
  More precisely, there exists $i \in \{1,\dots, \nu\}$ such that
  $\lim_{t\to\infty} \frac{\|x\ee^{-\ii t \Schr_\Gamma} \delta_w \|^2}{t^2} \neq 0$
  for any $w \in L_i = \{v_i+k_\fa:k\in \Z^d\}$.
\end{cor}
\begin{proof}
Let $w=w_i=v_i+k_\fa$. By the definition of the Floquet transform \eqref{eq:utile}, we have $(U\delta_w)(\theta) = \delta_i \ee^{2\pi\ii \theta\cdot k}$. 
In particular, for all $n \in \{1, \dots, \nu \}$, we have $P_n(\theta)(U\delta_w)(\theta) = \ee^{2\pi\ii \theta\cdot k}\langle \delta_i,u^{(n)}(\theta)\rangle u^{(n)}(\theta)$ and $\|P_n(\theta)(U\delta_w)(\theta)\|^2 = |u_i^{(n)}(\theta)|^2$, independently of $k$, where $P_n(\theta)$ are the orthogonal projections corresponding to the eigenvalues $E_n(\theta)$, cf. Theorem~\ref{thm:bal}. It follows that for all $k \in\Z^d$,
\begin{equation}\label{eq:sumlim}
	\sum_{i=1}^\nu \lim_{n\to \infty} \frac{\|x\ee^{-\ii t\Schr_\Gamma}\delta_{w_i}\|^2}{t^2} 
	= 
	\frac{1}{4\pi^2}\int_{\Torus^d}\sum_{n=1}^\nu |\nabla_\theta E_n(\theta)|^2\sum_{i=1}^\nu |u_i^{(n)}(\theta)|^2\,\dd\theta\,,
\end{equation}
where we used Tonelli's theorem to permute the sums. Since $\sum_{i=1}^\nu |u_i^{(n)}(\theta)|^2 = \|u^{(n)}(\theta)\|^2=1$, the RHS reduces to $\frac{1}{4\pi^2}\int_{\Torus^d}\sum_{n=1}^\nu |\nabla_\theta E_n(\theta)|^2\,\dd\theta$. This quantity cannot be zero. In fact, if it were, we would have $|\nabla_\theta E_n(\theta)|^2 = 0$ for all $n$ and a.e. $\theta$. Since $E_n$ is Lipschitz, this would mean that for all $n$, the function $E_n(\theta)$ is constant.\footnote{Indeed, this follows directly from the fundamental theorem of calculus as given in~\cite[Thm 7.18]{Rud} for absolutely continuous functions, if $d=1$. For higher $d$, one can see this by convolving with mollifiers $(\phi_\eps)$. That is, if $F$ is Lipschitz, then $F\in W^{1,\infty}$ by~\cite[Thm 4.5]{EG}, hence $\nabla_\theta (F\ast \phi_\eps) = (\nabla_\theta F)\ast \phi_\eps=0$ a.e. if $\nabla_\theta F=0$ a.e. As $F\ast \phi_\eps$ is smooth, this implies $F\ast\phi_\eps$ is a constant $c_\eps$. As $F\ast \phi_\eps \to F$, this implies $F$ is also constant. It is crucial here that $F$ be Lipschitz, the Cantor function being a counterexample otherwise.} But this is impossible in view of Theorem~\ref{thm:top}, as the top band cannot be fully flat.

We have shown the sum of limits~\eqref{eq:sumlim} cannot vanish, so at least one limit is non-zero. This completes the proof.
\end{proof}

\begin{cor}\label{cor:balnu1}
If the fundamental cell contains only one vertex, $\nu=1$, and the Floquet function $H(\cdot)=h(\cdot)$ defining $\Gamma$ is Lipschitz continuous, then for any $\psi\in \ell^2(V)$ with $\|x\psi\|<\infty$, we have
\[
\lim_{t\to\infty} \frac{\|x\ee^{-\ii t \Schr_{\Gamma}}\psi\|^2}{t^2} =\frac{1}{4\pi^2} \int_{\Torus^d} |\nabla_\theta h(\theta)|^2 |\Ftrafo{\psi}(\theta)|^2\,\dd\theta \, .
\]
If $\psi = \delta_w$ for some $w\in \Gamma$, then $|\Ftrafo{\psi}(\theta)|^2=1$ and the limit is non-zero. In this case, it suffices that $H=h\in W^{1,2}(\T^d)$ instead of Lipschitz.
\end{cor}
\begin{proof}
This follows immediately from the previous results, here there is no jump set since there is a single eigenvalue, i.e.\ $\nu'(\theta)=\nu=1$ $\forall \theta$. Also, $P_n(\theta)=\operatorname{Id}$ and $U$ reduces to the Fourier transform. In case $\psi=\delta_w$, $\widehat{\psi}(\theta)=\ee^{2\pi \ii \theta \cdot \lfloor w\rfloor_\fa}$, and the same proof continues to hold as long as $\nabla h\in L^2$.
\end{proof}

\begin{exa}\label{eq:exabal}
For simplicity we consider examples with $\nu=d=1$.
\begin{enumerate}[(a)]
\item Recall the graph defined by the Floquet function~\eqref{eq:parfla} which exhibited a partially flat band. We obtain $\lim_{t\to\infty} \frac{\|x\ee^{-\ii t\Schr_\Gamma}\psi\|^2}{t^2}=\frac{1}{4\pi^2}\int_{[0,\frac14]\cup[\frac34,1]} |\Ftrafo{\psi}(\theta)|^2\,\dd\theta$. If $\Ftrafo{\psi}$ is supported in $[\frac14,\frac34]$, this limit is zero. This is not surprising as such $\psi$ are eigenvectors, for $\Mult_h\Ftrafo{\psi} = 0$ and thus $\Schr_\Gamma \psi = \Schr_\Gamma U^{-1}\Ftrafo{\psi} = U^{-1} \Mult_h\Ftrafo{\psi}=0$. More generally, if $\psi$ is some eigenvector for an eigenvalue $\lambda$, then $\|x \ee^{-\ii t\Schr_\Gamma}\psi\|^2= \|x\ee^{-\ii t \lambda}\psi\|^2= \|x\psi\|^2$ is independent of time so the limit will always be zero.

In contrast, if we take $\psi=\delta_n$ for some $n$, the limit is $\frac{1}{8\pi^2}\neq 0$, meaning that part of the mass of $\delta_n$ moves at ballistic speed.
\item For the graph defined by the Floquet function~\eqref{eq:pecu}, we obtain $\lim_{t\to\infty} \frac{\|x\ee^{-\ii t\Schr_\Gamma}\psi\|^2}{t^2}=\frac{1}{4\pi^2}\int_{0}^1 |\Ftrafo{\psi}(\theta)|^2\,\dd\theta = \frac{1}{4\pi^2} \|\psi\|^2$. The limit is non-zero for every nontrivial $\psi$ such that $\|x\psi\|<\infty$. So motion is ballistic over the whole domain of $x$.
\item For the graph defined by the Floquet function~\eqref{eq:exa3}, the limit is $\frac{1}{2\pi^2}\int_0^1 |\theta-\frac{1}{2}||\Ftrafo{\psi}(\theta)|^2\,\dd\theta$, which is again non-zero for any nontrivial $\psi$. In fact, if the integral was zero, we would have $\Ftrafo{\psi}(\theta)=0$ a.e., so $\Ftrafo{\psi}=0$ and thus $\psi=0$, a contradiction.
\item Consider the fractional Laplacian $(-\Delta)^\alpha$, with $h(\theta) = (\sum_{i=1}^d 4\sin^2 \pi\theta_i)^\alpha$. It satisfies 
\[
    \partial_j h(\theta) = \frac{4^\alpha (2\pi \sin \pi\theta_j \cos \pi\theta_j)}{(\sum_{i=1}^d \sin^2\pi\theta_i)^{1-\alpha}}.
\] 
Since $|\partial_j h(\theta)|^2$ is unchanged under the substitution $\theta_i \mapsto 1-\theta_i$, in order to show that $\int_{\T^d} |\partial_j h(\theta)|^2 \dd\theta<\infty$, it is enough to show that $\int_{\Omega} |\partial_j h(\theta)|^2\dd\theta<\infty$, where $\Omega=[0,\frac{1}{2}]^d$ (e.g. in $d=1$ we have $\int_{\T}|h'|^2 = 2 \int_{\Omega}|h'|^2$).

As $\sin x\le x$ and $\sin x\ge \frac{2x}{\pi}$ for $x\in (0,\frac{\pi}{2})$, as follows from the concavity of the sine, we get 
\[
    \int_\Omega |\partial_j h(\theta)|^2 \le C \int_{\Omega}\frac{\theta_j^2}{(\sum_{i=1}^d \theta_i^2)^{2-2\alpha}}
\]
for $C > 0$. This integral may be bounded by the integral over a ball of radius $r_0$, yielding 
\[
    \int_0^{r_0} \dd r\int_0^{2\pi}\dd \varphi \int_0^\pi \dd\vartheta_1\cdots\int_0^\pi \dd\vartheta_d \frac{r^2 f_j(\vartheta)}{(r^2)^{2-2\alpha}} r^{d-1} g(\vartheta)
\]
in spherical coordinates, where $f_j(\vartheta)$ and $g(\vartheta)$ are products of cosine and sine functions, with $|f_j(\vartheta)|\le 1$ and $|g(\vartheta)|\le 1$. Hence, the integral is bounded by some $C \int_0^{r_0} r^{d+1-4+4\alpha}\dd r$. This converges as long as $4\alpha+d-3>-1$, meaning $\alpha> \frac{2-d}{4}$. In this case $h\in W^{1,2}(\T^d)$ and motion is ballistic by Corollary~\ref{cor:balnu1}.

For $d=1$, this imposes the condition $\alpha>\frac{1}{4}$. This condition is sharp, as Proposition~\ref{prp:smalquart} shows that motion becomes super-ballistic for $\alpha \le \frac{1}{4}$.

For $d\ge 2$, the condition $\alpha>\frac{2-d}{4}$ is automatically satisfied, since $\alpha>0$.

Conclusion: For $d=1$, motion changes from ballistic for $\alpha>\frac{1}{4}$ to super-ballistic if $\alpha\le \frac{1}{4}$. For $d\ge 2$, no such phase transition exists: motion is always standard ballistic.
\end{enumerate}
\end{exa}

The claims of ballistic transport in Theorem~\ref{thm:intro} now follow from Proposition~\ref{prp:abcvsabc}.

\subsection{Dispersion}\label{sec:dis}
Another aspect of transport is to study whether the waves disperse as
time goes on. Generally speaking, dispersive estimates are supremum
bounds of the form
$\|\ee^{-\ii t \Schr_{\Gamma}}\psi\|_\infty \le t^{-\alpha}
\|\psi\|_1$, $\alpha > 0$. Since
$\|\ee^{-\ii t\Schr_{\Gamma}}\psi\|_2 = \|\psi\|_2$ holds, such a
supremum bound implies that the wave is spreading out, and moreover,
the wave flattens out as time goes on. This excludes, in particular,
scenarios such as having a wave which is simply a bump that slides as
time goes on (soliton-like behaviour). However, although such a bump
wave would not disperse, it could in principle travel with ballistic
speed.

Dispersion is understood mostly in the context of free Laplacians and weak perturbations thereof, both in the continuum $\R^d$ and in $\Z^d$ and with speed $\alpha=d/2$ and $d/3$, respectively. The periodic setting is in general more difficult. Indeed, proving dispersive estimates for periodic Schr\"odinger operators on periodic graphs is already largely open in the locally finite case, even though ballistic transport is understood (see~\cite{BdMS23} and references therein). In fact, we are only aware of the recent works~\cite{YZ20,YZ22} treating the case of a periodic potential on $\Z$, and~\cite{AS23} treating dispersion on graphs of the form of a Cartesian product $\Gamma=\Z^d\mathop\square G_F$ for a finite graph $G_F$, mostly for adjacency matrices. In particular, proving dispersion for general periodic potentials on the simplest case of $\Z^d$, $d>1$, would already be interesting. The situation is also not much better in the continuum: for periodic Schr\"odinger operators, one again only understands the one-dimensional case~\cite{Fir96,Cu06,Ca06}.

In the following, we finish the proof of Theorem~\ref{thm:intro} in \S~\ref{sec:nodis}--\ref{sec:disfa}. Then in \S~\ref{sec:frapodis}, we discuss dispersion for fractional/integer powers of the standard Laplacian on $\Z^d$. Studying $\ee^{\ii t \cH_\Gamma}$ or $\ee^{-\ii t \cH_\Gamma}$ makes no difference to the Schr\"odinger wave dynamics as the operator $\cH_\Gamma$ is simply replaced by minus the operator.

\subsubsection{A graph with no dispersion}\label{sec:nodis}
Consider the periodic graph $\Gamma$ induced by the Floquet
function~\eqref{eq:parfla}. It is of course not surprising that the
eigenvalue at the spectral bottom will cause the corresponding
eigenvector $\phi$ to stay frozen and not disperse, since
$\Schr_{\Gamma}\phi = 0$ implies
$\ee^{\ii t\Schr_{\Gamma}}\phi = \phi$ for all times. However, it is
still interesting to understand the evolution of the more natural
initial states $\psi_0 := \delta_n$, for $n\in\Z$.

Given $n,m\in\Z$, we notice that
\begin{equation}
  \label{eq:eitaker}
  (\ee^{\ii t\Schr_{\Gamma}}\delta_n)(m)
  =\langle U \delta_m, (U \ee^{\ii t\Schr_{\Gamma}} U^{-1}) U\delta_n\rangle
  = \int_0^1 \ee^{2\pi\ii(n-m)\theta}\ee^{\ii t a(\theta)}  \dd \theta,
\end{equation}
%
%
That is
\[
  (\ee^{\ii t\Schr_{\Gamma}}\delta_n)(m)
  = \int_0^{1/4} \ee^{2\pi\ii(n-m)\theta}\ee^{\ii t(\frac14-\theta)}\dd\theta
  + \int_{1/4}^{3/4} \ee^{2\pi\ii(n-m)\theta}\,\dd\theta
  + \int_{3/4}^1 \ee^{2\pi\ii(n-m)\theta}\ee^{\ii t(\theta-\frac34)}\dd\theta.
\]
In particular, for any $t>0$,
\begin{align}
  (\ee^{\ii t\Schr_{\Gamma}}\delta_n)(n)
  &= \frac{1}{2} + \ee^{\ii t/4}\frac{\ee^{-\ii t/4}-1}{-\ii t}
    + \ee^{-3\ii t/4}\frac{\ee^{\ii t}-\ee^{3\ii t/4}}{\ii t}\\
  \label{eq:14dis}
  &= \frac{1}{2}+2\frac{\ee^{\ii t/4}-1}{\ii t}.
\end{align}
Thus, at time $t=0$, we start with all the mass concentrated at $n$,
namely $\delta_n(n)=1$, and at \emph{all times}, a sizable mass
remains at $n$, namely~\eqref{eq:14dis} easily implies that
$|\ee^{\ii t\Schr_{\Gamma}}\delta_n(n)|\ge \frac{3}{10}$. In
particular,
$\|\ee^{\ii t\Schr_\Gamma}\delta_n\|_\infty \ge \frac{3}{10}$ 
and there is no dispersion. At the level of vertices, as $t\to\infty$, one
has $|(\ee^{\ii t\Schr_{\Gamma}}\delta_n)(n)|\to\frac{1}{2}$, which
means that asymptotically, a quarter of the square modulus remains at
the starting point, while the other $\frac34$ travels. This traveling
component moves at ballistic speed, as we saw in
Example~\ref{eq:exabal}(a).

\subsubsection{A graph with two traveling tents} In this section we prove the dispersive behaviour stated in Theorem~\ref{thm:intro}(2). In view of Proposition~\ref{prp:abcvsabc}, it is equivalent to studying the Schr\"odinger operator defined by the Floquet function $b(\theta)$ in~\eqref{eq:pecu}.
\begin{thm}\label{thm:open}
Let $\cH_\Gamma$ be the Schr\"odinger operator corresponding to the Floquet function $b(\theta)=|\frac{1}{2}-\theta|$, given in~\eqref{eq:pecu}. Then the unitary group $(\ee^{-\ii t \Schr_\Gamma})_{t \in \R}$ does not satisfy a dispersive estimate: $\|\ee^{-\ii t\Adja_\Gamma}\delta_n\|_\infty >\frac{1}{\pi}\|\delta_n\|_1$ at all $t \in \R$.

More specifically, the mass of the wave is concentrated at $n$ at time zero, then splits into two tents traveling to the right and left from $n$ at linear speed, each peak having a square modulus $\ge \frac14$, and the mass decays polynomially away from the peaks.
\end{thm}
To our knowledge, this is the first time that such a sliding tents behaviour has been identified for a linear Schr\"odinger operator. As mentioned in the introduction, Theorem~\ref{thm:intro}(2) gives a negative to \cite[Open Problem 9]{DMY}, which was actually stated for general Schr\"odinger operators, so it is remarkable that we find a counterexample in the periodic realm, however this is in the larger class of essentially local operators.
\begin{proof}
Without loss, we may assume $n = 0$ and study $\ee^{-\ii t\Schr_\Gamma}\delta_0$. Arguing as in \S~\ref{sec:nodis}, we obtain,
\begin{equation}
\label{eq:eitanm}
\begin{split}
(\ee^{-\ii t\Schr_{\Gamma}}\delta_0)(m)&=\int_0^{1/2} \ee^{-2\pi\ii m\theta}\ee^{-\ii t(\frac{1}{2}-\theta)}\,\dd\theta + \int_{1/2}^1\ee^{-2\pi\ii m\theta}\ee^{-\ii t(\theta-\frac{1}{2})}\,\dd\theta \\
&= \frac{\ee^{-\pi\ii m}-\ee^{-\ii t/2}}{\ii t-2\pi\ii m} + \frac{\ee^{-2\pi\ii m-\ii t/2}-\ee^{-\pi\ii m}}{-\ii t-2\pi\ii m}\\
&=(\ee^{-\pi\ii m}-\ee^{-\ii t/2})\cdot \frac{-2\ii t}{t^2-4\pi^2m^2}\,.
\end{split}
\end{equation}
The previous calculation holds as long as $t\neq 2\pi |m|$, as $t\ge 0$. Hence, for $t\notin 2\pi \N$, using $|\ee^{\ii x}-1|^2=2-2\cos(x) = 4\sin^2(\frac{x}{2})$, we find
\begin{equation}\label{eq:modpecu}
|(\ee^{-\ii t\Schr_{\Gamma}}\delta_0)(m)| = \frac{4t \left|\sin\frac{1}{2}(\frac{t}{2}-\pi m)\right|}{|t^2-4\pi^2 m^2|} = \frac{t \left|\sin(\frac{t-2\pi m}{4})\right|}{\left|\frac{t}{2}-\pi m\right|\left|\frac{t}{2}+\pi m\right|}\,.
\end{equation}

For $m=\lfloor \frac{t}{2\pi}\rfloor$, as $\frac{t}{2\pi}$ is not an integer, we have $0< \frac{t-2\pi m}{4} < \frac{t-2\pi(\frac{t}{2\pi}-1)}{4}= \frac{\pi}{2}$. Recall $\sin x\ge \frac{2x}{\pi}$ for $x\in (0,\frac{\pi}{2})$. This implies that $\frac{|\sin x|}{2|x|}\ge \frac{1}{\pi}$. On the other hand, $0< \lvert \frac{t}{2}+\pi m \rvert < t$. Thus, $|(\ee^{-\ii t\Schr_{\Gamma}}\delta_0)(m)|> \frac{1}{\pi}$ and hence we conclude that $\|\ee^{-\ii t\Schr_\Gamma}\delta_0\|_\infty \ge \frac{1}{\pi}$ if $t\notin \{2\pi k\}_{k\ge 0}$.


On the other hand, if $t=2\pi k$ for some $k>0$, we get

\[
    \begin{split}
        (\ee^{-\ii t \Schr_{\Gamma}}\delta_0)(k) &= \int_0^{1/2} \ee^{-2\pi\ii k \theta}\ee^{-2\pi \ii k(\frac{1}{2}-\theta)}\,\dd\theta + \int_{1/2}^1 \ee^{-2\pi\ii k\theta}\ee^{-2\pi\ii k(\theta-\frac{1}{2})}\,\dd\theta\\
&= \frac{\ee^{-\pi\ii k}}{2}+ \frac{\ee^{-3\pi\ii k}-\ee^{-\pi\ii k}}{-4\pi\ii k} = \frac{\ee^{-\pi\ii k}}{2}
    \end{split}
\]
and so $|(\ee^{-\ii t \Schr_{\Gamma}}\delta_0)(k)| = \frac{1}{2}$ and $\|\ee^{-\ii t \Schr_{\Gamma}}\delta_0\|_\infty \ge \frac{1}{2}$ . Summarizing, we have proved that for any $t\ge 0$, we have $\|\ee^{-\ii t \Schr_{\Gamma}}\delta_0\|_\infty \ge \frac{1}{\pi}$. In particular, dispersion does not occur.

Let us now study the profile of the wave more precisely. Clearly, for $t > 0$, one has $(\ee^{\ii t\Schr_\Gamma}\delta_0)(-m) = (\ee^{\ii t\Schr_\Gamma}\delta_0)(m)$.
At time $t=0$, all of the mass is concentrated at $0$. 
If $t > 0$, $t \not \in 2 \pi \N$, there is a non-zero mass on all vertices as we see from~\eqref{eq:modpecu}, however, it falls off quadratically $\asymp |m|^{-2}$ away from zero. At times $t=2\pi k$, we have for $m\neq \pm k$,
\begin{align*}
(\ee^{-\ii t \Schr_\Gamma}\delta_0)(m)  &= \int_0^{1/2} \ee^{-2\pi\ii m \theta}\ee^{-2\pi \ii k(\frac{1}{2}-\theta)}\,\dd\theta + \int_{1/2}^1 \ee^{-2\pi\ii m\theta}\ee^{-2\pi\ii k(\theta-\frac{1}{2})}\,\dd\theta\\
&=(-1)^k\Big(\frac{(-1)^{k-m}-1}{2\pi\ii (k-m)} + \frac{1-(-1)^{k+m}}{-2\pi\ii (k+m)}\Big) = \begin{cases} \frac{2\ii (-1)^k k}{\pi(k^2-m^2)}& \text{if } k-m \text{ is odd,}\\ 0 & \text{if } k-m \text{ is even.}\end{cases}
\end{align*}
This with~\eqref{eq:modpecu} tells us that, as $t$ approaches $2\pi$, the mass of the wave decreases until it vanishes at all odd vertices, except $\pm 1$. The wave peaks at $0$ and $\pm 1$, with square modulus $\frac{4}{\pi^2}$ and $\frac14$, respectively, then falls off like $|m|^{-4}$.

More generally, as $t$ approaches $2 \pi k$, the mass will tend to zero at all vertices $m=k-2r$, $r\in\Z$, except for the vertices $\pm k$ themselves, where it will have mass $\frac{1}{4}$.
On the other vertices, its mass will decay polynomially away from $\pm k$.

Thus, as time goes on, the polynomially decaying ``tent'' around zero splits into two tents to the right and the left, each of them traveling linearly with time, the respective peaks being located near the vertices $\pm k$ at times $t=2\pi k$, and the mass falls off polynomially away from the peaks. Curiously, the local behaviour is not monotone, in the sense that at a given vertex, the mass of the wave will start decreasing with time until vanishing at some $t=2\pi k$, then it starts increasing again, but reaches a lower peak than before, then it decreases again and so on.
\end{proof}
\subsubsection{A graph dispersing faster than $\Z$}\label{sec:disfa}
In this section we prove the dispersive behaviour stated in Theorem~\ref{thm:intro}(1). Again in view of Proposition~\ref{prp:abcvsabc}, it is enough to consider the Schr\"odinger operator corresponding to $c(\theta) = (\theta-\frac{1}{2})^2$. For reference, recall that for the adjacency matrix of $\Z$, we have the sharp estimate $\|\ee^{\ii t \cA_{\Z}}\|_{L^1\to L^\infty} \lesssim t^{-1/3}$.
\begin{prp}[Faster dispersion]\label{prp:megadis}
Let $\Gamma$ be the $\Z$-periodic graph with one-element fundamental cell, defined by the Floquet function $h(\cdot)=(\theta-\frac{1}{2})^2$. 
Then, the unitary group $(\ee^{\ii t \Schr_\Gamma})_{t \in \R}$ disperses faster than the standard Euclidean lattice $\Z$. More explicitly,
\[
\|\ee^{\ii t \Schr_\Gamma}\|_{L^1\to L^\infty} \le 8t^{-1/2}
\]
for all $t>0$. The exponent $-\frac{1}{2}$ is sharp.
\end{prp}
\begin{proof}
Note that $(\ee^{\ii t\Schr_\Gamma}\delta_n)(m) = (\ee^{\ii t\Schr_\Gamma}\delta_0)(m-n)$ by~\eqref{eq:eitaker}.
We have
\begin{align*}
(\ee^{\ii t\Schr_\Gamma}\delta_0)(n) &= \int_0^1\ee^{\ii t (\theta-\frac{1}{2})^2}\ee^{-2\pi\ii n\theta}\,\dd\theta = \int_{-\frac{1}{2}}^{\frac{1}{2}}\ee^{\ii t\phi^2}\ee^{-2\pi\ii n(\phi+\frac{1}{2})}\,\dd\phi \\
&=(-1)^n \ee^{\frac{-\ii \pi^2n^2}{t}}\int_{-\frac{1}{2}}^{\frac{1}{2}}\ee^{\ii t(\phi-\frac{\pi n}{t})^2}\,\dd\phi = (-1)^n \ee^{\frac{-\ii \pi^2n^2}{t}}\int_{-\frac{1}{2}-\frac{\pi n}{t}}^{\frac{1}{2}-\frac{\pi n}{t}} \ee^{\ii t u^2}\,\dd u \,.
\end{align*}
By the van-der-Corput Lemma, see for example~\cite[p. 332]{Ste}, we have $|\int_a^b \ee^{\ii t u^2}\,\dd u|\le 8 t^{-1/2}$, independently of $a,b$.

It follows that $|(\ee^{\ii t\Schr_\Gamma}\delta_m)(n)|\le \frac{8}{\sqrt{t}}$. Hence, given any $\psi\in \ell^1(\Gamma)$,
\[
|(\ee^{\ii t\Schr_\Gamma}\psi)(n)| = \Big| \sum_m \ee^{\ii t\Schr_\Gamma}(n,m)\psi(m)\Big| \le \frac{8}{\sqrt{t}} \|\psi\|_1 \,.
\]
To see that the exponent is sharp, note that $(\ee^{\ii t\Schr_{\Gamma}}\delta_0)(0) = \int_{-\frac{1}{2}}^{\frac{1}{2}} \ee^{\ii tu^2}\,\dd u = 2\int_0^{\frac{1}{2}} \ee^{\ii t u^2}\,\dd u$, and
\[
\Big|\int_0^{\frac{1}{2}} \ee^{\ii tu^2}\,\dd u - \frac{\ee^{\pi\ii/4}}{2}\sqrt{\frac{\pi}{t}}\Big|\le \frac{2}{t}\ ,
\]
by~\cite[Ex. A.3]{AS22}, so $|(\ee^{\ii t\Schr_{\Gamma}}\delta_0)(0)|\asymp t^{-1/2}$.
\end{proof}

\subsubsection{(Fractional) powers of the Laplacian}\label{sec:frapodis}
Consider dimension $d=1$. In the next statement we show that we may slow down the sharp dispersion $\|\ee^{\ii t\cA_\Z} \psi\|_\infty \lesssim t^{-1/3}$ by taking \emph{integer powers} of the adjacency matrix $\cA_\Z$. 

\begin{lem}
For any even $p>2$, we have $|(\ee^{\ii t\cA_\Z^p}\delta_0)(0)|\sim \frac{1}{t^{1/p}}$.
\end{lem}
Note that this implies $\|\ee^{\ii t \cA_\Gamma} \delta_0\|_\infty \gtrsim t^{-1/p}$. The same bound holds for powers of the Laplacian $-\Delta$, however we formulated it for $\cA_\Z$ because powers of $-\Delta$ do not define a graph $\Gamma$, cf. \S\,\ref{sec:fractio}.
\begin{proof}
We have $\ee^{\ii t\cA_\Z^p}(0,0) = \int_0^1 \ee^{\ii t (2\cos 2\pi\theta)^p}\,\dd\theta = 2\int_0^{1/2} \ee^{\ii t (2\cos 2\pi\theta)^p}\,\dd\theta$, where we used that $p$ is even. Let $\phi(\theta)=(2\cos 2\pi\theta)^p$. By induction $\phi^{(k)}(\theta) = (2\cos 2\pi\theta)^{p-k}g_k(\theta)$ for some $g_k$, implying that $\phi^{(k)}(\frac{1}{4})=0$ for all $k<p$. On the other hand, $\phi(\theta) = (\ee^{2\pi \ii \theta}+\ee^{-2\pi\ii\theta})^p = \sum_{k=0}^p \binom{p}{k} \ee^{2\pi\ii (2k-p)\theta}$, so $\phi^{(p)}(\theta)=\sum_{k=0}^p \binom{p}{k}\ee^{2\pi\ii(2k-p)\theta}(2\pi\ii(2k-p))^p$ and thus $\phi^{(p)}(\frac{1}{4})=(2\pi\ii)^p\sum_{k=0}^p\binom{p}{k} \ii^{2k-p} (2k-p)^p = (-4\pi)^p\sum_{k=0}^p \binom{p}{k}(-1)^k(\frac{p}{2}-k)^p = (-4\pi)^pp!\neq 0$, where we used \cite{Ruiz} in the last equality.

The function $\varphi$ only has two other critical points, $\theta_0=0,\frac{1}{2}$, satisfying $\phi''(\theta_0)\neq 0$.

Let $\chi_i$ be bump functions of small support, $\chi_i(\theta)=1$ near $0,\frac{1}{4}$ and $\frac{1}{2}$, for $i=1,2,3$, respectively. Let $I_i = \int \ee^{\ii t\varphi(\theta)} \chi_i(\theta) \dd\theta$ for $i\le 3$ and $I_4 = \int \ee^{\ii t\varphi(\theta)}(1-\chi_1(\theta)-\chi_2(\theta)-\chi_3(\theta))\dd\theta$.

It follows from \cite[p. 334]{Ste} that $I_1\sim t^{-1/2}$, $I_2 \sim t^{-1/p}$ and $I_3\sim t^{-1/2}$. Finally, for $I_4$, we can simply integrate by parts to get a decay $\lesssim t^{-1}$. Gathering the estimates, we conclude that $\int_0^{1/2} \ee^{\ii t\phi(\theta)}\dd\theta = I_1+I_2+I_3+I_4 \sim t^{-1/p}$ as stated.
\end{proof}

Let us now turn to fractional powers for the Laplacian. In view of Proposition~\ref{prp:smalquart} one may expect the dispersion to increase as much as desired by decreasing $\alpha$. Quite surprisingly, however, this is not the case. 

\begin{lem}
We have  $|(\ee^{\ii t(-\Delta)^\alpha}\delta_0)(0)|\sim t^{-1/2}$ for any $\alpha< 1$.
\end{lem}
\begin{proof}
We have $\ee^{\ii t(-\Delta)^\alpha}(0,0) = \int_0^1 \ee^{\ii t (4\sin^2\pi\theta)^\alpha}\dd\theta = \frac{2}{\pi}\int_0^{\pi/2} \ee^{\ii t 4^\alpha \sin^{2\alpha}\theta}\dd\theta$. Let $\phi(\theta)=4^\alpha \sin^{2\alpha}\theta$. In a first step we analyze $\int_0^{\pi/4} \ee^{\ii t\phi(\theta)}\dd\theta$: If $\beta:=2\alpha<1$, then $\phi'(\theta)$ is undefined at $0$. Performing the change of variables $x=(\sin\theta)^\beta$, yields $\int_0^{2^{-\beta/2}} \frac{\ee^{\ii t2^\beta x}x^{\frac{1}{\beta}-1}}{\beta\sqrt{1-x^{2/\beta}}}\dd x$. Let $f(x) = \frac{x^{\gamma-1}}{\sqrt{1-x^{2\gamma}}}$. Then $f'(x) = \frac{(1-x^{2\gamma})(\gamma-1)x^{\gamma-2}+\gamma x^{3\gamma-2}}{(1-x^{2\gamma})^{3/2}}$. Note that $\int_0^a |f'(x)|\,\dd x = \int_0^a f'(x)\dd x = f(a)$ is integrable for $\gamma=\frac{1}{\beta}>1$ and $a<1$. So we may integrate by parts to get $|\int_0^a\ee^{\ii t2^\beta x} f(x) \dd x|=\big|\frac{\ee^{\ii t 2^\beta x} f(x)}{2^\beta \ii t}|_0^{a} - \frac{1}{2^\beta \ii t}\int_0^{a} \ee^{\ii t 2^\beta x} f'(x)\,\dd x\big| \lesssim t^{-1}$.

If $\beta=1$, for $\int_0^{\pi/4} \ee^{\ii t 2\sin\theta}\dd\theta$, the function $\phi(\theta)=2\sin\theta$ has no critical points, so we can immediately integrate by parts to see that $|\int_0^{\pi/4} \ee^{\ii t 2\sin\theta}\dd\theta| \lesssim t^{-1}$.

If $1< \beta\le 2$, then for $\int_0^{\pi/4} \ee^{\ii t 2^\beta \sin^\beta \theta}\dd\theta$, the phase function $\phi(\theta) = 2^\beta \sin^\beta \theta$ has a unique critical point at $0$, where $\phi(\theta)\sim 2^\beta \theta^\beta$. It follows from \cite{Olver} by taking $\lambda=1$, $\mu=\beta$, $m=0$, $n=1$, $\nu=1$, that $\int_0^{\pi/4} \ee^{\ii t\phi(\theta)} \dd \theta = \ee^{\pi\ii/2\beta}\Gamma(\frac{1}{\beta})\frac{1}{2\beta t^{1/\beta}} + \delta_{0,1}(t)+\varepsilon_{0,1}(t)$, with $|\varepsilon_{0,1}(t)|\lesssim t^{-1}$ by \cite[(6.7)]{Olver}, and $\delta_{0,1}(t)=\int_0^{\pi/4} \ee^{\ii t 2^\beta\sin^\beta\theta}(1-(\frac{\beta}{2\beta}(2^\beta\sin^\beta\theta)^{1/\beta})')=\int_0^{\pi/4}\ee^{\ii t 2^\beta\sin 2^\beta\theta}(1-\cos\theta)$. Integrating by parts gives $|\delta_{0,1}(t)|\lesssim t^{-1}$. 

Summarizing, we have $\int_0^{\pi/4} \ee^{\ii t\phi(\theta)} \dd \theta \lesssim t^{-1}$ if $\beta\le 1$ and $\sim t^{-1/\beta}$ if $1<\beta< 2$.

On the other hand, $\int_{\pi/4}^{\pi/2} \ee^{\ii t\phi(\theta)} \dd \theta = \frac{1}{2}\int_{\pi/4}^{3\pi/4}\ee^{\ii t\phi(\theta)}$. Since $\phi'(\theta)=\beta\sin^{\beta-1}\theta\cos\theta$, $\phi(\theta)$ has a unique critical point at $\frac{\pi}{2}$, and $\phi''(\theta) = \beta(\beta-1)\sin^{\beta-2}\theta\cos^2\theta-\beta\sin^\beta\theta$, so $\phi''(\frac{\pi}{2})=-\beta \neq 0$. As before, this implies by \cite[p. 334]{Ste} that $\int_{\pi/4}^{3\pi/4} \ee^{\ii t\phi(\theta)} \dd \theta  \sim t^{-1/2}$ (introduce $\chi$ supported near $\frac{\pi}{2}$ as the major contribution, while the one for $1-\chi$ will be $\lesssim t^{-1}$ through integration by parts).
To conclude, we obtain $\int_0^1 \ee^{\ii t 4^\alpha \sin^{2\alpha}\theta}\dd \theta \sim t^{-1/2}$ for any $\alpha<1$.
\end{proof}

\section{Combes-Thomas estimates}\label{sec:ct}

Green's function estimates are invaluable tools in spectral theory: they show up both, in the
analysis of localization and delocalization regimes. More concretely, in this section we are going to discuss Combes-Thomas estimates. Those are generally understood to induce an exponential decay of the resolvent kernel away from the diagonal, for energies off the spectrum.

In the following we write for $z \notin \sigma(\mathcal{H}_{\Gamma})$, 
\begin{equation*}
    G_{\Schr_\Gamma}^z(v,v^{\prime}):=(\mathcal{H}_{\Gamma}-z)^{-1}(v,v^{\prime})\ ,\quad v,v^{\prime} \in V\ .
\end{equation*}
In a first step, one then realizes that the Combes-Thomas estimate in~\cite[Thm. 10.5]{AZ} is sufficiently general to imply the following upper bound.

\begin{prp}\label{prop:CombesThomasI}
    Let $\Gamma$ be a $\Z^d$-periodic graph with finite fundamental cell, and suppose that the edge weights $w_{ij}(k)$ decay exponentially in $k$ for all $i,j \in \{1, \dots, \nu\}$. 
Then the corresponding periodic operator $\Schr_\Gamma$ satisfies the Combes-Thomas estimate, that is $|G_{\Schr_\Gamma}^z(v,v^{\prime})|\le C_z \ee^{-\alpha_z |v-v^{\prime}|}$ for some $C_z,\alpha_z>0$.

In case of polynomial decay, namely $w_{ij}(k) \le C \langle k\rangle^{-\alpha-\beta}$ with $\beta>d$, we have that $|G_{\Schr_\Gamma}^z(v,v^{\prime})|\le C_z (|v-v^{\prime}|_2+1)^{-c_{z,\alpha}}$ for some $c_{z,\alpha}>0$.
\end{prp}

It seems natural to ask if one can improve the polynomial decay of the Green's function to an exponential one in the second case of Proposition~\ref{prop:CombesThomasI}. As stated in the following result, this is in general not possible.

\begin{prp}[Failure of exponential decay]\label{prp:ct}
Consider the $\Z$-periodic graphs with one-element fundamental cell, defined by the Floquet functions~\eqref{eq:parfla},~\eqref{eq:pecu} and \eqref{eq:exa3}. Then, the corresponding resolvent kernels do not decay exponentially away from the diagonal. More precisely, for all three graphs, the associated Green's functions satisfy
\[
G_{\Schr_\Gamma}^z(0,n)\asymp n^{-2}
\]
as $n\to\infty$, for $z\notin \sigma(\mathcal{H}_{\Gamma})$. The implicit constant only depends on $z$, and is different for the different graphs.
\end{prp}
\begin{proof}
Recall that, when $\nu=1$, the map $U$ is simply the (inverse) Fourier transform. We obtain
\[
G_{\Schr_\Gamma}^z(n,m) = \langle \delta_n, U^{-1}U(\mathcal{H}_\Gamma -z)^{-1}U^{-1}U\delta_m\rangle = \int_0^1 \frac{\ee^{2\pi\ii (n-m)\theta}}{h(\cdot)-z}d\theta\,.
\]
We explicitly discuss the case of~\eqref{eq:pecu}. We have
\[
G_{\Schr_\Gamma}^z(0,n) = \int_0^{1/2} \frac{\ee^{-2\pi \ii n\theta}}{\frac{1}{2}-\theta-z}\dd\theta + \int_{1/2}^1 \frac{\ee^{-2\pi\ii n\theta}}{\theta-\frac{1}{2}-z}\dd\theta \,.
\]
We perform an integration by parts. If $f(\theta) = (\frac{1}{2}-\theta	-z)^{-1}$ and $g(\theta)=(\theta-\frac{1}{2}-z)^{-1}$, then because $f(0)=g(1)$ and $f(\frac{1}{2})=g(\frac{1}{2})$, the boundary terms cancel out and we get
\[
G^z(0,n)= \frac{1}{2\pi \ii n} \Big(\int_0^{1/2} \frac{\ee^{-2\pi\ii n\theta}}{(\frac{1}{2}-\theta-z)^2}\dd\theta - \int_{1/2}^1 \frac{\ee^{-2\pi \ii n\theta}}{(\theta-\frac{1}{2}-z)^2}\dd\theta\Big) \,.
\]
We perform a second integration by parts. The inverse squares again coincide at the respective boundary values, however, because of the minus sign, they now add up. We get
\[
G_{\Schr_\Gamma}^z(0,n) = \frac{1}{2\pi^2n^2}\Big(\frac{(-1)^n}{z^2}-\frac{1}{(\frac{1}{2}-z)^2} - \Big[\int_0^{1/2}\frac{\ee^{-2\pi\ii n\theta}}{(\frac{1}{2}-\theta-z)^3}\dd\theta + \int_{1/2}^1 \frac{\ee^{-2\pi\ii n\theta}}{(\theta-\frac{1}{2}-z)^3}\dd\theta\Big]\Big) \,.
\]
By the Riemann-Lebesgue lemma, both integrals on the right vanish as $n\to\infty$. Hence, $G_{\Schr_\Gamma}^z(0,n)\asymp n^{-2}$ as asserted.

The case of~\eqref{eq:parfla} is similar: we get three integrals. After a first integration by parts, the six boundary terms cancel out. After a second one however, they do not. Similarly, for~\eqref{eq:exa3}, we get a single integral whose boundary terms cancel out after a first, but not a second, integration by parts.
\end{proof}

In relation to Proposition~\ref{prp:ct}, we learned of the recent paper \cite[Thm 7]{DER} showing that the fractional Laplacian Green's function satisfies $G^z_{(-\Delta)^\alpha}(0,n) \asymp |n|^{-d-2\alpha}$ for $\alpha<1$.

\section{Conclusion}\label{sec:conc}

This paper is a first attempt to better understand the class of essentially local periodic operators. We have demonstrated unusual phenomena using the important link with trigonometric series and oscillatory integrals, and we believe much more can be done in that direction, though the corresponding models may be no longer solvable in general.

Throughout the paper, we have assumed that the edge weights are nonnegative, symmetric and summable; these assumptions are natural as we view these weights as being related to probabilities of hopping from one vertex to another. More precisely, up to normalization and translation, $p_{ij}(k) = w_{ij}(k)$ represents the probability of hopping from $v_i$ to $v_j+k_\fa$, and $p_{ii}(0) = Q_i$ is the probability of staying at $v_i$. Therefore, constructing crystals with unusual properties under these restrictions is particularly interesting. However, most results of general nature that we derived, such as the regularity of eigenvalues with edge decay, the ballistic speed of transport and so on, generalize immediately to the context of complex Hermitian edge weights $w_{ij}\in \ell^1(\Z^d)$. The notable exception is Theorem~\ref{thm:top}, which already fails in the locally-finite case. For example, the  $1d$ graph associated with the Floquet matrix
\[
H(\theta) = \begin{pmatrix} -\ii \ee^{2\pi\ii \theta} + \ii \ee^{-2\pi\ii \theta}& \ee^{2\pi\ii \theta}+\ee^{-2\pi\ii \theta}\\ \ee^{2\pi\ii\theta}+\ee^{-2\pi\ii\theta}& \ii \ee^{2\pi\ii\theta} - \ii \ee^{-2\pi\ii\theta}\end{pmatrix}
\]
has pure point spectrum $\sigma(\cH_\Gamma) = \{\pm 2\}$. In other words, the spectrum is entirely flat. Graphs of this shape are known as \emph{Creutz ladders} in the physics literature.

The example in Proposition~\ref{prp:smalquart} can be generalized as follows: if $h(\theta)$ defines a crystal with $\nu=1$, and if $\ee^{\ii t h(\theta)}\notin H^1(\T^d)$ for $t\neq 0$, then the speed of motion is super-ballistic.

While we have tried to study several aspects of the theory in this paper, the territory  remains very much unexplored. Below we collect some immediate questions that arise from our research. Here, ``crystals'' is synonymous with ``periodic graphs with finite $\nu=|V_0|$ satisfying Assumption~\ref{AssumptionEdgeWeights}''.

\begin{problem}
Is there a crystal with infinitely many distinct eigenvalues?
\end{problem}

We saw that all spectral types may occur in crystals. This motivates

\begin{problem}\label{pro:gen}
What is the ``generic spectral type'' of crystals ?
\end{problem}

Problem~\ref{pro:gen} is intentionally vague. One can restrict this question to crystals with $\nu=1$, which amounts to studying the subset $S\subset A(\T^d)$ of functions with nonnegative, symmetric Fourier coefficients, and ask about the generic spectrum that a function in $S$ induces. 

We can reduce this to yet a smaller subclass. We saw that the crystals in Theorem~\ref{thm:intro} all have a decay $\sim k^{-2}$, but have quite different spectral/dispersive behaviours. The zeroes of $w(k)$ are periodic in each example. This prompts the following questions:

\begin{problem}\label{pro:ran}
Consider the crystal over $\Z$ with $\nu=1$ defined by the weight function $w(k) = \frac{c_k(\omega)}{k^2}$, where $c_k(\omega)$, $k\in \N$ are i.i.d. random variables on $[0,1]$ (say uniform or Bernoulli) and $c_{-k}(\omega):= c_k(\omega)$. Are there almost-sure spectral types and dynamics for the corresponding $\Gamma$?
\end{problem}

Physically, this means we choose the neighbours of $0$ at random with a fixed decay rate, and then copy this profile to all integers $k$, so that each realization of $\omega$ gives a periodic graph. Variants of Problem~\ref{pro:ran} are 

\begin{problem}
For each $c=(c_k)\in \{0,1\}^\N$, associate a weight function $w(k) = \frac{c_k}{k^2}$ for $k\in \N$ and let $w(-k)=w(k)$. Is there a dense $G_\delta$-set of $c\in \{0,1\}^\N$ such that the corresponding $\Gamma(c)$ has a specific spectral type ? Can we also understand the dynamics ? 
\end{problem}

\begin{problem}\label{pro:per}
Consider the crystal over $\Z$ with $\nu=1$ defined by the weight function $w(k) = \frac{c_k}{k^2}$. Do periodic and almost periodic sequences of $c_k$ over $\{0,1,\dots,p\}$ have specific spectral/dynamical features ? 
\end{problem}

\begin{problem}
Study problems analogous to~\ref{pro:ran}--\ref{pro:per} more generally for the decay rates $\asymp k^{-p}$, $p\ge 1$.
\end{problem}

Our example of a crystal with purely singular continuous spectrum relied on a function which is not differentiable a.e. Besides Theorem~\ref{thm:purely_absolutely_continuous_spectrum}, the only way to rule out singular continuous spectrum is if all Floquet functions $h_{ij}$ are analytic, or equivalently, the $w_{ij}(k)$ decay exponentially. In this case, all $E_{j}$ are analytic a.e. as in the locally finite case, and the standard theory applies. Can smooth or Lipschitz functions induce crystals having singular continuous spectrum?

\begin{problem}
Study the existence of singular continuous spectrum for ``regular'' Floquet functions $h_{ij}$.
\end{problem}

We have exhibited crystals with ballistic and super-ballistic speeds of motion.

\begin{problem}\label{prb:sub}
Is there a crystal with sub-ballistic transport? If yes, can we construct one with diffusive speed $\|x\ee^{\ii t \cH_\Gamma}\psi\| \sim t^{-1/2}$? Or anomalous speed $\sim t^{-\alpha}$, $\alpha\notin \{0,\frac{1}{2},1\}$?
\end{problem}

Here, sub-ballistic transport is interpreted as the vanishing of $\lim_{t\to\infty}\frac{\|x\ee^{\ii t \cH_\Gamma}\psi\|}{t}$ for all $\psi$ of compact support. Consider the simpler case where $\nu=1$. Then we know from Corollary~\ref{cor:balnu1} that if $\Gamma$ is defined by a Lipschitz function, this is impossible since the limit is nonzero for $\psi=\delta_v$. A candidate would thus be a continuous, non-Lipschitz function $h(\cdot)$ on $[0,1]$ such that $h'(\cdot)=0$ a.e., with Fourier coefficients $\widehat{h}(k)\ge 0$, $\widehat{h}(-k)=\widehat{h}(k)$ and $\sum \widehat{h}(k)<\infty$. Does such a function exist? The first property is reminiscent of the Cantor function, but this doesn't satisfy the other requirements. One reason one may think such a function exists is
    that the subset $A_0(\Torus)\subset A(\T)$ of functions which are locally constant
    a.e. is actually dense in $A(\Torus)$, see~\cite{BM}. The question
    is whether there is a function in $A_0(\Torus)$ with coefficients
    $\widehat{h}(k)\ge 0$ and $\widehat{h}(-k)=\widehat{h}(k)$; see also~\cite{Lit}. We also refer to~\cite{DT} and references therein for background on sub-ballistic motion.

We saw in Proposition~\ref{prp:megadis} that we can speed up dispersion to $t^{-1/2}$. By varying this example, we can find $h(\theta)$ such that $|(\ee^{\ii t \cH_\Gamma}\delta_n)(m)|\sim t^{-1/\nu}$ with $\nu>1$ arbitrarily close to $1$. However, the error term we get is not uniform in $n,m$ and so we cannot deduce that $\|\ee^{\ii t \cH_\Gamma} \|_{L^1\to L^\infty} \lesssim t^{-1/\nu}$.
\begin{problem}\label{prb:dis}
What is the fastest dispersion speed for crystals in dimension $d=1$?
\end{problem}
More precisely, the question refers to finding a universal lower bound on $\|\ee^{\ii t \cH_\Gamma}\|_{L^1\to L^\infty}$ and showing that it is sharp. Note that fixing the dimension is important here because we can always speed up dispersion by moving to higher dimensions; for example, one has $\|\ee^{\ii t\cA_{\Z^d}}\|\lesssim t^{-d/3}$. Some ideas in \cite{BKS} might be pertinent to address Problems~\ref{prb:sub}-\ref{prb:dis}.

\appendix

\section{Proof of Proposition~\ref{prp:local}}
\label{app:prp_local}

\begin{proof}[Proof of Proposition~\ref{prp:local}]
  The kernel of $T$ is given by
  \[
    T(n,n')
    = (\Schr_\Gamma\one\delta_{n'})(n)-(\one\Schr_\Gamma\delta_{n'})(n)
    =\left[\one(n')-\one(n)\right]\Schr_\Gamma(n,n')\,.
  \]
  For $\one=\one_\N$, this can only be non-zero if ($n'\ge 0$ and
  $n<0$) or ($n'<0$ and $n\ge 0$).

  We have $\Schr(n,n') = w(|n-n'|)$, so
  \begin{align*}
    \sum_{n\ge 0} |T(n,-1)|^2 + \sum_{n\ge 0} |T(n,-2)|^2 + \ldots
    &= \sum_{n\ge 0} w(n+1)^2 + \sum_{n\ge 0} w(n+2)^2 + \ldots\\
    &= \sum_{n\ge 0} \sum_{m\ge 1} w(n+m)^2.
  \end{align*}
  Similarly,
    $\sum_{n<0} \sum_{n'\ge 0} |T(n,n')|^2
    = \sum_{n<0} \sum_{n'\ge 0} w(n'-n)^2
    = \sum_{k\ge 0} \sum_{m\ge 1} w(k+m)^2$.
  It follows that
  \[
  \|T\|_{HS}^2 = 2 \sum_{k\ge 0} \sum_{m\ge 1} w(k+m)^2.
  \]

  If we assume that $w(k)$ is decreasing then we get
  \begin{align*}
    w(n+m)^2
    \le w(n)w(m),
    \qquad\text{so}\qquad
    \sum_{n,m\ge 1} w(n+m)^2<\infty
  \end{align*}
  and $T$ is Hilbert-Schmidt indeed. This is also true if
  $\sum_{k\ge 1} \sup_{n\ge k} w(n)<\infty$ since
  \begin{align*}
    \|T\|_{HS}^2
    \le 2 \sum_{k\ge 1} \sup_{r\ge k} w(r)\sum_{m\ge 1}w(k+m)
    \le 2 \|w\|_1 \sum_{k\ge 1} \sup_{r\ge k} w(r).
  \end{align*}
  This proves~\eqref{local.b}.

  To check~\eqref{local.c}, consider on $\N$ the weight
  \[
    w(k)
    = \begin{cases}
      \frac{1}{\sqrt{k}}& \text{if } k=2^n \text{ for some } n\ge 0\\
      0 &\text{otherwise}
    \end{cases}
  \]
  and extend $w$ to $\Z$ by $w(-k)=w(k)$.  Then
  $\sum_k w(k) = 2\sum_{n\ge 0} \frac{1}{2^{n/2}} =
  \frac{2}{1-\frac{1}{\sqrt{2}}} = 2(2+\sqrt2)$ is
  finite, and $w$ is clearly non-negative and symmetric. However,
  \begin{align*}
    \|T\|_{HS}^2\ge  \sum_{n\ge 1} \sum_{m\ge n} w(m)^2
    = \sum_{n\ge 1} \sum_{2^k\ge n} \frac{1}{2^k}
    = \sum_{n\ge 1} \sum_{k\ge \log_2 n} \frac{1}{2^k}
      = 2\sum_{n\ge 1} \frac{1}{2^{\log_2n}}
      = 2\sum_{n\ge 1} \frac{1}{n}=\infty.
  \end{align*}
  So for this weight the commutator $T$ is not Hilbert-Schmidt.

  We finally prove~\eqref{local.a}. Define the sequence of operators
  $(T_N)$ by
  \[
    T_N(n,n')
    = \begin{cases}
      \left[\one(n')-\one(n)\right] \Schr_\Gamma(n,n') & \text{if } |n-n'|\le N,\\
      0 &\text{otherwise.}
    \end{cases}
  \]
  The restriction ($n'\ge 0$ and $n<0$) together with
  $|n-n'| = n'-n = n'+|n|\le N$ imply that $0\le n'\le N-1$ and
  $-N\le n\le -1$. Similarly, ($n'<0$ and $n\ge 0$) with
  $|n-n'| = n+|n'|\le N$ imply that $-N\le n'\le -1$ and
  $0\le n\le N-1$. Hence, $T_N$ has a finite rank. More explicitly, it follows that
  \begin{align*}
    T_N
    = \sum_{n=-N}^{-1} \langle w(|n|+\cdot)\one_{[0,N-|n|]},\cdot\rangle \delta_n
    - \sum_{n=0}^{N-1} \langle w(n-\cdot)\one_{[-N+n,-1]},\cdot\rangle \delta_n.
  \end{align*}
  On the other hand, for $-N\le n<0$,
  \[
    (T-T_N)f(n)
    = \sum_{n'=N-|n|+1}^\infty w(n'+|n|)f(n')
    = \sum_{m=N+1}^\infty w(m)f(m-|n|),
  \]
  therefore,
  \begin{equation}\label{eq:t-tn2}
    |(T-T_N)f(n)|^2
    \le \sum_{m=N+1}^\infty w(m) \sum_{m=N+1}^\infty w(m) |f(m-|n|)|^2.
  \end{equation}
  For $n<-N$, $T_Nf(n)=0$ so we get
  $\sum_{m=|n|}^\infty w(m) f(m-|n|)$, and so $|(T-T_N)f(n)|^2$ is
  bounded from above by the RHS of~\eqref{eq:t-tn2} since the latter
  contains more positive terms.

  Similarly, for $0\le n\le N-1$,
  \[
    (T-T_N)f(n)
    = -\sum_{n'<-N+n} w(n-n')f(n')
    =- \sum_{m=N+1}^\infty w(m)f(n-m),
  \]
  so $|(T-T_N)f(n)|^2\le \sum_{m=N+1}^\infty w(m)\sum_{m=N+1}^\infty w(m)|f(n-m)|^2$.

  Finally, for $n\ge N$,
  $(T-T_N)f(n)=Tf(n) = -\sum_{m=n+1}^\infty w(m)f(n-m)$, whose square
  is bounded by the same sum.

  Summarizing the estimates for all $n$ and using Tonelli's theorem,
  \[
    \|(T-T_N)f\|^2
    \le \sum_{m=N+1}^\infty w(m) \sum_{m=N+1}^\infty w(m)
    \Bigl(\sum_{n<0} |f(m-|n|)|^2 + \sum_{n\ge 0} |f(n-m)|^2\Bigr).
  \]
  As the last term in parentheses is bounded by $2\|f\|^2$ for any
  $m$, this implies
  \[
    \|T-T_N\| \le \sqrt{2}\sum_{m=N+1}^\infty w(m) \to 0
  \]
  as $N\to\infty$, since $\sum_m w(m)<\infty$. This shows that $T$ is
  the norm limit of finite rank operators, hence $T$ is compact.
\end{proof}

\section{Technical details}\label{app}
Here, we give the proof of Theorem~\ref{thm:top}, i.e., we show how to deduce that the spectral top of a Schr\"odinger operator is not flat, taking into account one knows that the spectral bottom of a Laplacian is not flat. This is a perturbation argument which we adapt from~\cite[\S~4.1]{KorSa2}.

We first note that, recalling $D_{\Gamma}(v_i) = \sum_{j=1}^\nu\sum_{k\in \Z^d} w(v_i,v_j+k_\fa)$, one obtains the relation, $f \in \ell^2(\funda)$, 
\begin{equation}\label{eq:dirthet}
\langle f,(D_{\Gamma}-A(\theta))f\rangle = \frac{1}{2}\sum_{v_i\in \funda}\sum_{j=1}^\nu \sum_{k\in I_{ij}}w(v_i,v_j+k_\fa) |f(v_i) - \ee^{2\pi\ii\theta\cdot k}f(v_j)|^2\, ,
\end{equation}
where $D_{\Gamma}:=\diag(D_{\Gamma}(v_1),\dots,D_{\Gamma}(v_{\nu}))$. The double sum $\sum_{j=1}^\nu\sum_{k\in I_{ij}}$ is the sum over all (possible) neighbours of $v_i$. To prove~\eqref{eq:dirthet}, expand the square modulus to obtain four terms, the first two being
\[
\sum_{v_i\in \funda}\sum_{j=1}^{\nu}\sum_{k\in I_{ij}}w(v_i,v_j+k_\fa)\big(|f(v_i)|^2 - \overline{f(v_i)}\ee^{2\pi\ii\theta\cdot k}f(v_j)\big)  = \langle f,(D_{\Gamma}-A(\theta))f\rangle\, .
\]
Finally, noting that $w(v_i,v_j+k_\fa) = w(v_j+k_\fa,v_i) = w(v_j,v_i-k_\fa)$ and $I_{ji} = -I_{ij}$, we have
\begin{multline*}
\sum_{v_i\in \funda}\sum_{j=1}^{\nu}\sum_{k\in I_{ij}}w(v_i,v_j+k_\fa)\big(|f(v_j)|^2 - \overline{f(v_j)}\ee^{-2\pi\ii\theta\cdot k}f(v_i)\big)\\
 = \sum_{v_j\in \funda} \sum_{i=1}^\nu \sum_{k\in I_{ji}} w(v_j,v_i+k_\fa)\big(|f(v_j)|^2-\overline{f(v_j)}\ee^{2\pi\ii\theta\cdot k}f(v_i)\big) = \langle f,(D_{\Gamma}-A(\theta))f\rangle\ ,
\end{multline*}
proving~\eqref{eq:dirthet}.

\begin{proof}[Proof of Theorem~\ref{thm:top}]
The matrix $A(0) = (t_{ij})_{i,j\le \nu}$, where $t_{ij} = \sum_{k\in I_{ij}} w_{ij}(k)$, is the adjacency operator of a connected graph over the vertices $\{v_1,\dots,v_\nu\}$ (it is actually the quotient graph $\Gamma/\Z_\fa^d$). Indeed, if on the contrary some $v_i$ was isolated, it would mean that for any $j\neq i$, we have $t_{ij}=0$, implying $w_{ij}(k)=0$ for all $j\neq i$ and all $k\in I_{ij}$, in turn implying $h_{ij}(\cdot) = \sum_{k\in I_{ij}} w_{ij}(k) \ee^{2\pi\ii\theta\cdot k}=0$ for all $j\neq i$. This would mean that in $\Gamma$, the vertex $v_i$ can only be connected to neighbours of the form $v_i+k_\fa$ but is disconnected from any other $v_j+k_\fa$, so that $\Gamma$ splits into at least two connected components, one consisting of $\{v_i+k_\fa\}_{k\in\Z^d}$, a contradiction. It follows that the matrix $A(0)$ is irreducible.

Since $H(0) = A(0)+Q$, then for $t$ large enough, we see that $H(0)+t\operatorname{Id}$ is also irreducible, so its top eigenvalue is simple and has a corresponding eigenvector $h$ with positive entries (by the theorem of Perron-Frobenius). It follows that the top eigenvalue $E_{\star}$ of $H(0)$ is simple, with eigenvector $h$ whose entries are positive. Now, for $(hf)(v):=h(v)f(v)$,
\begin{multline*}
\langle hf, (E_{\star}-H(\theta))hf\rangle  = \sum_{v\in \funda} |f(v)|^2h(v)(H(0)h)(v) - \sum_{v\in \funda} \overline{f(v)}h(v)(H(\theta) hf)(v)\\
= \sum_{v\in \funda} |f(v)|^2h(v)(A(0)h)(v) - \sum_{v\in \funda}\overline{f(v)}h(v)(A(\theta)hf)(v)\\
= \sum_{v_i\in \funda}\sum_{j=1}^\nu\sum_{k\in I_{ij}}w(v_i,v_j+k_\fa)h(v_i)h(v_j)\big(|f(v_i)|^2-\overline{f(v_i)}\ee^{2\pi \ii\theta\cdot  k}f(v_j)\big)\\
= \frac{1}{2}\sum_{v_i\in \funda}\sum_{j=1}^\nu\sum_{k\in I_{ij}} w(v_i,v_j+k_\fa)h(v_i)h(v_j)|f(v_i)-\ee^{2\pi\ii\theta\cdot k}f(v_j)|^2\,,
\end{multline*}
where the last equality is a calculation similar to~\eqref{eq:dirthet}.

Let $\lambda_1(\theta)\ge 0$ be the smallest eigenvalue of $D_{\Gamma}-A(\theta)$. By the min-max principle, for any $f\in \ell^2(\funda)$ with $\|f\|=1$, we obtain
\begin{equation*}
\begin{split}
\lambda_1(\theta)\le \langle& f,(D_{\Gamma}-A(\theta))f\rangle = \frac{1}{2}\sum_{v_i\in \funda}\sum_{j=1}^\nu\sum_{k\in I_{ij}} w(v_i,v_j+k_\fa)|f(v_i)-\ee^{2\pi\ii k\cdot \theta}f(v_j)|^2\\
&\le \frac{1}{2h_-^2}\sum_{v_i\in \funda}\sum_{j=1}^\nu\sum_{k\in I_{ij}} w(v_i,v_j+k_\fa)h(v_i)h(v_j)|f(v_i)-\ee^{2\pi\ii\theta\cdot k}f(v_j)|^2 \\
&= \frac{\langle hf,(E_{\star}-H(\theta))hf\rangle}{h_-^2}\, ,
\end{split}
\end{equation*}
where $h_-=\min_{v\in \funda} h(v)$. In particular, if $f:= g/h$ (defined componentwise), where $g$ is an eigenvector of $H(\theta)$ corresponding to the top eigenvalue $E_{\star}(\theta)$, and we choose $g$ such that $\|f\|=1$, then
\begin{equation}\label{eq:varq0q}
\lambda_1(\theta)\le \frac{\|g\|^2}{h_-^2}(E_{\star}-E_{\star}(\theta))\le \frac{h_+^2}{h_-^2}(E_{\star}-E_{\star}(\theta))\,,
\end{equation}
where $h_+ = \max_{v\in \funda}h(v)$.

Finally, recall that $E_\star(\theta)$ is the top eigenvalue of $H(\theta)$. Since $\lambda_1(\theta)\ge 0$, the inequality~\eqref{eq:varq0q} implies that $E_\star\ge E_\star(\theta)$ for all $\theta$, which in turn implies that $E_\star$ is the top of the spectrum. Moreover,~\eqref{eq:varq0q} also implies that $E_\star(\theta)$ cannot be flat near $E_\star$: in fact, if we had $E_\star(\theta)=E_\star$ for a set of $\theta$ of positive measure, this would imply $\lambda_1(\theta)=0$ on a set of positive measure. Consequently, Proposition~\ref{PartlyFlatGeneral} would imply that $0$ is an eigenvalue of $\cD_\Gamma-\Adja_\Gamma$, contradicting Corollary~\ref{cor:botla}.
\end{proof}

\bigskip

\subsection*{Acknowledgements} We are very happy to thank the CIRM (Centre International de Rencontres Math\'ematiques, France) for its hospitality during a research stay in March~2024, where this project was initiated. We also thank Constanza Rojas-Molina (CY Cergy Paris Université) for bringing the fractional Laplacian to our attention.

\providecommand{\bysame}{\leavevmode\hbox to3em{\hrulefill}\thinspace}
\providecommand{\MR}{\relax\ifhmode\unskip\space\fi MR }
\providecommand{\MRhref}[2]{%
  \href{http://www.ams.org/mathscinet-getitem?mr=#1}{#2}
}
\providecommand{\href}[2]{#2}

\makeatletter
\providecommand\@dotsep{5}
\makeatother


\begin{thebibliography}{10}

\bibitem{AZ}
M.~Aizenman and S.~Warzel, \emph{Random Operators. Disorder Effects on Quantum Spectra and Dynamics}, GSM 168, AMS 2015.

\bibitem{AGHH}
S.~Albeverio, F.~Gesztesy, R.~H\o egh-Krohn and H.~Holden, \emph{Solvable models in quantum mechanics}, Second Ed., With an Appendix by Pavel Exner, AMS 2005.

\bibitem{AL}
D.~Aldous and R.~Lyons, \emph{Processes on Unimodular Random Networks}, Electron. J. Probab. \textbf{12} (2007), 1454--1508.

\bibitem{AS}
D.~Aldous and J.M.~Steele, \emph{The objective method: probabilistic combinatorial optimization and local weak convergence}, Probability on discrete structures, Encyclopaedia Math. Sci. \textbf{110} (2004), 1--72.

\bibitem{AKML}
D.~Alekseevsky, A.~Kriegl, P.W.~Michor and M.~Losik, \emph{Choosing roots of polynomials smoothly}, Israel J. Math. \textbf{139} (2004), 183--188.

\bibitem{AS22}
K.Ammari and M.~Sabri, \emph{Dispersion for Schrödinger operators on regular trees} Anal. Math. Phys. \textbf{12}, article number: 56 (2022).

\bibitem{AS23}
K.~Ammari and M.~Sabri, \emph{Dispersion on certain Cartesian products of graphs}, In ``Control and Inverse Problems. The 2022 Spring Workshop in Monastir, Tunisia'', Trends. Math. (2023), 217--222.

\bibitem{AHP}
J.~M.~Anderson, E.~A.~Housworth and L.~D.~Pitt, \emph{The spectral theory of multiplication operators and recurrence properties for nondifferentiable functions in the Zygmund class $\Lambda_a^\ast$}, Mathematika \textbf{39} (1992), 136--151.

\bibitem{Ar}
V.I.~Arnold, \emph{Geometrical Methods in the theory of ordinary differential equations}, Second edition, Springer 1988.

\bibitem{BM}
A.~Bernard and G.~Muraz, \emph{Locally constant almost everywhere Fourier transform}, in Function Spaces: Proceedings of the Third Conference on Function Spaces, May 19-23, 1998, Southern Illinois University at Edwardsville, Contemporary Mathematics \textbf{232} (1999), 65--68.

\bibitem{BKS}
S.~Baker, O.~Khalil and T.~Sahlsten, \emph{Fourier decay from $L^2$-flattening}, arXiv:2407.16699.

\bibitem{BDMY}
A.~Black, D.~Damanik, T.~Malinovitch and G.~Young, \emph{Directional Ballistic Transport for Partially Periodic Schr\"odinger Operators}, 	arXiv:2311.08612.

\bibitem{BdMS23}
A.~Boutet de Monvel and M.~Sabri, \emph{Ballistic transport in periodic and random media}, In ``From Complex Analysis to Operator Theory: A Panorama, In Memory of Sergey Naboko'', Oper. Theory Adv. Appl. \textbf{291} (2023), 163--216.

\bibitem{BdMS24}
A.~Boutet de Monvel and M.~Sabri, \emph{Ergodic theorems for continuous-time quantum walks on crystal lattices and the torus}, Ann. Henri Poincaré (2024), online first.

\bibitem{Bre07}
J.~Breuer, \emph{Singular Continuous Spectrum for the Laplacian
on Certain Sparse Trees}, Commun. Math. Phys. \textbf{269} (2007), 851--857.

\bibitem{BreF09}
J.~Breuer and R.L.~Frank, \emph{Singular Spectrum for Radial Trees}, Rev. Math. Phys. \textbf{21} (2009), 929--945.

\bibitem{BT}
H.~Broer and F.~Takens, \emph{Dynamical Systems and Chaos}, Springer 2011.

\bibitem{Ca06}
K.~Cai, \emph{Dispersion for Schr\"odinger Operators with One-gap Periodic Potentials on $\R^1$}, Dynamics of PDE. \textbf{3} (2006), 71--92.

\bibitem{PhysicsReportPaper}
A.~Campaa, T.~Dauxois, and S.~Ruffo., \emph{Statistical mechanics and dynamics of solvable models with long-range interactions}, \textbf{480} (2009), Issues 3-6, 57--159.

\bibitem{CGGSVWW}
C. Cedzich, T. Geib, F. A. Gr\"unbaum, C. Stahl, L. Vel\'azquez, A. H. Werner and R. F. Werner, \emph{The Topological Classification of One-Dimensional Symmetric Quantum Walks}, Ann. Henri Poincar\'e \textbf{19} (2018), 325--383.

\bibitem{CGGSVWW2}
C. Cedzich, T. Geib, F. A. Gr\"unbaum, C. Stahl, L. Vel\'azquez, A. H. Werner and R. F. Werner, \emph{Quantum walks: Schur functions meet symmetry protected topological phases}, Comm. Math. Phys. \textbf{389} (2022), 31--74.

\bibitem{CGSVWW}
C. Cedzich, T. Geib, C. Stahl, L. Vel\'azquez, A. H. Werner and R. F. Werner, \emph{Complete homotopy invariants for translation invariant symmetric quantum walks on a chain}, Quantum \textbf{2}, 95 (2018).

\bibitem{CGWW}
C. Cedzich, T. Geib, A. H. Werner and R. F. Werner, \emph{Chiral Floquet systems and quantum walks at half period}, Ann. Henri Poincar\'e \textbf{22} (2021), 375--413.

\bibitem{Cha}
A.~S.~Chaves, \emph{A fractional diffusion equation to describe {L}\'evy flights}, Phys. Lett. A \textbf{239} (1998), 13--16.


\bibitem{CS}
J-H.~Chung and J.~Shapiro, \emph{Topological Classification of Insulators:
I. Non-interacting Spectrally-Gapped One-Dimensional Systems}, arXiv:2306.00268.

\bibitem{CRS+}
O.~Ciaurri, L.~Roncal, P.R.~Stinga, J.L~Torrea, J.L.~Varona, \emph{Nonlocal
discrete diffusion equations and the fractional discrete laplacian, regularity
and applications}, Adv. Math. \textbf{330} (2018), 688--738.

\bibitem{Cu06}
S.~Cuccagna, \emph{Stability of standing waves for NLS with perturbed Lam\'e potential}, J. Diff. Eq. \textbf{223}
(2006), 112--160.

\bibitem{DMY}
D.~Damanik, T.~Malinovitch and G.~Young, \emph{What is Ballistic Transport?}, J. Spectr. Theory (2024), online first.

\bibitem{DT}
D.~Damanik and S.~Tcheremchantsev, \emph{Upper bounds in quantum dynamics}, J. Amer. Math. Soc. \textbf{20} (2007), 799--827.

\bibitem{DER}
M.~Disertori, R.~Maturana Escobar and C.~Rojas-Molina, \emph{Decay of the Green’s function of the fractional Anderson model and connection to long-range SAW}, J Stat Phys \textbf{191} (2024), 33.

\bibitem{Dyson}
F.~J.~Dyson, \emph{Existence of a phase-transition in a one-dimensional Ising ferromagnet}, \textbf{12} (1969), 91–107.

\bibitem{EinWa}
M.~Einsiedler and T.~Ward, \emph{Functional Analysis, Spectral Theory and Applications}, GTM 276, Springer 2017.

\bibitem{EG}
L.C.~Evans and R.F.~Gariepy, \emph{Measure Theory and Fine Properties of Functions}, Revised Ed., CRC Press 2015.

\bibitem{FLM}
 J.~Fillman, W.~Liu, R.~Matos, \emph{Irreducibility of the Bloch variety for finite-range Schr\"odinger operators}, J. Funct. Anal. \textbf{283} (2022), 109670.

\bibitem{Fir96}
N.~E.~Firsova, \emph{On the time decay of a wave packet in a one-dimensional finite band periodic lattice}, J. Math. Phys. \textbf{37} (1996), 1171--1181.

\bibitem{FHSZ}
M.~Fraczyk, B.~Hayes, M.~Sudan and Y.~Zhao, \emph{Spectral non-concentration near the top for unimodular random graphs}, arXiv:2401.07165.

\bibitem{GallavottiMiracleSole}
G.~Gallavotti and S.~Miracle-Sole, \emph{Statistical Mechanics of Lattice Systems}, Commun.~Math.~Phys., \textbf{5} (1967), 317--323.

\bibitem{GallavottiBook}
G.~Gallavotti, \emph{Statistical Mechanics: A short treatise}, Springer-Verlag, Berlin, 1999.


\bibitem{GebMol}
M.~Gebert and C.~Rojas-Molina, \emph{Lifshitz tails for the fractional Anderson model}, J. Stat. Phys. \textbf{179} (2020), 341--353.

\bibitem{GS}
G.M.~Graf and J.~Shapiro, \emph{The Bulk-Edge Correspondence for Disordered Chiral
Chains}, Commun. Math. Phys. \textbf{363} (2018), 829--846.

\bibitem{grafakos:09}
L.~Grafakos, \emph{Classical {F}ourier analysis}, third ed., Graduate Texts in
  Mathematics, vol. 249, Springer, New York, 2014.
  
\bibitem{Han}
R.~Han, \emph{Shnol's theorem and the spectrum of long range operators}, Proc. Amer. Math. Soc. \textbf{147} (2019), 2887--2897.

\bibitem{Har}
G.~H.~Hardy, \emph{Weierstrass's non-differentiable function}, Trans. Amer. Math. Soc. \textbf{17} (1916), 301-325.

\bibitem{HR}
G.~H.~Hardy and W.~W.~Rogosinski, \emph{Fourier Series}, Dover Publications 1999.

\bibitem{HN}
Y.~Higuchi, Y.~Nomura, \emph{Spectral structure of the Laplacian on a covering graph}, European Journal of Combinatorics \textbf{30} (2009) 570--585.

\bibitem{JS}
W.~Jian and Y.~Sun, \emph{Dynamical localization for polynomial long-range hopping random operators on $\Z^d$}, Proc. Amer. Math. Soc.150(2022), no.12, 5369–5381.

\bibitem{JL}
S.~Jitomirskaya and W.~Liu, \emph{Upper bounds on transport exponents for long-range operators}, J. Math. Phys. \textbf{62} (2021), 073506.

\bibitem{Kah}
J-P.~Kahane, \emph{S\'eries de Fourier absolument convergentes}, Ergebnisse der Mathematik und ihrer Grenzgebiete, Band 50, Springer 1970.

\bibitem{KangKagome}
M.~Kang, S.~Fang, L.~Ye, et al. \emph{Topological flat bands in frustrated kagome lattice CoSn}. Nat. Commun. \textbf{11} (2020), 4004. 

\bibitem{Kato}
T.~Kato, \emph{Perturbation theory for linear operators}, Classics in Mathematics, Springer-Verlag, Berlin, 1995, Reprint of the 1980 edition.

\bibitem{KellerLW-book}
M.~Keller, D.~Lenz, R.~K.~Wojciechowski, \emph{Graphs and Discrete Dirichlet Spaces}, Grundlehren der mathematischen Wissenschaften, Springer Cham, 2021.

\bibitem{KernerTaeuferWintermayr}
J.~Kerner, M.~Täufer, and J.~Wintermayr, \emph{Robustness of Flat Bands on the Perturbed Kagome and the Perturbed Super-Kagome Lattice},  Ann. Henri Poincaré \textbf{25} (2024), 3831--3857.

\bibitem{KorSa}
E.~Korotyaev and N.~Saburova, \emph{Schrödinger operators on periodic discrete graphs}, J. Math. Anal. Appl. \textbf{420} (2014), 576--611.

\bibitem{KorSa2}
E.~Korotyaev and N.~Saburova, \emph{Spectral estimates for Schr\"odinger operators on periodic discrete graphs}, St. Petersburg Math. J. \textbf{30} (2019), 667--698.

\bibitem{KT16}
S.~Kowalczyk and M.~Turowska, \emph{On the property $N^{-1}$}, Abstr. Appl. Anal.(2016), Art. ID 1256906, 5 pp.

\bibitem{KMR}
A.~Kriegl, P.W.~Michor, A.~Rainer, \emph{Many parameter H\"older perturbation
of unbounded operators}, Math. Ann. \textbf{353} (2012), 519--522.

\bibitem{KuchmentFloquetTheory} 
P.~Kuchment, \emph{To the Floquet
    theory of periodic difference equations }, in : ''Geometrical and
  Algebraical Aspects in Several Complex Variables'', Cetraro (Italy),
  June 1989, 203-209, EditEl, 1991.

\bibitem{KuchmentPeriodicOperators} 
P.~Kuchment, \emph{An overview of
    periodic elliptic operators}, Bull. Amer. Math. Soc. (N.S.),
  \textbf{53} (2016), no. 3, 343-414.

\bibitem{KP}
P.~Kuchment and O.~Post, \emph{On the Spectra of Carbon Nano-Structures}, Commun. Math. Phys. \textbf{275} (2007), 805--826.

\bibitem{Lit}
J.E.~Littlewood, \emph{On the Fourier coefficients of functions of bounded variation}, Q. J. Math. \textbf{os-7} (1936), 219--226.

\bibitem{Liu}
W.~Liu, \emph{Irreducibility of the Fermi variety for discrete periodic Schr\"odinger operators and embedded eigenvalues}, Geom. Funct. Anal. \textbf{32} (2022), 1--30.

\bibitem{Liu2}
W.~Liu, \emph{Fermi isospectrality for discrete periodic Schr\"odinger operators}, Comm. Pure Appl. Math. \textbf{77} (2024), 1126--1146.

\bibitem{LPW}
W.~Liu, M.~Powell and X.~Wang, \emph{Quantum dynamical bounds for long-range operators with skew-shift potentials}, arXiv:2411.00176.

\bibitem{MS}
T.~McKenzie, M.~Sabri, \emph{Quantum ergodicity for periodic graphs}, Comm. Math. Phys. \textbf{403} (2023), 1477--1509.

\bibitem{YZ20}
Y. Mi and Z. Zhao, \emph{Dispersive estimate for two-periodic discrete one-dimensional Schr\"odinger operator},
J. Math. Anal. Appl., \textbf{485} (2020), no. 1, 123768.

\bibitem{YZ22}
Y. Mi and Z. Zhao, \emph{Dispersive estimates for periodic discrete one-dimensional Schr\"odinger operators},
Proc. Amer. Math. Soc., \textbf{150} (2022), 267--277.

\bibitem{NagGarg}
S.~Nag and A.~Garg, \emph{Many-body localization in the presence of long-range interactions and long-range hopping}, Phys. Rev. B, \textbf{99} (2019), 224203.

\bibitem{Olver}
F.W.J.~Olver, \emph{Error bounds for stationary phase approximations}. SIAM J. Math. Anal. \textbf{5} (1974), 19--29.

\bibitem{PKL+}
J.~L.~Padgett, E.~G.~Kostadinova, C.~D.~Liaw, K.~Busse, L.~S.~Matthews and T.~W.~Hyde, \emph{Anomalous diffusion in one-dimensional disordered systems: a discrete fractional Laplacian method}, J. Phys. A: Math. Theor. \textbf{53} (2020) 135205 (21pp).

\bibitem{Pi82}
L.D.~Pitt. \emph{An example of stability of singular spectrum under smooth perturbations}, Integral Eqns. and Op. Th. \textbf{5} (1982), 114--126.

\bibitem{Pon}
S.P.~Ponomarev, \emph{The $N^{-1}$-property of maps and Luzin's condition ($N$)},  Math Notes \textbf{58} (1995), 960--965.

\bibitem{RS}
C.~Radin and B.~Simon, \emph{Invariant Domains for the Time-Dependent Schr\"odinger Equation}, J. Diff. Eq. \textbf{29}, 289--296 (1978).


\bibitem{Rud}
W.~Rudin, \emph{Real and Complex Analysis}, Third Edition, McGraw-Hill 1986.

\bibitem{RuelleCMP}
D.~Ruelle, \emph{Statistical Mechanics
of a One-Dimensional Lattice Gas}, Comm.~Math.~Phys., \textbf{9} (1968), 267-278.

\bibitem{RuelleBook}
D.~Ruelle, \emph{Statistical Mechanics. Rigorous results}, Reprint of the 1989 edition, World Scientific Publishing Co., Inc., River Edge, NJ; Imperial College Press, London, 1999.


\bibitem{Ruiz}
S.~Ruiz, \emph{An Algebraic Identity Leading to Wilson's Theorem}, Math. Gaz. \textbf{80} (1996), 579--582.

\bibitem{SabriY-23}
M.~Sabri and P.~Youssef, \emph{Flat bands of periodic graphs}, J. Math. Phys. \textbf{64} (2023), no. 9.

\bibitem{SS}
M.~Shamis and S.~Sodin, \emph{Upper bounds on quantum dynamics in arbitrary dimension}, J. Funct. Anal. \textbf{285} (2023), 110034.

\bibitem{Shi}
Y.~Shi, \emph{A multi-scale analysis proof of the power-law localization for random operators on $\Z^d$}, J. Differential Equations \textbf{297} (2021), 201--225.

\bibitem{Si96}
B.~Simon, \emph{Operators with singular continuous spectrum, VI: Graph Laplacians and Laplace--Beltrami operators}, Proc. Amer. Math. Soc. \textbf{124} (1996), 1177--1182.

\bibitem{Ste}
E.M.~Stein, \emph{Harmonic Analysis: Real-Variable Methods, Orthogonality, and Oscillatory Integrals},
Princeton University Press, Princeton, NJ, 1993.

\bibitem{Taka}
M.~Takahashi, \emph{Thermodynamics of One-Dimensional Solvable Models}, CUP 1999.

\bibitem{TaeuferPeyerimhoff}
M.~T\"aufer and N.~Peyerimhoff, \emph{Eigenfunctions and the integrated density of states on Archimedean tilings}, J. Spectr. Theory \textbf{11} (2021), 461--488.

\bibitem{Unt}
J.~Unterberger, \emph{Stochastic calculus for fractional {B}rownian motion with
              {H}urst exponent {$H>\frac 14$}: a rough path method by
              analytic extension}, Ann. Probab. \textbf{37} (2009), 565--614.

\bibitem{Wil}
C.H.~Wilcox, \emph{Theory of Bloch Waves}, J. Analyse Math. \textbf{33} (1978), 146--167.

\bibitem{Woe}
W.~Woess, \emph{Random walks on infinite graphs and groups}, CUP 2000.

\bibitem{WraggGehring}
M.~J.~Wragg and G.~A.~Gehring, \emph{The Ising model with long-range ferromagnetic interactions}, J.~Phys.~A.: Math.~and~Gen., \textbf{23} (1990), 2157.

\bibitem{Zyg}
A.~Zygmund, \emph{Trigonometric Series}, Third Edition, Volumes I\&II combined, CUP 2002.

\end{thebibliography}
\end{document}